\documentclass{amsart}

\usepackage[pdfstartpage={1},
pdfstartview={FitH},
pdftitle={Traffic Distributions and Independence II: Universal Constructions for Traffic Spaces},
pdfsubject={2010 Mathematics Subject Classification: 15B52; 46L53; 46L54; 60B20},
pdfcreator={},pdfproducer={},pdfkeywords={Freeness; permutation invariance; random matrices; traffic probability},
pagebackref=false,colorlinks=true,linkcolor={red!87.5!black},
citecolor={blue!50!black},linktocpage=true]{hyperref}
\hypersetup{pdfauthor={Guillaume C\'{e}bron, Antoine Dahlqvist, and Camille Male}}

\usepackage{accents}
\usepackage{afterpage}
\usepackage{amssymb}
\usepackage{appendix}
\usepackage{array}
\usepackage{bm}
\usepackage{bbm}
\usepackage{centernot}
\usepackage{cite}
\usepackage{enumitem}
\usepackage{float}
\usepackage{graphicx}
\usepackage{indentfirst}
\usepackage{mathdots}
\usepackage{mathrsfs}
\usepackage{mathtools}
\usepackage{pgf}
\usepackage{scalerel}
\usepackage{stackengine}
\usepackage{subcaption}
\usepackage{tikz-cd}
\usetikzlibrary{matrix,arrows,backgrounds,positioning}
\usepackage[linecolor=white,backgroundcolor=white,bordercolor=white,textsize=tiny]{todonotes}
\usepackage{xcolor}
\allowdisplaybreaks[4]

\author[C\'{e}bron, Dahlqvist,and Male]
{Guillaume C\'{e}bron$^\sharp$, 
Antoine Dahlqvist$^\dagger$, and Camille Male$^\star$}

\thanks{$\sharp$ Researchsupported in part by the ERC advanced grant ``Non-commutative distributions in free probability", held by Roland Speicher.\\\indent 
$\dagger$ Research supported in part by RTG 1845 and  EPSRC grant EP/I03372X/1. \\\indent 
$\star$ Research supported in part by the Fondation Sciences Math\'ematiques de Paris and the University Paris Descartes UMR 8145. \\\indent
$\sharp, \dagger, \star$ Research supported in part by the INSMI-CNRS through a PEPS JCJC grant.}

\address{IMT; UMR 5219\\
	 Universit\'{e} de Toulouse; CNRS\\
	 UPS, F-31400 Toulouse, France}
       \email{\href{mailto:guillaume.cebron@math.univ-toulouse.fr}{guillaume.cebron@math.univ-toulouse.fr}}

       \address{Department of Mathematics\\
         University of Sussex\\
         Falmer Campus\\
         Brighton BN1 9QH, UK}
       \email{\href{mailto:a.dahlqvist@sussex.ac.uk}{a.dahlqvist@sussex.ac.uk}}

       \address{Universit\'{e} de Bordeaux\\
         Institut de Math\'{e}matiques de Bordeaux\\
         351 Cours de la Lib\'{e}ration\\
         33400 Talence\\
         France}
       \email{\href{mailto:camille.male@math.u-bordeaux.fr}{camille.male@math.u-bordeaux.fr}}     

\date{\today}
\title[Universal Constructions for Traffic Spaces]{Traffic Distributions and Independence II: Universal Constructions for Traffic Spaces}


\newtheorem{Th}{Theorem}[section]

\newtheorem{Def}[Th]{Definition}
\newtheorem{Prop}[Th]{Proposition}

\newtheorem{Lem}[Th]{Lemma}

\newtheorem{Cor}[Th]{Corollary}

\theoremstyle{remark}
\newtheorem{Ex}[Th]{Example}
\newtheorem{Rk}[Th]{Remark}

\renewcommand\leq\leqslant

\renewcommand\geq\geqslant

\def\Wg{\mathrm{Wg}}
\def\Id{\mathrm{Id}}
\def\Tr{\mathrm{Tr}}

\def\UN{\mathrm{U}(N)}
\def\esp{\mathbb E}
\def\N{\mathbb N}

\def\etc{,\ldots ,}
\def\one{\mathbbm{1}}

\def\limN{\underset{N \rightarrow \infty}\longrightarrow}
\def\Nlim{\underset{N \rightarrow \infty}\lim}

\def\C{\mathbb C}

\def\eps{\varepsilon}

\def\lara{\langle \mathbf x, \mathbf x^*\rangle}

\def\larac{\langle \mcal A\rangle}

\def\eq{\begin{eqnarray*}}
\def\qe{\end{eqnarray*}}
\def\eqa{\begin{eqnarray}}
\def\qea{\end{eqnarray}}

\newcommand{\circlearrowleftRot}{\mathbin{\text{\rotatebox[origin=c]{90}{$\circlearrowleft$}}}}
\newcommand{\circlearrowleftRotBis}{\mathbin{\text{\rotatebox[origin=c]{-90}{$\circlearrowright$}}}}

\def\mbf{\mathbf}
\def\mcal{\mathcal}

\def\mbb{\mathbb}
\def\trm{\textrm}
\def\mrm{\mathrm}

\def\Circlearrowleft{\ensuremath{%
  \rotatebox[origin=c]{180}{$\circlearrowleft$}}}

\def\bor{\begin{color}{orange}}
\def\eor{\end{color}}

\def\bte{\begin{color}{teal}}
\def\ete{\end{color}}
     
\begin{document}
\maketitle

\begin{center}
\begin{minipage}{14cm}
\begin{center}{\sc abstract:}\end{center}We investigate questions related to the notion of \emph{traffics} introduced by the third author as a non-commutative probability space with additional operations and equipped with the notion of \emph{traffic independence}. We prove that any sequence of unitarily invariant random matrices that converges in non-commutative distribution converges as well in traffic distribution whenever it fulfils some factorisation property. We provide an explicit description of the limit which allows to recover and extend some applications (a result by Mingo and Popa on the asymptotic freeness from the transposed ensembles, and of Accardi, Lenczewski and Salapata on the freeness of infinite transitive graphs). We also improve the theory of traffic spaces by considering a positivity axiom related to the notion of \emph{state} in non-commutative probability. We construct the free product of traffic spaces and prove that it preserves the positivity condition. This analysis leads to our main result stating that every non-commutative probability space endowed with a tracial state can be enlarged and equipped with a structure of traffic space. 
\end{minipage}
\end{center}
~\\~\\~
\tableofcontents

\section{Introduction}

\subsection{Presentation of the results}
\subsubsection{Motivations for traffics}

Thanks to the fundamental work of Voiculescu~\cite{Voiculescu1991}, it is now understood that free probability is a good framework for the study of large random matrices. Here are two important considerations which sum up the role of non-commutative probability in the description of the macroscopic behavior of large random matrices:
\begin{enumerate}
\item A large class of families of random matrices $\mbf A_N\in \mrm{M}_N(\mbb C)$ \emph{converge in non-commutative distribution} as $N$ tends to $\infty$ (in the sense that the normalized trace of any polynomial in the matrices converges).
\item If two independent families of random matrices $\mbf A_N$ and $\mbf B_N$ converge separately in non-commutative distribution and are invariant in law when conjugating by a \emph{unitary} matrix, then the joint non-commutative distribution of the family $\mbf A_N\cup\mbf B_N$ converges as well. The joint limit can be described from the separate limits thanks to the relation of \emph{free independence} introduced by Voiculescu.
\end{enumerate}

In~\cite{Male2011,Male122,MP14}, it was pointed out that there are cases where other important macroscopic convergences occur in the study of large random matrices and graphs. One example is the adjacency matrix of the so-called sparse Erd\"os-Re\'nyi graph: it is the symmetric real random matrix $X_N$ whose sub-diagonal entries are independent and distributed according to Bernoulli random variable with parameter $\frac p N$, where $p$ is fixed. Let $Y_N$ be a deterministic matrix bounded in operator norm. Then the possible limiting $^*$-distributions of $(X_N,Y_N)$ depend on more than the limiting $^*$-distribution of $Y_N$ \cite{Male122}.
  
  The notion of non-commutative probability is too restrictive and should be generalized to get more information about the limit in large dimension. This is precisely the motivation to introduce the concept of \emph{traffic space}, which comes together with its own notions of distribution and independence: a traffic space is a non-commutative probability space where one can consider not only the usual operations of algebras, but also more general $n$-ary operations called \emph{graph operations}.
We will introduce those concept in detail, but let us first describe the role of traffics enlightened  in~\cite{Male2011} for the description of large $N$ asymptotics of random matrices:\begin{enumerate}
\item A large class of families of random matrices $\mbf A_N\in \mrm{M}_N(\mbb C)$ \emph{converge in traffic distribution} as $N$ tends to $\infty$ (in the sense that the normalized trace of any \emph{graph operation} in the matrices converges).
\item If two independent families of random matrices $\mbf A_N$ and $\mbf B_N$ converge separately in traffic distribution, satisfy a factorization property and are invariant in law when conjugating by a \emph{permutation} matrix, then the joint traffic distribution of the family $\mbf A_N\cup\mbf B_N$ converges as well. Moreover, the joint limit can be described from the separate limits thanks to the relation of \emph{traffic independence} introduced in \cite{Male2011}.
\end{enumerate}

As a sequel of \cite{Male2011}, the purpose of this monograph is to develop the theory of traffics and provide more examples.

\subsubsection{Limiting traffic distribution of large unitarily invariant random matrices}

For concreteness, we first describe how we encode new operations on matrix spaces and state one example of matrices that are considered in this monograph. 

For all $K\geq 0$, a \emph{$K$-graph operation} is a connected graph $g$ with $K$ oriented and ordered edges, and two distinguished vertices (one input and one output, not necessarily distinct).
The set $\mathcal G$ of graph operations is the set of all $K$-graph operations for all $K\geq 0$.
A $K$-graph operation $g$ has to be thought as an operation that accepts $K$ objects and produces a new one. 
\par For example, it acts on the space $\mrm{M}_N(\mbb C)$ of $N$ by $N$ complex matrices as follows. For each $K$-graph operation $g\in \mathcal G$, we define a linear map $Z_g: \mrm{M}_N(\mbb C) \otimes \cdots \otimes \mrm{M}_N(\mbb C)\to \mrm{M}_N(\mbb C)$ in the following way. Denoting by
 \begin{itemize}
 	\item  $V$ the vertex set of $g$,
	\item $(v_1,w_1),\ldots, (v_K,w_K)$ the ordered edges of $g$,
	\item $in$ and $out$ the distinguished vertices of $g$,
	\item $E_{k,l}$ the matrix unit $(\delta_{ik}\delta_{jl})_{i,j=1}^N\in \mrm{M}_N(\mbb C)$,
\end{itemize}
we set, for all $A_1,\ldots,A_K\in \mrm{M}_N(\mbb C)$, 
	$$Z_g(A_1 \otimes \dots \otimes A_K)=\sum_{\phi:V\to \{1,\ldots,N\}}\left( \prod_{k=1}^K A_k \big( \phi(w_k),\phi(v_k) \big) \right)\cdot E_{\phi(out),\phi(in)}.$$

Those operations appear quite naturally in investigations of random matrices, see for instance \cite[Appendix A.4]{Bai2010} and \cite{Mingo2012}. Following~\cite{Mingo2012}, we can think of the linear map $\mbb C^N\to \mbb C^N$ associated to $Z_g(A^{(1)}\otimes \dots \otimes A^{(K)})$ as an algorithm, where we feed a vector into the input vertex and then operate it through the graph, each edge doing some calculation thanks to the corresponding matrix $A^{(i)}$, and each vertex acting like a logic gate, doing some compatibility checks. This description relies only on the so-called \emph{commutative special $\dagger$-Frobenius comonoid structure} of matrix spaces \cite{CPV2013}.

The linear maps $Z_g$ encode naturally the product of matrices, but also other natural operations, like the Hadamard (entry-wise) product $(A,B)\mapsto A\circ B$, the real transpose $A\mapsto A^t$ or the degree matrix $deg(A) = diag(\sum_{j=1}^NA_{i,j})_{i=1\etc N}$.

Starting from a family $\mbf A=(A_j)_{j\in J}$ of random matrices of size $N\times N$, the smallest algebra close by adjunction and by the action of the $K$-graph operations is the \emph{traffic space} generated by $\mbf A_N$. The \emph{traffic distribution} of $\mbf A_N$ is the data of the non-commutative distribution of the matrices which are in the traffic space generated by $\mbf A_N$. More concretely, it is the collection of the quantities
$$\frac{1}{N}\esp\Big[\Tr \big(Z_g(A_{j_1}^{\epsilon_1}\otimes \dots \otimes A_{j_K}^{\epsilon_K})\big)\Big]$$
for all $K$-graph operations $g\in\mathcal{G}$, all $j_1,\ldots, j_K\in J$ and all $\epsilon_1,\cdots, \epsilon_K\in \{1,\ast\}$. 

In this monograph, we prove the following theorem. It shows that for a general class of  unitarily invariant matrices, the convergence of the $*$-distribution is sufficient to deduce the convergence in traffic distribution.

\begin{Th}\label{Th:Matrices}For all $N\geq 1$, let $\mbf A_N=(A_j)_{j\in J}$ be a family of random matrices in $\mrm{M}_N(\mbb C)$. We assume
\begin{enumerate}
\item The unitary invariance: for all $N\geq 1$ and all $U\in \mrm{M}_N(\mbb C)$ which is unitary, $U\mbf A_NU^*:=(UA_jU^*)_{j\in J}$ and $\mbf A_N$ have the same law.
\item The convergence in $*$-distribution of $\mbf A_N$: for all  indices $j_1,\ldots, j_K\in J$ and labels $\epsilon_1,\cdots, \epsilon_K\in \{1,\ast\}$, the quantity $(1/N)\esp[\Tr (A_{j_1}^{\epsilon_1}\cdots A_{j_K}^{\epsilon_K})]$ converges.
\item The factorization property: for all $*$-monomials $m_1,\ldots,m_k$, we have the following convergence
\begin{multline*}\lim_{N\to \infty} \mathbb{E}\left[\frac{1}{N}\mathrm{Tr}\left(m_1(\mbf A_N)\right)\cdots \frac{1}{N}\mathrm{Tr}\left(m_k(\mbf A_N)\right)\right]\\=\lim_{N\to \infty} \mathbb{E}\left[\frac{1}{N}\mathrm{Tr}\left(m_1(\mbf A_N)\right)\right]\cdots \lim_{N\to \infty} \mathbb{E}\left[\frac{1}{N}\mathrm{Tr}\left(m_k(\mbf A_N)\right)\right].\end{multline*}
\end{enumerate}
Then, $\mbf A_N$ converges in traffic distribution: for all $K$-graph operation $g\in\mathcal{G}$, indices $j_1,\ldots, j_K\in J$ and labels $\epsilon_1,\cdots, \epsilon_K\in \{1,\ast\}$, the following quantity converges
$$\frac{1}{N}\esp\Big[\Tr \big(Z_g(A_{j_1}^{\epsilon_1} \otimes \ldots \otimes A_{j_K}^{\epsilon_K})\big)\Big].$$
The limit of the traffic distribution of $\mbf A_N$ is unitarily invariant and depends explicitly on the limit of the non-commutative $*$-distribution of $\mbf A_N$.
\end{Th}

Note that the convergence is about macroscopic quantities build from the matrices. However, it contains more information than the convergence in $^*$-moments.

A recent result of Mingo and Popa \cite{MingoPopa2014} tells that for all sequence of unitarily invariant random matrices $\mbf A_N$, then the family $\mbf A_N^t$ of the transposes is asymptotically freeness with $\mbf A_N$ (under assumptions stronger than those of Theorem~\ref{Th:Matrices} that also imply the asymptotic free independence of second order). Thanks to the description of the limiting traffic distribution of unitarily invariant matrices, we get that for a family $\mbf A_N=(A_j)_{j\in J}$  as in Theorem \ref{Th:Matrices}, $\mbf A_N$, $\mbf A_N^t$ and $deg(\mbf A_N)$ are asymptotically free independent.

A result similar to Theorem \ref{Th:Matrices}, about the convergence of the permutation invariant observables on random matrices, is also proved independently by Gabriel in \cite{Gabriel2015}. More generally, up to some conventions the framework developed in \cite{Gabriel2015a,Gabriel2015,Gabriel2015c} is equivalent to the framework of traffics. Interestingly, it develops aspects that are not yet considered for traffics, such as the central notion of cumulants.

\subsubsection{Non-commutative probability spaces and traffic spaces}

We now introduce the abstract notion of traffic spaces. The purpose is to define a structure for the limit of large matrices that captures the limiting traffic distribution, in a similar way the model of non-commutative random variables captures the limiting joint distribution of large matrices in the theory of free probability.

We first recall the setting of non-commutative probability. A non-commutative probability space is a pair $(\mcal A, \Phi)$, where $\mcal A$ is an algebra and $\Phi$ is linear form. One often assumes that $\mcal A$ is unital and $\Phi(1_\mcal A)=1$, and that $\Phi$ is a trace, that is $\Phi(ab)=\Phi(ba)$ for any $a,b \in \mcal A$.  A $^*$-probability space is a unital non-commutative probability space equipped with an anti-linear involution $\cdot ^*$ satisfying $(ab)^*=b^*a^*$ and such that $\Phi$ is positive, that is $\Phi(a^*a)\geq 0$ for any $a\in \mcal A$. 
The distribution of a family $\mbf a$ of elements of a non-commutative probability space is the linear form $\Phi_{\mbf a}:P \mapsto \Phi\big(P(\mbf a)\big)$ defined for non-commutative polynomials in elements of $\mbf a$. On $^*$-probability spaces, the $^*$-distribution is defined by the same formula for non-commutative polynomials in the elements and their adjoints. The convergence in ($^*$-)distribution of a sequence $\mbf a_N$ is the pointwise convergence of $\Phi_{\mbf a_N}$.

\par An algebraic traffic space is equivalent to the data of a non-commutative probability space $(\mcal A, \Phi)$ and of a collection of $K$-linear maps from $\mcal A^K$ to $\mcal A$ indexed by the $K$-graph operations satisfying mild assumptions. More precisely, to each $K$-graph operation $g\in \mcal G$ there is a linear map $$Z_g: \underbrace{\mathcal{A}\otimes \cdots \otimes \mathcal{A}}_{K\ \text{times}}\to \mathcal{A}$$ subject to some requirements of compatibility. Namely, it should be a so-called operad algebra over the set of graph operations (Definition~\ref{Def:Gaction}). The traffic distribution of a family $\mbf a = (a_j)_{j\in J} \in \mcal A^J$ is equivalent to the collection of the quantities $\Phi\big[ Z_g(a_{\gamma(1)} \otimes \dots \otimes a_{\gamma(K)})\big]$ for any $K$ graph operation $g$ and for any map $\gamma: \{1\etc K\}\to J$. Actually, the definition of the traffic spaces will be given as pairs $(\mcal A, \tau)$, where $\tau$ is a combinatorial function that is equivalent to the data of $\Phi$, although it is more intrinsic.

Finally, a \emph{traffic} (an element of $\mcal A$) is a non-commutative random variable, albeit coming with more information, as the action of graph operations permits to consider additional operations: the Hadamard product, the transpose, the degree, etc. As an example, let us highlight that if a matrix $A_N$ converges in traffic distribution to $a\in \mcal A$, the joint non-commutative distribution of $A_N, A_N \circ A_N, A_N^t, deg(A_N),\dots $ converges to the distribution of $a, a\circ a , a^t, deg(a), \dots$ in $(\mcal A, \Phi)$.

\subsubsection{Independence and positivity}

Voiculescu's definition of freeness is the vanishing of the trace in alternating product of centered elements. Contrary, 
the original definition of traffic independence is based on a notion of transform specific to traffic spaces, and hence is formally far from the latter one. In Theorem \ref{Equivalence Free product free independence} \emph{we present a natural characterization of traffic independence} which is the analogue of Voiculescu's fundamental definition. Roughly speaking, traffic independence is the vanishing of the trace in alternating \emph{operations} of \emph{reduced} elements. The \emph{operations} are no longer products by more complex patterns from a larger operad (the bi-graph operations, also called wiring diagrams \cite{RuSp13,Sp13,Sp15,VaSpEu15}). As well, the notion of \emph{reduceness} must be defined in terms specific to traffic theory. 

We deduce from it a simple criterion to characterize the free independence of variables assuming their traffic independence. An example of application is a new proof of the free independence of the spectral distributions of the free product of infinite deterministic graphs \cite{Accardi2007}. Another development of this is the connection with freeness over the diagonal, presented in \cite{ACDGM}.

Moreover in non-commutative probability theory, the three products of noncommutative probability spaces, that are in relation with the notions of tensor, free and Boolean independence,  preserve the positivity of the linear form. A second contribution of the present paper is the definition of the free product of traffic spaces which yields to the appropriate notion of independence for traffics defined in~\cite{Male2011}. More precisely, in Section \ref{Sec:FreeProdAlg}, for any collection $\mcal A_j, j\in J$ of algebraic traffic spaces (with traces $\Phi_j$), we define their free product $\ast_{j\in J}\mcal A_j$, in such a way that the algebras $A_j$ seen as traffic subspaces of $\ast_{j\in J}\mcal A_j$ are traffic independent with respect to the canonical trace.

It has to be noted that the positivity of the traces $\Phi_j$ on the spaces $\mcal A_j$ is not sufficient to ensure the positivity of the resulting trace on $\ast_{j\in J}\mcal A_j$. One has to require more positivity conditions on $\Phi_j$ to get positivity at the end. This is one motivation to define the good notion of positivity for traffic spaces.
In Definition~\ref{Def:Traffic} of Section \ref{Sec:Recall}, we define a \emph{traffic space} as an algebraic traffic space $\mcal A$ with trace $\Phi$ with two additional properties: the compatibility of the involution $\cdot ^*$ with graph operations, and a positivity condition on $\Phi$ which is stronger than assuming that $\Phi$ is a state. The main point is to prove the compatibly between traffic independence and the notion of positivity, stated in the following theorem.

\begin{Th}\label{Th:PosFreeProd} The free product of traffic spaces preserves the positivity of traffic spaces, so that the free product of traffic spaces is well-defined as a traffic space. 
\end{Th}

In particular, for any traffic $a$, there exists a traffic space that contains a sequence of traffic independent variables distributed as $a$. Moreover, a traffic space can always be enlarged in order to introduce traffic independent random variables.

\subsubsection{Three canonical models of traffics}

We turn now to our last result, which was the first motivation of this monograph and whose demonstration uses both Theorem~\ref{Th:Matrices} and Theorem~\ref{Th:PosFreeProd}. It states that there exist three different ways of enlarging a $^*$-probability space into a traffic space, each one related to respectively the  tensor, the free and the Boolean independence. Let us be more explicit, starting with the model related to freeness. As explained, Theorem~\ref{Th:Matrices} in its full form gives a formula for the limiting traffic distribution of large unitary invariant random matrices which involves only the limiting non-commutative distribution. Replacing in this formula the limiting non-commutative distribution of matrices by an arbitrary distribution, we obtain a traffic distribution which is related to free independence as the following result highlights.

\begin{Th}\label{MainThfree} Let $(\mcal A,\Phi)$ be a tracial $^*$-probability space. There exists a traffic space $\mcal {B}$ such that :
\begin{enumerate} \item $\mcal A\subset  \mcal {B}$ as $*$-algebras and the trace induced by $\mcal B$ on $\mcal A$ is $\Phi$;
\item two families $\mbf a$ and $\mbf b\in \mcal A\subset  \mcal {B}$ are \emph{freely independent} in $\mcal A$ if and only if they are traffic independent in $\mcal {B}$. 
\end{enumerate}
\end{Th}
The formula for the traffic distribution given, the difficulty consists in proving that this distribution satisfies the positivity condition.

Remark that, as described in~\cite{Male2011} and recalled in Section~\ref{Sec:freetype}, an Abelian non-commutative probability space can be endowed with a structure of traffic space.
\begin{Th}\label{MainThtensor} Let $(\mcal A,\Phi)$ be a Abelian $^*$-probability space. There exists a traffic space $\mcal {B}$ such that :
\begin{enumerate} \item $\mcal A\subset  \mcal {B}$ as $*$-algebras and the trace induced by $\mcal B$ on $\mcal A$ is $\Phi$;
\item two families $\mbf a$ and $\mbf b\in \mcal A\subset  \mcal {B}$ are \emph{tensor independent} in $\mcal A$ if and only if they are traffic independent in $\mcal {B}$. 
\end{enumerate}
\end{Th}

Finally, thanks to Section \ref{Sec:Booleantype}, one can produce an analogue construction for Boolean independence. We recall that any traffic space is endowed with two linear forms: a trace and a second linear form called the \emph{anti-trace}.

\begin{Th}\label{MainThBoolean} Let $(\mcal A,\Psi)$ be a $^*$-probability space. There exists a traffic space $\mcal {B}$ such that :
\begin{enumerate} \item $\mcal A\subset  \mcal {B}$ as $*$-algebras and the anti-trace induced by $\mcal B$ on $\mcal A$ is $\Psi$;
\item two families $\mbf a$ and $\mbf b\in \mcal A\subset  \mcal {B}$ are \emph{Boolean independent} in $\mcal A$ if and only if they are traffic independent in $\mcal {B}$. 
\end{enumerate}
\end{Th}

This construction comes together with a large model for asymptotically Boolean independent random matrices.

In other words, the free product of traffic space leads to the tensor product, Boolean product or the free product of the probability spaces, depending on the way the $*$-distribution and the traffic distribution of our random variables are linked. It corresponds to three different types of traffic that we will define in Section~\ref{Sec:three_types_traffics} : the traffics of free, tensor, or Boolean types. Interestingly, we also see that the last notions of monotone and anti-monotone independence (see \cite{Mu2002,Mu2003}) appear to describe the relations between traffics of different types when they are traffic independent. We sum up the non-commutative independences which follows from traffic independence in Figure~\ref{Fig:Triangle}.

  \begin{figure}[h]
    \begin{center}
     \includegraphics[width=80mm]{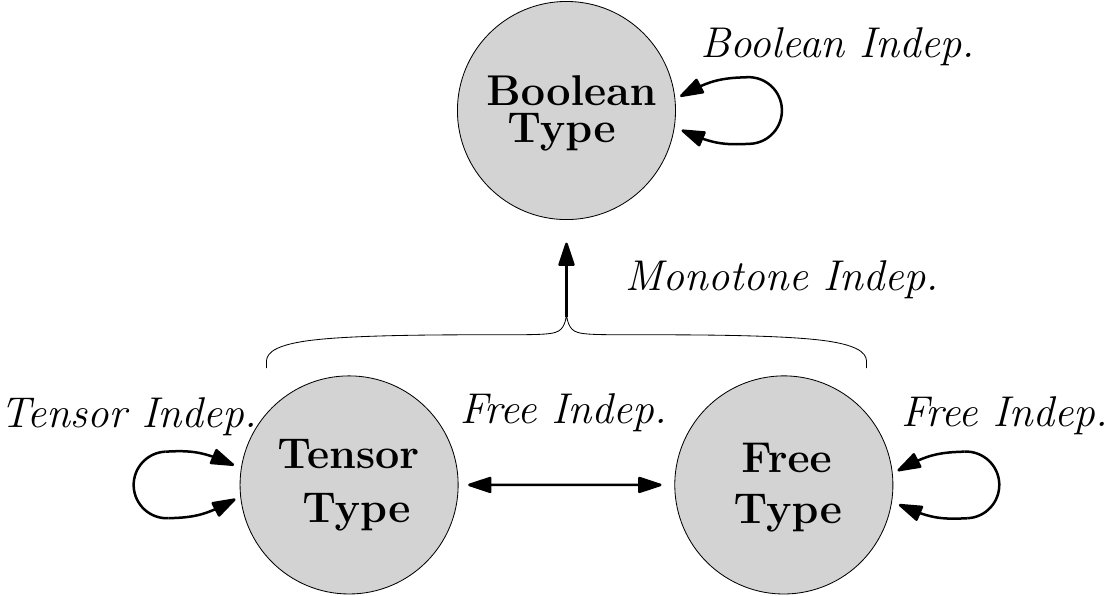}
    \end{center}
    \caption{The non-commutative independences of traffics of free, tensor, and Boolean types which are traffic independent}
    \label{Fig:Triangle}
  \end{figure}

{\bf Organization of the monograph:} In the rest of this introduction, we first recall the definitions of algebraic traffic spaces and traffic independence. 
Part I is dedicated to general facts on traffics. In Section \ref{Sec:NaturalCha} we introduce an equivalent definition of traffic independence. In Section \ref{Sec:PdtTrafSp} we define the free product of traffic spaces and prove Theorem \ref{Th:PosFreeProd}. 
Part II is devoted to particular types of traffics, starting with the so-called unitarily invariant traffics that are introduced and described in Section \ref{Sec:GeneralUnitTraffic} and \ref{Sec:UISCMC}. Theorem \ref{Th:Matrices} on unitarily invariant matrices is proved in Section \ref{Sec:AsymUImatrices}. 
In Section \ref{section:Canonical_construction}, we prove Theorem \ref{MainThfree} on the canonical extension of $^*$-probability spaces via traffics of free type. In Section \ref{Sec:three_types_traffics}, we investigate the canonical extensions of tensor and Boolean type, and prove Theorems \ref{MainThtensor} and \ref{MainThBoolean}.

\subsection{Definitions}\label{Sec:Recall}
   This section provides basic definitions from \cite[Chapter 4]{Male2011} in the theory of traffic spaces.
\subsubsection{Algebras over an operad}\label{Sec:AlgOp}
We first make more precise the definition of graph operations given in the introduction.

\begin{Def}For all $K\geq 0$, a \emph{$K$-graph operation} is a finite, connected and oriented graph with $K$ ordered edges, and two particular
vertices (one input and one output). The set of $K$-graph operations is denoted by $\mcal G_K$, and we define $\mcal G = \bigcup_{K\geq 0} \mcal G_K$.
\end{Def}
A $K$-graph operation can produce a new graph operation from $K$ different graph operations thanks to the following \emph{composition maps}
\begin{align*} 
	\mcal G_K\times \mcal G_{L_1}\times \cdots \times \mcal G_{L_K}&\to \mcal G_{L_1+\cdots +L_K}\\
(g,g_1,\ldots, g_K)&\mapsto g  (g_1,\ldots, g_K),
\end{align*}
for $K\geq 1$ and $L_i\geq 0$, $i=1\etc K$ which consist in replacing the $i$-th edge of $g\in \mcal G_K$ by the $L_i$-graph operation $g_i$ (leading at the end to a $(L_1+\cdots +L_K)$-graph operation). Let also consider the \emph{action of the symmetric group} 
\begin{align*} 
	S_K \times \mcal G_K&\to \mcal G_K\\
(\sigma, g)&\mapsto g^{(\sigma)},
\end{align*}
for $K\geq 2$ which consists in reordering the edges of $g$ according to $\sigma$: if $e_1,\ldots ,e_K$ are the ordered edges of $g$, $e_{\sigma^{-1}(1)},\ldots ,e_{\sigma^{-1}(K)}$ are the ordered edges in $g^{(\sigma)}$. Finally, let us denote by $id$ the graph operation which consists in two vertices and one edge from the input to the output. Endowed with those composition maps and the action of the symmetric groups, the set $\mcal G$ is a \emph{symmetric operad}, in the sense that it satisfies
\begin{enumerate}
\item the \emph{identity} property $g(id,\ldots, id)=g=id(g)$,
\item the \emph{associativity} property
\begin{align*}
& g \,  \big(\, g_1  \, (g_{1,1}, \ldots, g_{1,k_1}), \ldots, \, g_K   (g_{K,1}, \ldots,g_{K,k_K})\big) \\
 &= \big(\, g \,  (g_1, \ldots, g_K)\big)  \, (g_{1,1}, \ldots, g_{1,k_1}, \ldots, g_{K,1}, \ldots, g_{K,k_K}),
\end{align*}
\item the \emph{equivariance} properties
$
(g^{(\sigma)})\, (g_{\sigma^{-1}(1)},\ldots,g_{\sigma^{-1}(K)}) = g \, (g_1,\ldots,g_K); 
$ and $
g\, ( \, g_1^{(\sigma_1)},\ldots, \, g_K^{(\sigma_K)}) = \big(\, g \, (g_1,\ldots,g_K)\big)^{(\sigma_1\times\ldots\times\sigma_K)}.
$
\end{enumerate}
The element $id\in \mcal G_1$ is called the \emph{identity} of the operad.

Let us now define how a $K$-graph operation can produce a new element from $K$ elements of a vector space in a linear way.

\begin{Def}\label{Def:Gaction}An \emph{action} of an operad $\mcal G = \bigcup_{K\geq 0} \mcal G_K$ on a vector space $\mcal A$ is the data, for all $K\geq 0$ and $g\in \mcal G_K$, of a linear map
$Z_g:\mathcal{A}^{\otimes K}\to \mathcal{A}$ such that: $\forall g \in \mcal G_K, g_i \in \mcal G, a_i\in \mcal A, \sigma \in S_K$,
\begin{enumerate}
\item $Z_{id}$ is the identity on $\mcal A$, where $id\in \mcal G_1$ is the identity of the operad,
\item $Z_g  \, (Z_{g_1}\otimes \ldots \otimes Z_{g_K})=Z_{g \,  (g_1,\ldots, g_K)}$,
\item $Z_g(a_1\otimes \ldots\otimes a_K)=Z_{g_\sigma}(a_{\sigma^{-1}(1)}\otimes\ldots\otimes a_{\sigma^{-1}(K)})$.
\end{enumerate}
A vector space on which acts $\mcal G$ is called a \emph{$\mcal G$-algebra}. A $\mcal G$-subalgebra is a vector subspace of a $\mcal G$-algebra stable by the action of $\mcal G$. A $\mcal G$-morphism between two $\mcal G$-algebras $\mcal A$ and $\mcal B$ is a linear map $f:\mcal A \to \mcal B$ such that $f\big( Z_g(a_1\etc a_K)\big)=Z_g\big( f(a_1) \etc f(a_K)\big)$ for any $K$-graph operation $g$ and $a_1\etc a_K \in \mcal A$.
\end{Def}

In the following, $\mcal G$ always denotes the operad of graph operations. We now review some linear maps $Z_g$ of particular interest by describing the underlying graphs $g$. At each time, we shall represent $g$ graphically, forgetting the mention of the ordering of edges when it is not relevant, and assuming the input is the rightmost vertex of the graph and the output the leftmost one when they are not equal.
 
\begin{itemize}
	\item The only element of $\mcal G_0$ is the graph $(\cdot)$ with a single vertex and no edge. By convention, the map $Z_{(\cdot)}$ is a linear map $\mbb C \to \mcal A$. It is then characterized by the image of $1\in \mbb C$ that is denoted by $\mbb I:=Z_{(\cdot)}(1)$ and is called the unit of $\mcal A$. 
	\item By definition, $Z_{\cdot \leftarrow \cdot} = id_{\mcal A}$. The graph $(\cdot \rightarrow \cdot )\in \mcal G_1$, which consists in two vertices and one edge from the output to the input, induces another involution on $\mcal A$  which will be denoted by $a\mapsto a^t:=Z_{\cdot \rightarrow \cdot}(a)$. We call  $a^t$ the transpose of $a$.
	\item  The graph operation $(\cdot \overset{1}{\leftarrow} \cdot \overset{2}{\leftarrow} \cdot)$, which consists in three vertices and two successive edges from the input to the output, induces a bilinear map $(a,b)\in \mathcal{A}^2 \mapsto ab:= Z_{\cdot \overset{1}{\leftarrow} \cdot \overset{2}{\leftarrow} \cdot}(a\otimes b) \in \mathcal{A}$ which gives to $\mathcal{A}$ a structure of associative algebra over $\C$, with unit $\mbb I$. Hence, every $\mcal G$-algebra is in particular a unital algebra.	\item The Hadamard product is the bilinear map $(a,b) \in \mcal A^2 \mapsto a\circ b := Z_{\cdot \leftleftarrows \cdot}(a \otimes b)$, where the graph operation consists in two vertices and two edges from the input to the output. Its defines an associative and commutative product.
	\item The diagonal of an element $a \in \mcal A$ is defined by $\Delta(a):=Z_{\Circlearrowleft}(a)$, for the graph $\Circlearrowleft$ with one vertex and one edge (which is a self loop). The space $\Delta(\mcal A):= \big\{ \Delta(a), \, a\in \mcal A \big\}$ is a commutative $\mcal G$-subalgebra of $\mcal A$.
	\item The degree of an element $a \in \mcal A$ is defined by $deg(a):= Z_{\downarrow}(a)$, for the graph $\downarrow$ with two vertices, where one is both the input and the output, and an edge from the second vertex to the input/output. The map $\deg$ is a projection with image $\Delta(\mcal A)$.
\end{itemize}

\begin{Ex}\label{Ex:MgAlg}

Denote $\mrm{M}_N ( \mbb  C)$ the algebra of $N$ by $N$ complex matrices. 
For any $K\geq 1$ and $g\in \mcal G_K$ with vertex set $V$ and ordered edges $(v_1,w_1) \etc (v_K,w_K)$, let us define $Z_g$ by setting, for all $A_1,\ldots ,A_K\in \mrm{M}_N(\mbb C)$, the $(i,j)$-coefficient of $Z_g(A_1\otimes \ldots \otimes A_K)$ as
$$\left[Z_g(A_1\otimes \ldots \otimes A_K)\right] (i, j):= \sum_{
		\substack {
		\phi: V \to [N] \\ \phi(in) = j,\, \phi(out) = i}} \prod_{k=1}^K
		  A_k \big( \phi(w_k), \phi(v_k) \big) .$$
This defines an action of the operad $\mcal G$ on $\mrm{M}_N(\mbb C)$. The product $AB= Z_{\cdot \overset{1}{\leftarrow} \cdot \overset{2}{\leftarrow} \cdot}(A\otimes B)$ induced by this action coincides with the classical product of matrices. The Hadamard product $ A\circ B = Z_{\cdot \leftleftarrows \cdot}(A\otimes B)$ is the entry-wise product of matrices $\big(A(i,j) B(i,j) \big)_{i,j=1}^N$. The diagonal of a matrix $ \Delta(A):= Z_{\Circlearrowleft} (A)$ and the transpose $ A^t= Z_{\cdot \rightarrow \cdot}(A)$ are the diagonal $\big(\delta_{ij}A(i,i)\big)_{i,j=1}^N$ and the transpose $\big(A(j,i) \big)_{i,j=1}^N$ in the usual sense. The degree $deg(A): = Z_{\downarrow}(A)$ is the row sum diagonal matrix $  \big(\delta_{ij} \sum_{k}A(i,k) \big)_{i,j=1}^N$. For more information about the traffic distribution of matrices, see \cite[Section 1.2.]{Male2011}.
\end{Ex}

\begin{Ex}\label{Ex:MgAlgGraph} Let $\mcal V$ be an infinite set and let $\trm{M}_{\mcal V}(\mbb C)$ denotes the set of complex matrices $A=\big(A(v,w)\big)_{v,w\in \mcal V}$ indexed by $\mcal V$ (of possible infinite size) such that each row and column have a finite number of nonzero entries. For any $g\in \mcal G$ and $A_1,\ldots ,A_K\in \trm{M}_{\mcal V}(\mbb C)$, we define $Z_g(A_1\otimes \ldots \otimes A_K)$ by the same formula as in Example \ref{Ex:MgAlg} with summation now over the maps $\phi: V \to \mcal V$. This defines as well a structure of $\mcal G$-algebra for $\trm{M}_{\mcal V}(\mbb C)$. When the entries of the matrices are non negative integers, they encode the adjacency operator of a locally finite directed graph: the graph associated to a matrix $A$ has $A(v,w)$ edges from a vertex $v\in V$ to a vertex $w\in V$.
\end{Ex}

The graph operations can be equivalently encoded in terms of analogues of polynomials, turning the linearity on $\mcal A^{\otimes K}$ into $K$-linearity on $\mcal A$, $K\geq 2$. We also define now a notion with no input and output for the purpose of the next section, and later we will consider a generalization with arbitrary numbers of in/outputs.

\begin{Def}\label{Def:nGraphPol} Let $J$ be a labelling set.
\begin{itemize}
	\item A test graph labeled in $J$ is a collection $T = (V,E,\gamma)$, where $(V,E)$ is a finite, connected and oriented graph and $\gamma:E\to J$ is a labeling of the edges by indices. 
	\item A graph monomial labeled in $J$ is a collection $g = (V,E,\gamma,\mbf v)$, where $T=(V,E,\gamma)$ is a test graph and $\mbf v=(in,out)$ is an ordered pair of vertices of $T$, considered respectively as the input and the output of $T$. 
\end{itemize}
We denote by $\mcal T\langle J \rangle$ the set of test graphs labeled in $J$, and by $\mathcal{G}\langle J \rangle$ the set of graph monomials labeled in $J$. We denote by $\mbb C \mcal T\langle J \rangle$ and $\mbb C \mcal G\langle J \rangle$ the vector spaces generated by elements of the respective sets.
\end{Def}

The labelling map $\gamma$ of a graph monomial is not a bijection in general, so that a same variable can appear on several edges of the graph.

Let us consider a family $\mbf a = (a_j)_{j\in J} \in \mcal A^J$ of elements of a $\mcal G$-algebra, and consider a graph monomial $t=(V,E,\gamma)$ with labels in $J$. Let us list arbitrarily the edges $E=\{e_1,\ldots, e_K\}$ and denote by $g$ the $K$-graph operation $(V,E)$ with the ordered edges $e_1,\ldots, e_K$. We set $t(\mbf a)=Z_g\big(a_{\gamma(e_1)}\otimes \cdots \otimes a_{\gamma(e_K)}\big)$, which does not depend on the choice of the ordering of $e_1,\ldots, e_K$, thanks to the equivariance property. For more details about graph polynomials, see \cite[Section 4.2.2.]{Male2011}

\subsubsection{Algebraic traffic spaces}

\begin{Def}\label{Def:Traffic}An \emph{algebraic traffic space} is a couple $(\mcal A, \tau)$ where $\mcal A$ is a $\mcal G$-algebra and $\tau:\mbb C \mcal T\langle \mcal A \rangle\to \C$ is a linear functional, called the \emph{combinatorial trace}, defined on the space of test graphs labeled in $\mcal A$, satisfying
\begin{itemize}
\item the \emph{unity} property $\tau\big[  (\cdot ) \big]=1$ for $(\cdot)$ the graph with a single vertex and no edge,
\item the \emph{multi-linearity w.r.t. the edges} $\tau[T_{a+\lambda b}] = \tau[T_a] + \lambda \tau[T_b]$, for any test graph $T_{a+\lambda b} \in \mcal T\langle \mcal A \rangle$ having an edge $e_0$ with label $a+ \lambda b$, where $a,b\in \mcal A$ and $\lambda \in \mbb C$, and for $T_a$ and $T_b$ defined as $T$ with label $a$ and $b$ respectively for the edge $e_0$,
\item the \emph{substitution} property $\tau[T] = \tau[T_g]$ for any test graph $T \in \mcal T\langle \mcal A \rangle$ having an edge $e_0$ with label $g(\mbf a)$, where $g$ is a graph monomial and $\mbf a$ a family of elements of $\mcal A$, and $T_g$ obtained from $T$ by replacing the edge $e_0$ by the graph $g$ whose edges are labelled by the element of $\mbf a$.
\end{itemize}
An element of an algebraic traffic space is called a traffic. A homomorphism between two algebraic traffic spaces $(\mcal A,\tau)$ and $(\mcal A', \tau')$ is a $\mcal G$-morphism $f:\mcal A\to \mcal A'$ such that $\tau'\big[T( f(\mbf a))\big] = \tau\big[T(\mbf a)\big]$, for any $T \in \mcal T\langle J\rangle$ and $\mbf a =(a_j)_{j\in J}\in \mcal A^J$, where $f(\mbf a):= \big(f(a_j)\big)_{j\in J}$.
\end{Def}The map $\tau$ takes as entry a test graph whose edges are labeled by elements of $\mcal A$ and produces a complex number from. There is no meaning in the expression $\tau[a]$ for an element $a \in \mcal A$. 

In particular, $(\mcal A, \tau)$ is not an algebraic non-commutative probability spaces. It can always be endowed with two different structures of algebraic non-commutative probability spaces.

\begin{Def}\label{Def:TraceAndAntiTr} Let $(\mathcal{A},\tau)$  be an algebraic traffic space. The \emph{trace} $\Phi:\mathcal{A}\to \mathbb{C}$ and the anti-trace $\Psi:\mcal A\to \mathbb{C}$  are the linear maps given by the application of $\tau$ on a self loop and on a simple edge, namely
	$$\Phi: a \mapsto \tau\big[ \Circlearrowleft^a \big], \ \ \Psi: a \mapsto \tau\big[ \cdot \overset{a}\leftarrow \cdot \big].$$
\end{Def}

Recall that the product of two elements $a,b\in \mcal A $ is defined by $ab:=Z_{\cdot \overset{1}{\leftarrow} \cdot \overset{2}{\leftarrow} \cdot}(a \otimes b)$, and that endowed with this product $\mathcal{A}$ is an associative algebra. Then $(\mcal A, \Phi)$ and $(\mcal A, \Psi)$ are two algebraic non-commutative probability spaces. The map $\Phi$ is tracial in the sense that $\Phi(ab) = \Phi(ba)$ for any $a,b\in \mcal A$, and it satisfies $\Phi\big( \Delta( a ) \big) = \Phi(a)$ for any $a\in \mcal A$. Properties relating the different functionals $\tau$, $\Phi$ and $\Psi$ are explained in \cite[Section 4.2.4.]{Male2011}

In the following definition, for a test graph $T$ of $\mcal T \langle J \rangle$ and a family $\mbf a \in \mcal A^J$ of elements of a set $\mcal A$, we denote $T(\mbf a)\in\mcal T \langle \mcal A \rangle$ the test graph obtained by replacing labels $j\in J$ of the edges of $T$ by $a_j$. This definition is extended for $T \in \mbb C \mcal T \langle J\rangle$ by linearity. 

\begin{Def}
Let $(\mathcal{A},\tau)$ and $(\mathcal{A}_N,\tau_N), N\geq 1$, be algebraic traffic spaces, and $J$ be an index set.
	\begin{enumerate}
		\item  The \emph{traffic distribution} of a family $\mbf a=(a_j)_{j\in J}$ of elements in $\mathcal{A}$
 is the linear map $\tau_{\mbf a}: T\in \mbb C \mcal T\langle J  \rangle \mapsto \tau\big[ T(\mbf a)\big] \in  \mathbb{C} $. 
		\item A sequence of families $\mbf a_N\in \mcal A_N^J$ \emph{converges in traffic distribution} to $\mbf a$ if the traffic distribution of $\mbf a_N$ converges pointwise to the traffic distribution of $\mbf a$ on $\mbb C \mcal T\langle J \rangle$.
		\end{enumerate}\label{Def:DistrTraff}
\end{Def}

\begin{Ex}(Example \ref{Ex:MgAlg} continued)\label{Ex:MtSpa}
Let $(\Omega, \mcal F, \mbb P)$ be a probability space in the classical sense and let us consider the algebra $\mrm{M}_N\big(L^{\infty-}(\Omega, \mbb C )\big)$ of matrices whose coefficients are random variables with finite moments of all orders. Endowed with the action of the operad $\mathcal G$ described in Example \ref{Ex:MgAlg}, it is a $\mathcal G$-algebra, and it becomes an algebraic traffic space endowed with the combinatorial trace $\tau_N$ given by: for any test graph $T=(V,E,M)$ labeled in $\mrm{M}_N\big(L^{\infty-}(\Omega, \mbb C )\big)$, where $M:E\to \mrm{M}_N\big(L^{\infty-}(\Omega, \mbb C )\big)$,
	\eqa\label{eq:IntroDefTrace}
		 \tau_N \big[ T \big] =  \esp\Big[ \frac 1 N  \sum_{ \phi: V \to [N]}  \prod_{e=(v,w)\in E} \big(M(e) \big)  \big( \phi(w), \phi(v) \big) \Big].
	\qea
 The trace associated to $\tau_N$ is the usual normalized trace $\Phi_N: A \mapsto \esp\big[ \Tr A] \big]/N$ and the anti-trace is the map $\Psi_N:A\mapsto \esp\big[\sum_{i,j} A(i,j)/N\big]$. 
\end{Ex}

\begin{Ex}(Example \ref{Ex:MgAlgGraph} continued)\label{Ex:MtSpaGraph}
 Let $\mcal V$ be an infinite set. A locally finite rooted graph on $\mcal V$ is a pair $(G,\rho)$ where $G$ is a directed graph such that each vertex has a finite number of neighbors (or equivalently an element of the space $\trm{M}_{\mcal V}(\mbb C)$ of Example \ref{Ex:MgAlgGraph} with integers entries) and $\rho$ is an element of $\mcal V$. Recall briefly that the so-called weak local topology is induced by the sets of $(G,\rho)$ such that the subgraph induced by vertices at fixed distance of the root is given, see for instance \cite{BeCu12}. The notion of locally finite random rooted graphs refers to the Borel $\sigma$-algebra given by this topology. Let $(\Omega, \mcal F, \mbb P)$ be a probability space, let $\mcal V$ be a set and let $\rho\in \mcal V$. Let $\mbf G$ be a family of locally finite random rooted graphs on $\Omega$ with vertex set $\mcal V$ and common root $\rho$. Consider the $\mcal G$-subalgebra $\mcal A$ of $\trm{M}_{\mcal V}(\mbb C)$ induced by the adjacency matrices of $\mbf G$. For any test graph $T=(V,E,M)$ labeled in $\mcal A$ and any root $r\in V$ of $T$, denote
	\eqa
		 \tau_\rho \big[ (T,r) \big] =  \esp\Big[   \sum_{ \substack{ \phi: V \to \mcal V\\ \phi(r) = \rho}}  \prod_{e=(v,w)\in E}  \big(M(e) \big)\big(\phi(w), \phi(v)\big) \Big].
	\qea
We assume that all the above quantities exist, which is true for instance if the degree of the vertices of the graphs $\mbf G$ are bounded by a deterministic constant. If moreover the random graph is \emph{unimodular} \cite[Section 2.2]{BeCu12}, then $\tau$ is independent of the root of $T$, and $(\mcal A, \tau)$ is an algebraic traffic space (by applying \cite[Equation (2.3)]{BeCu12} to graph operations). This covers the case of random groups with given generator $(\Gamma, \gamma_1 \etc \gamma_n)$ which is identified with the Cayley graph of $\Gamma$ generated by $(\gamma_1\etc \gamma_n)$. 
\end{Ex}

\subsubsection{M\"obius inversion and injective trace}\label{Sec:Mobius}

In order to define traffic independence, we need first to define a transform of combinatorial traffic traces. It is based on a general principle that is used several times in this monograph. Recall that a poset is a set $\mcal X$ with a partial order $\leq$ (see \cite[Lecture 10]{NS} and \cite[Section 3.7]{Stan12}). Moreover $\mcal X$ is a lattice whenever every two elements have a unique supremum and a unique infimum. If $\mcal X$ is a finite lattice, then there exists a map $\mrm{Mob_{\mcal X}}:\mcal X\times \mcal X\to \mbb C$, called the M\"obius function on $\mcal X$, such that for two functions $F,G:\mcal X \to \mbb C$ the statement that 
	$$F(x) = \sum_{ x'\geq x} G(x'),  \ \forall x\in \mcal X$$
is equivalent to
	$$G(x) = \sum_{ x' \geq x} \mrm{Mob_{\mcal X}}(x,x') F(x'),  \ \forall x\in \mcal X.$$
Hence the first formula implicitly defines the function $G$ in terms of $F$.

For any set $X$, denote by $\mcal P(X)$ the poset of partitions of $X$ equipped with refinement order, that is $\pi\leq \pi'$ if the blocks of $\pi$ are included in blocks of $\pi'$. Let $(\mcal A, \Phi)$ be a non-commutative probability space and denote by $\mrm{N.C.}(K) \subset \mcal P(\{1\etc K\})$ the poset of non-crossing partitions of $\{1\etc K\}$ \cite[Lecture 9]{NS}. We recall that in an algebraic non-commutative probability space $(\mcal A, \Phi)$, the free cumulants are the multi-linear maps $(\kappa)_{L\geq 1}$ on $\mcal A^L$ given implicitly by 
	\eqa \label{Def:FreeCumm}
	 \Phi( a_1 \times \dots \times a_K) = \sum_{\pi \in \mrm{N.C.}(K)} \underbrace{\prod_{ \{i_1 < \dots < i_L\} \in \pi} \kappa_L(a_{i_1}\etc a_{i_L})}_{=:\kappa(\pi)}.
	\qea
With $\Phi(\pi)$ defined as $\kappa(\pi)$ using $\Phi(a_{i_1}\dots  a_{i_L})$ instead of $\kappa_L(a_{i_1}\etc a_{i_L})$, we can express $\kappa(\pi)$ in terms of $\Phi(\pi')$ for $\pi'\geq \pi$ thanks to M\"obius inversion in the poset on non crossing partitions. 

Let now $T=(V,E,\gamma)$ be a test graph in $\mcal T\langle \mathcal{A} \rangle$, with vertex set $V$. For any partition $\pi\in \mcal P(V)$ of $V$, we denote by $T^\pi=(V^\pi,E^\pi,\gamma^\pi)$ the test graph obtained by identifying vertices in a same block of $\pi$. More precisely:
\begin{itemize}
	\item  the vertex set of $T_\pi$ is the set of blocks of $\pi$,
	\item each edge $e=(v,w)$ of $T$ generates an edge $e^\pi =(B_v,B_w)$, where $B_v$ denotes the block of $\pi$ containing $v$,
	\item the label of $e^\pi$ is the label of $e$, namely $\gamma^\pi(e^\pi)=\gamma(e)$. 
\end{itemize}
We say that $T^\pi$ is a \emph{quotient} of $T$. Denote $0_V$ the partition of $V$ with singletons only (it then satisfies $T^{0_V} = T$).

 \begin{Def}\label{Def:TraffCum}Let $\mcal A$ be an ensemble and let $\tau : \mbb C \mcal T\langle \mathcal{A} \rangle \to \mbb C$ be a linear form. We define the injective version of $\tau$, and denote $\tau^0$, the linear form on $\mbb C \mcal T\langle \mathcal{A} \rangle$ implicitly given by the following formula: for any test graph $T\in \mcal T\langle \mathcal{A} \rangle$
 	\eqa\label{eq:TraffCum}
		\tau\big[ T\big] = \sum_{\pi \in \mcal P(V)} \tau^0 \big[ T^\pi\big],
	\qea
in such a way for any test graph $T$ one has
	$$ \tau^0 \big[ T\big] = \sum_{\pi \in \mcal P(V)} \mrm{Mob}_{\mcal P(V)}(0_V, \pi ) \cdot \tau\big[ T^\pi\big].$$
The injective version of a combinatorial trace (resp. a traffic distribution) is called the \emph{injective trace} (resp. the \emph{injective distribution}).
 \end{Def}

\begin{Ex}	 The injective version $ \Tr^0$ of the trace of test graph in random matrices  of $M_N(\C)$ defined in \eqref{eq:IntroDefTrace} is given, for $T=(V,E,M)$ a test graph labeled in $\mrm{M}_N\big(L^{\infty-}(\Omega, \mbb C )\big)$, by
\eqa\label{eq:IntroDefTraceInj}
	 \tau_N^0 \big[ T \big] =  \esp\Big[ \frac 1 N  \sum_{ \substack{ \phi: V \to [N]\\ \trm{injective}}}  \prod_{e=(w,v)\in E}  \big(M(e)\big) \big( \phi(w), \phi(v) \big) \Big].
	\qea
Limiting injective combinatorial distributions of usual matrix models (unitary Haar matrices, uniform permutation matrices, certain Wigner matrices) are proved to exist \cite[Chapter 3]{Male2011} and are shown to have simple and natural expressions.
	\end{Ex}
\begin{Rk} The map $\tau^0$ satisfies the property of multi-linearity w.r.t. the edges of Definition \ref{Def:Traffic}, but not the substitution property, see \cite[Section 2.1.]{Male2011}.
\end{Rk}

\subsubsection{Traffic independence}\label{Sec:DefFree}

Let $J$ be a fixed index set and, for each $j\in J$, let $\mcal A_j$ be some sets. Given a family of linear maps $\tau_j : \mathbb{C}\mcal T\langle \mcal A_j \rangle \to \mbb C, $ $j\in J$,  sending the graph with no edge to one, we shall define a linear map denoted $\star_{j\in J}\tau_j: \mathbb{C} \mcal T\langle \, \bigsqcup_{j\in J} \mcal A_j  \rangle$ with the same property and called  the free product of the $\tau_j$'s. The terminology \emph{free} product should be understood as \emph{canonical} product, and may not be confused with the terminology \emph{free} independence. Therein, $\bigsqcup_{j\in J}  \mcal A_j $ denotes the disjoint union of copies of $\mcal A_j$, although the sets $\mcal A_j$ can originally intersect or be equal: it is formally defined as the set of all couples $(j,a)$ where $j\in J$ and $a\in \mcal A_j$.

Let  us consider a test graph $T$ in $ \mathcal T\langle \, \bigsqcup_{j\in J}\mcal A_j \rangle$ and introduce an undirected graph as follow. We first call colored components of $T$ with respect to the families $(\mcal A_j)_{j\in J}$ the maximal nontrivial connected subgraphs of $T$ whose edges are labelled by elements of $\mcal A_j$ for some $j\in J$ (they are elements of $\mcal T\langle \mcal A_j \rangle$). There is no confusion about the definition of colored components because of the convention for $\bigsqcup_{j\in J}\mathcal{A}_j $. When there is no ambiguity about the collection $(\mcal A_j)_{j\in J}$, we denote by $\mcal C \mcal C(T)$ the set of colored components of $T$.  We call connectors of  $T$ the vertices of $T$ belonging to at least two different colored components. The graph $\mcal G\mcal C\mcal C( T)$ defined below is called graph of colored components of $T$ with respect to $(\mathcal{A}_j)_{j\in J}$:
\begin{itemize}
	\item the vertices of $\mcal G\mcal C\mcal C( T)$ are the colored components of $T$ and its connectors;
	\item there is an edge between a colored component in $\mcal C\mcal C( T)$  and a connector if the connector belongs to the component.
\end{itemize}
The following definition is from \cite[Section 2.2.]{Male2011}.

\begin{Def} \label{Def:Freeness}
	\begin{enumerate}
		\item For each $j\in J$, let  $\mcal A_j$ be a set and $\tau_j : \mbb C \mcal T\langle \mathcal{A}_j \rangle \to \mbb C$ be a linear map sending the test graph with no edges to one. The free product of the maps $\tau_j$ is the linear map $\star_{j\in J}\tau_j:  \mathbb{C}\mcal T\langle \, \bigsqcup_j \mcal A_j \rangle \to \mbb C$ whose injective version is given by: for any test graph $T$,
	\eqa
		(\star_{j\in J}\tau_j)^0 [T ] = \one\big( \mcal G \mcal C \mcal C(T)  \mrm{ \ is  \ a  \ tree} \big) \times   \prod_{ S  \in \mcal C \mcal C(T) } \ \tau_{j(S)}^0\big[S\big],
	\qea
where $j(S)$ is the index of the labels of $S$.
	\item  Let $(\mcal A, \tau)$ be an algebraic traffic space and let $J$ be a fixed index set. For each $j\in J$, let $\mcal A_j \subset \mcal A$ be a $\mcal G$-subalgebra. The subalgebras $(\mcal A_j)_{j\in J}$ are called traffic independent whenever the restriction of $\tau$ on the test graphs labeled by elements of $\mcal A_j, j\in J$, coincides with $\star_{j\in J}\tau_j$.
	\item Let $X_j, j\in J$ be subsets of $\mcal A$ and let $(\mbf a_j)_{j\in J}$ be a family of elements of $\mcal A$. Then $(X_j)_{j\in J}$ (resp. $(\mbf a_j)_{j\in J}$) are called traffic independent whenever the $\mcal G$-subalgebra induced by the $X_j$'s (resp. by the $\mbf a_j$'s) are traffic independent.  	
\end{enumerate}
\end{Def}

The motivation for introducing this definition is, in the context of large matrices, Example \ref{Ex:MtSpa}, the asymptotic traffic independence for permutation invariant matrices, see \cite[Theorem 1.8]{Male2011}.

We end this section by the following elementary property of traffic independence.
\begin{Lem}\label{Lem:DefTraffInd} Traffic independence is symmetric and associative, i.e. $\mcal A_1$ and $\mcal A_2$ are independent if and only if $\mcal A_2$ and $\mcal A_1$ are independent, and $\mcal A_j,j=1,2,3$ are independent if an only if $\mcal A_1$ and $(\mcal A_2,\mcal A_3)$ are independent and $\mcal A_2$ and $\mcal A_3$ are independent.
\end{Lem}

\newpage
\part{General traffic spaces}

\section*{Presentation of Part 1}

According to Section \ref{Sec:Recall}, traffic independence in an algebraic traffic space $(\mcal A, \tau)$  is defined in terms of the injective version $\tau^0$ of $\tau$, thanks to the formula involving the graph of colored components. Such a definition of independence is unusual  in non-commutative probability, where the injective trace has no analogue. As a comparison, let us remind the two equivalent definitions of free independence in free probability. It is usually defined by a relation of moments, namely the centering of alternated products of centered elements.  The second usual characterization of free independence is the vanishing of mixed free cumulants. 

We propose in Theorem \ref{Equivalence Free product free independence} of Section \ref{Sec:NaturalCha} a characterization of traffic independence in terms of moment functions as the centering of some \emph{generalized alternated products of reduced elements}, in an appropriate sense that we shall make precise. Note that Gabriel proposes in \cite{Gabriel2015a} a definition of traffic cumulants, and traffic independence is the vanishing of these mixed traffic cumulants. 

In Section \ref{Sec:PdtTrafSp}, we construct the product of traffic spaces: given for each $j\in J$ an algebraic traffic space $(\mcal A_j, \tau_j)$, we construct a new algebraic traffic space $(\mcal A, \tau)$ that contains the $\mcal A_j$ as independent $\mcal G$-subalgebras. The space $\mcal A$ will be made with equivalent classes of graph operations with an input and output whose edges are labelled by the $\mcal A_j$. The combinatorial trace $\tau$ will be the extension to $\mcal A$ of the free product of the combinatorial traces $\tau_j$, $j\in J$. 

Positivity of state is another important notion in noncommutative probability. We propose a definition of positivity for combinatorial trace in Section \ref{Sec:PosFreeProd}. We prove that the free product traffic spaces with positive traces also admits  a positive trace.

\section{A natural characterization of traffic independence}\label{Sec:NaturalCha}

\subsection{Statement}

In order to give the characterization of traffic independence which is the analogue of the usual presentation of freeness, we need a generalization of test graphs and graph polynomials with arbitrary numbers of marked vertices. To explain this fact, recall that the definition of traffic independence involves the graph of colored components. To define correctly the operation which consists in reconstructing a test graph from its colored components and its graph of colored components, we need formal objects that are specified in the two following definitions (see Figure \ref{Fig:17}).

  \begin{figure}[h]
    \begin{center}
     \includegraphics{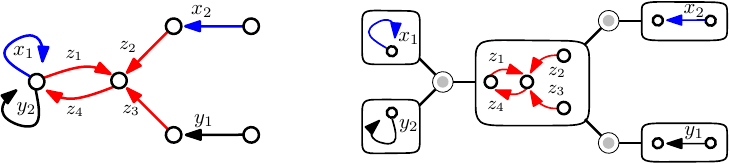}
    \end{center}
    \caption{Left: a test graph $T$ in three families of traffics $(x_1,x_2)$, $(y_1,y_2)$ and $(z_1,z_2,z_3,z_4)$. Note that $\tau[T] = \Phi[ \Delta(x_1)\Delta(y_2)\,  (z_4\circ z_1^t) \, \mrm{deg}(z_2\, x_2)\mrm{deg}(z_3\, y_1)\, z_2\, \Delta(y_1)\big]$. Right: the graph of colored component $\mcal{GCC}(T)$.}
    \label{Fig:17}
  \end{figure}

\begin{Def} A graph monomial of rank $n\geq 1$ (in short a $n$-graph monomial) labeled in $J$ is the data $g=(V,E,\gamma, \mbf v)$ of a test graph $T=(V,E,\gamma)$ and of a $n$-tuple $\mbf v=(v_1\etc v_n)$ of vertices of $T$, called the \emph{outputs}. We denote by $\mcal G^{(n)}\langle J \rangle $ the set of $n$-graph monomials and by $\mbb C \mcal G^{(n)}\langle J \rangle $ the space of $n$-graph polynomials.
\end{Def}

We have $\mbb C \mcal G^{(2)}\langle J \rangle = \mbb C \mcal G\langle J \rangle$ where a graph monomial of rank 2 is identified with the graph monomial whose input is the first output. A test graph is also called a $0$-graph monomial and we set $\mbb C\mcal G^{(0)}\langle J \rangle  := \mbb C \mcal T\langle J \rangle $.
To define generalized products of graph polynomials of arbitrary rank, we use the following objects, drawn in Figure \ref{Fig:08}.

\begin{Def} A \emph{bigraph operation of rank $n\geq 1$} (in short a $n$-bigraph operation) in $L\geq 0$ variables is the data of 
\begin{itemize}
	\item a finite, connected, undirected and bipartite graph $g$, endowed with a  bipartition of its vertices into two sets $V_{\mathit{in}}(g)$ and $V_{\mathit{co}}(g)$, whose elements are called \emph{inputs} and \emph{connectors},
	\item with exactly $L$ ordered inputs, given together with an ordering of its edges around each input 
	\item and the data of an ordered subset $V_{\mathit{out}}(g)$ consisting in $n$ elements of the connectors $V_{\mathit{co}}(g)$ that we call outputs,
\end{itemize}
and such that all connectors that are not an output have degree greater than or equal to $2$. We denote by $\mcal B^{(n)}$ the set of $n$-bigraph operations. For any  $L,n\geq 0$ and any tuple $\mbf d=(d_1\etc d_L)\in (\N^*)^L,$ we denote by $\mcal B_{L,\mbf d}^{(n)}$ if $L\neq0$ and by $\mcal B_0^{(n)}$ otherwise the set of $n$-bigraph operations with $L$ inputs such that the $\ell$-th one has degree $d_\ell.$ \end{Def}

A $n$-bigraph operation in $L$ variables with degrees $d_1,\ldots, d_L$ has to be thought as an operation that accepts $L$ objects with ranks $d_1,\ldots, d_L,$ and produces a new object of rank $n$. The set of bi-graph operations is actually an operad, although we do not use this fact (see Section \ref{Sec:ConcluI} for comments). 

In particular, a $n$-bigraph operation can produce a new $n$-graph monomial from $L$ different graph monomials in the following way, see Figure \ref{Fig:08}. Let us consider $L$ graph monomials $t_1,\ldots, t_L$ labeled on some set $\mcal A$, with respective number of outputs given by $\mbf d\in (\N^*)^L$ (that is $t_\ell \in  \mcal G^{(d_\ell)}\larac$),  and a bigraph operation  $g\in \mcal B_{L,\mbf d}^{(n)}.$   Replacing the $\ell$-th input of $g$ and its adjacent ordered edges $(e_1,\ldots,e_{d_{\ell}})$ by the graph of $t_\ell$, identifying for each  $k\in[L]$ the  connector attached to $e_k$ with the $k$-th output of $t_\ell$, yields a connected graph. We  denote   by $T_g(t_1\otimes \ldots \otimes t_L)\in  \mcal G^{(n)}\larac$ the $n$-graph monomial whose labelling is induced by those of $t_1,\ldots, t_L$, and with outputs given by the outputs of $g$.   We then define by linear extension 
\begin{align*}
T_g: \mbb C\mcal G^{(d_1)}\larac \otimes \dots \otimes \mbb C\mcal G^{(d_L)}\larac &\longrightarrow \mbb C\mcal G^{(n)}\larac \\ 
t_1\otimes\ldots\otimes t_L&\longmapsto T_g(t_1\otimes\ldots\otimes t_L).
\end{align*}

  \begin{figure}
    \begin{center}
     \includegraphics{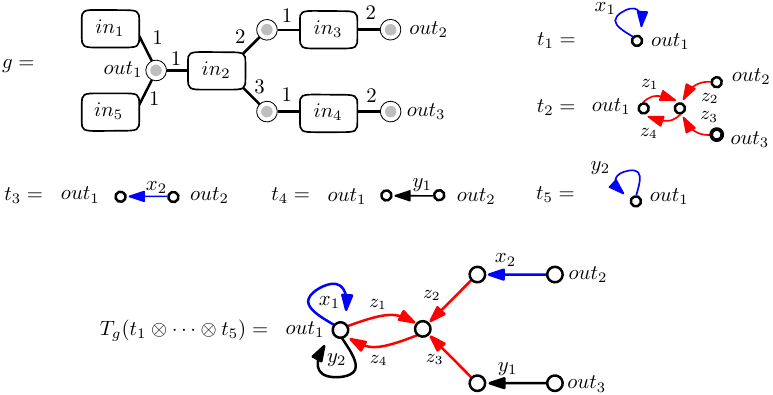}
    \end{center}
    \caption{A bigraph operation $g$ of order 3 with 5 inputs and 3 outputs with degree sequence $(1,3,2,2, 1 )$; the numbers in the figure describe the order of the edges around each input. Five graph operations $t_1\etc t_5$ which satisfy that $t_1\otimes \cdots \otimes t_5$ is $g$-alternated. The graph operation $T_g(t_1\otimes \cdots \otimes t_5)$}
    \label{Fig:08}
  \end{figure}

\begin{Ex} \label{Rk:BigraphOp}
\begin{itemize} 
	\item Let $\mbf A_N = (A_j)_{j\in J}$ be a family of matrices and $t=(V,E, \gamma, \mbf v)$ be a $n$-graph monomial labeled in $J$. We define a random tensor matrix $t(\mbf A_N) \in (\mbb C^N)^{\otimes n}$ as follows. Denoting by $\mbf v=(v_1\etc v_n)$ the sequence of outputs of $t$ and by $(\xi_i)_{i=1\etc N}$ the canonical basis of $\mbb C^N$, we set,
	\eqa\label{Rk:BigraphOp1}
		t(\mbf A_N) =\sum_{ \phi: V \to [N]}  \prod_{e=(v,w)\in E} A_{\gamma(e)}\big( \phi(w), \phi(v) \big) \xi_{\phi(v_1)} \otimes \dots \otimes \xi_{\phi(v_n)}.
	\qea
	
	\item More generally, let $g$ be a $n$-bigraph operation with $K$ inputs and $A_1 \etc A_K$ be tensors matrices such that the rank $n_k$ of $A_k$ (so that $A_k\in(\mbb C^N)^{\otimes n_k}$) is the degree of the $k$-th input of $g$. Denote by $(v_1\etc v_n)$ the outputs of $g$ and for each $k=1\etc K$ denote by $(w_1^k \etc w_{n_k}^k)$ the ordered neighborhood connectors of the $k$-th input. Then we define a element of $(\mbb C^N)^n$ by
	\eqa\label{Rk:BigraphOp2}
		T_g(\mbf A_N) = \sum_{\phi: V_{\mathit{co}}(g)  \to [N]} \prod_{ k=1}^K A_k\big( \phi(w_1^k) \etc \phi(w_{n_k}^k)  \big) \xi_{\phi(v_1)} \otimes \dots \otimes \xi_{\phi(v_n)}.
	\qea
\end{itemize}
\end{Ex}

\begin{Def} Let $J$ be an index set and $(\mcal A_j)_{j\in J}$  be  a family of ensembles, and let $g\in \mcal B_{L,\mbf d}^{(n)}$ be a bigraph operation with $\mbf d=(d_1,\ldots,d_L)$. A tensor product $(t_1 \otimes \dots \otimes t_{L'})$ of graph polynomials labeled by $\bigsqcup_j \mcal A_j$ is alternated along $g$ (in short \emph{$g$-alternated}) whenever
\begin{enumerate}
	\item  $L'=L$, 
	\item $t_i\in \C\mcal G^{(d_i)} \langle \mcal A_{j_i}\rangle$ for each $i=1\etc L$, and
	\item for all $p,q\in [L]$ such that the $p$-th and  the $q$-th inputs are neighbors of a same connector,  then $j_p\neq j_q.$
\end{enumerate}
\end{Def}

Let $T_g$ be a bigraph operation and let $m_1 \otimes \dots \otimes m_L$ be a tensor product of graph monomials, labeled in a set $\bigsqcup_j \mcal A_j, j \in J$, alternated along $g\in \mcal B_{L,\mbf d}^{(0)}$. Assume that $T_g$ does not identify any pair of outputs of each $m_\ell$ and that the output vertices of each $m_\ell$ are pairwise distinct. Then $T_{g}(m_1\otimes \ldots \otimes m_L)$ is a test graph with graph of colored components $g,$ and its colored components are  $m_1,\ldots,m_L,$ (considered as graphs with no outputs). Reciprocally, the graph of colored component gives a decomposition of any test graph as an element of the form $T_{g}(m_1\otimes \ldots \otimes m_L)$. This decomposition is unique up to the symmetry of a certain automorphism group introduced later in Section \ref{Sec:PosFreeProd}.

We shall now define a notion of \emph{reduced} $n$-graph polynomials. For any $n\ge 2,$ any partition $\pi\in \mcal P(n)$ of $\{1\etc n\}$, and any $n$-graph monomial $g$ with outputs $(v_1,\ldots, v_n)$, let us denote by $g^\pi$  the quotient graph obtained by identifying vertices $v_1,\ldots, v_n$ that belong to a same block of $\pi$, with outputs given by the images of $(v_1,\ldots,v_n)$ by the quotient map, so that the edges of $g^\pi$ can be identified with the one of $g.$ This defines a linear map $\Delta_\pi: \mbb C\mcal G^{(n)} \larac\to \mbb C\mcal G^{(n)} \larac$ such that $\Delta_\pi (g) = g^\pi$ for $n$-graph monomials $g$. The map $\Delta_\pi$ can also be seen as the action of a bigraph operation (see an example in Figure \ref{Fig:14}). Denote respectively by $0_n$ and $1_n$ the partitions of $\{1\etc n\}$ made of $n$ singletons and of one single block respectively. Note that $\Delta_{0_n}(g) = g$ for any $g\in \mbb C \mcal G^{(n)}\larac,$

  \begin{figure}
    \begin{center}
     \includegraphics{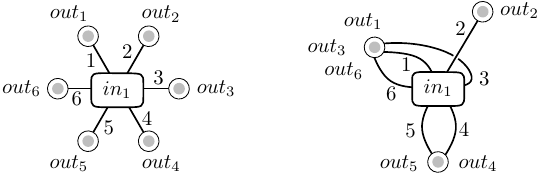}
    \end{center}
    \caption{The bi-graph operations $\Delta_{ 0_6}$ (left) and $\Delta_{\left \{ \{1,3,6\} , \{2\} , \{4,5\} \right \}}$ (right). }
    \label{Fig:14}
  \end{figure}

\begin{Def} Let $\mcal A$ be an ensemble and $\tau:\C\mcal T \larac\to\C$ be a linear form. We extend $\tau$ in a linear map $\C\mcal T \larac\oplus\C\mcal G^{(1)} \larac\to\C$ by forgetting the position of the output in 1-graph monomials. A $n$-graph polynomial $t\in \C\mcal G^{(n)}\larac$ is called \emph{reduced} with respect to $\tau$, if
\begin{itemize}
	\item $n\in \{0,1\}$ and $\tau(t)=0$, or
	\item $n\ge 2$ and for any $\pi \in \mathcal P(n)\setminus \{0_n\}$ one has  $\Delta_\pi(t)=0$.
\end{itemize}
\end{Def}

Note that the reduceness condition does not depend on $\tau$ when $n\geq 2$.

\begin{Ex} If $n=2$, then $\Delta_{1_2}(t)=\Delta (t),$ where we recall that the diagonal operator $\Delta$ is the graph operation with one vertex and one edge. So $t$ is reduced if and only if $\Delta(t)=0$. 
\end{Ex}

\begin{Ex}Let $\mbf A_N$ be a family of matrices of size $N$ by $N$ and let $t$ be a $n$-graph polynomial, $n\geq 2$. Then the tensor matrix $t(\mbf A_N)$ defined in Example \ref{Rk:BigraphOp} is reduced if and only if, denoting by $B_{\mbf i}, \mbf i \in [N]^n,$ its components in the canonical basis, one has $ B_{\mbf i}=0$ as soon as two indices of $\mbf i$ are equal. In particular for $n=2$, a matrix is reduced whenever its diagonal entries are equal to zero.
\end{Ex}

We can now state the main result of the section.

\begin{Th} \label{Equivalence Free product free independence}Let $(\mcal A, \tau)$ be an algebraic traffic space with trace $\Phi$ and anti-trace $\Psi$. For each $j\in J$ let $\mcal A_j$ be a $\mcal G$-subalgebra. The following properties are equivalent:
\begin{enumerate}
	\item The $\mcal G$-subalgebras  $\mcal A_j, j\in J,$ are traffic independent (Definition \ref{Def:Freeness}),
	\item One has $\tau[h]=0$ for any $h=T_g(t_1\otimes \dots \otimes t_L)$ in $ \mbb C \mcal T\langle \bigsqcup_j \mcal A_j \rangle$ where $g\in \mcal B^{(0)}$ is a bigraph operation and $t_1\otimes \dots \otimes t_L$ is a $g$-alternated tensor product of reduced elements with respect to $\tau$.
	\item  One has $\Phi[h] = 0$ for any $h=T_g(t_1\otimes \dots \otimes t_L)$ in $ \mbb C \mcal G^{(2)}\langle \bigsqcup_j \mcal A_j \rangle$, where $g\in \mcal B^{(2)}$ is a bigraph operation and $t_1\otimes \dots \otimes t_L$ is a $g$-alternated tensor product of reduced elements with respect to $\tau$.
	\item  One has $\Psi[h] = 0$ for any $h=T_g(t_1\otimes \dots \otimes t_L)$ in $ \mbb C \mcal G^{(2)}\langle \bigsqcup_j \mcal A_j \rangle$, where $g\in \mcal B^{(2)}$ is a bigraph operation and $t_1\otimes \dots \otimes t_L$ is a $g$-alternated tensor product of reduced elements with respect to $\tau$.
\end{enumerate}
\end{Th}
Hence traffic independence is the centering of alternated bigraph operations of reduced elements with respect to $\tau, \Phi$ or $\Psi$ indifferently. The proof of the proposition is given in the next section.

As a direct application, we get a useful criterion of free independence.
\begin{Cor}\label{Corfreeind}
Let $(\mcal A, \tau)$ be an algebraic traffic space such that $\mcal A$ is a $^*$-algebra and the associated trace $\Phi$ is a state. Denote for any $a\in \mcal A$
	$$\eta(a)=\tau\big[^{a}\circlearrowleftRotBis \cdot  \circlearrowleftRot^{a^*}\big] -| \tau[\circlearrowleftRot^{a}]|^2= \Phi\big(\Delta(a^*)\Delta(a)\big)-|\Phi(a)|^2=\Phi (a^*\circ a )-|\Phi(a)|^2,$$
where we recall (see section \ref{Sec:AlgOp}) that $\Delta  = Z_{\circlearrowleft}$ is the diagonal operator and $(a\circ b) = Z_{\cdot \leftleftarrows \cdot}(a\otimes b)$ is the Hadamard product. Let $\mcal B\subset \mcal A$ be a unital $^*$-subalgebra such that $\eta(a)=0$ for any $a\in \mcal B$, and let $\mcal B_j \subset \mcal B$, $j\in J$, be subalgebras.  If $(\mcal B_j)_{j\in J}$ are traffic independent in $(\mcal A, \tau)$, then they are freely independent in the $^*$-probability space $(\mcal B, \Phi_{|\mcal B})$.
\end{Cor}

\begin{Ex} \begin{enumerate}
	\item In \cite[Proposition 2.16]{Male2011}, it is proved that two independent traffics $a$ and $b$ such that $\eta(a)\neq 0 \neq \eta(b)$ are not free independent with respect to the trace. If $\eta(a)\neq 0$ and $\eta(b)=0$, both situations can happen as we can see with the limits of Wigner matrices, uniform permutation matrices and diagonal matrices \cite{Male2011}: the map $\eta$ vanishes only for the two first models, a Wigner matrices is asymptotically free from a diagonal matrix, but a uniform permutation matrix is not asymptotically free from a diagonal matrix.
	\item In the context of the so-called \emph{asymptotically unitarily invariant} random matrices defined in Part \ref{Sec:CanonicalExtension}, the assumption of Corollary \ref{Corfreeind} is satisfied. Nevertheless we will see that in this particular case the free independence is explained in a more direct way and has stronger implications.
	\item Yet, the above corollary covers a much larger situation than the example previously mentionned. The example of the large uniform permutation matrix can be generalized for infinite rooted graphs. More precisely, recall Example \ref{Ex:Graphs2} of the $\mcal G$-algebra $\mcal A$ of locally finite rooted graphs on a set of vertices $\mcal V$. It is a classical fact that an element $A$ of $\mcal A$ which is both deterministic and unimodular is vertex-transitive (there exist automorphisms exchanging each pair of vertices). This property implies that the diagonal $\Delta(A) = \big(A(v,v) \one_{v=w}\big)_{v,w\in \mcal V}$ of $A$ is constant, and so one can apply the lemma. This gives a new proof of a result of Accardi, Lenczewski and Salapata \cite{Accardi2007} stating that the spectral distribution of the free product of infinite deterministic graphs is the free product of the spectral distributions.
\end{enumerate}
\end{Ex}

 \begin{proof}[Proof of Corollary \ref{Corfreeind}]Since the trace defined on $\mcal A$ is a state, the assumption implies, for every $a\in \mcal B$, that $\Delta(a)$ has the same $^*$-distribution as $\Phi(a)\mbb I$. Let $(\mcal B_j)_{j\in J}$ traffic independent $^*$-subalgebras of $\mcal B$. Let $a_1,\ldots,a_n \in \mcal B,$ such that for any $k\in [n],$ $\Phi(a_k) =0$ and $a_k\in \mcal B_{j_k},$ with $j_{1}\neq j_{2} \neq  \dots \neq j_n $. Then, 
 $$\Phi\Big( \big(a_1-\Delta(a_1)\big) \ldots\big(a_{n}-\Delta(a_{n})\big) \Big)=\Phi\Big(\big(a_1-\Phi(a_1)\big)\ldots \big(a_{n}-\Phi(a_{n})\big) \Big)=\Phi(a_1\ldots a_n ).$$
 Let $g$ be the bigraph operation with two outputs $in$ and $out$, $n$ inputs and $n$ connectors, whose graph is a directed line from $in$ to $out$, with input vertices (alternating with the connectors) ordered consecutively from $in$ to $out$. Then one has
 $$\Phi\Big( \big(a_1-\Delta(a_1)\big) \ldots\big(a_{n}-\Delta(a_{n})\big) \Big)=\Phi\Big[T_g\Big( \big(a_1-\Delta(a_1)\big) \otimes\ldots\otimes\big(a_{n}-\Delta(a_{n})\big) \Big)\Big],$$
and $\big(a_1-\Delta(a_1)\big) \otimes\ldots\otimes\big(a_{n}-\Delta(a_{n})\big)$ is a  $g$-alternated tensor product of reduced elements, so that by Theorem \ref{Equivalence Free product free independence} we get $\Phi(a_1\ldots a_n)=0$. 
 \end{proof}

\subsection{Proof of Theorem 2.8}\label{seq:Equivalence Free product free independence}

\subsubsection{A decomposition of graph polynomials}

We start by stating several preliminary lemmas. The first three statements are about the space of $n$-graph polynomials $\mbb C \mcal G\langle \bigsqcup_j \mcal A_j\rangle$. Note that in these lemmas we only assume that the sets $\mcal A$ and $\mcal A_j$, $j\in J$, are arbitrary ensembles, we do not use their $\mcal G$-algebra structure. The first lemma gives an explicit characterization of reducedness.

\begin{Lem}\label{Mobius} Let $\mcal A$ be an ensemble and for $n\geq 0$ let $m$ a $n$-graph monomial labeled in $\mcal A$. Denote by $\mcal O$ the output set of $m$ (empty if $n=0$). For each partition $\sigma$ of $\mcal O$, recall that $\Delta_\sigma(m)=m^\sigma$ denotes the graph monomial obtained by identifying the outputs of $m$ that belong to a same block of $\sigma$. Let us denote by $\mrm{Mob}$  the M\"obius function for the poset of partitions of $\mcal O$ (Section \ref{Sec:Mobius}) and $0_{\mcal O}$ the partition of $\mcal O$ made of singletons. Then, with $(\cdot)$ denoting the graph with no edges,
	$$p(m):= \left\{ \begin{array}{ccc}
		m - \tau(m) \times (\cdot) & \mrm{if} & n=1,0,\\
		\sum_{\sigma\in \mcal P(\mcal O)} \mrm{Mob}(0_{\mcal O}, \sigma)m^\sigma & \mrm{if} & n\geq 2,
		\end{array} \right.
	$$
 is a reduced $n$-graph polynomial with respect to $\tau$. Moreover, extending $p$ by linearity on $n$-graph polynomials, every reduced $n$-graph polynomial $t$ satisfies $t=p(t)$.
\end{Lem}
\begin{proof}The proposition is clear if $n=0,1$. Assume $n\geq 2$ in the following. For any $\nu \in \mcal P(\mcal O),$ 
\begin{align*}
\Delta_\nu\big(p(m) \big) = \Delta_\nu \left(\sum_{\sigma\in \mcal P(\mcal O)} \mrm{Mob}(0_{\mcal O}, \sigma)m^\sigma\right)=\sum_{\mu\in \mcal P(\mcal O)}  \left(\sum_{\sigma\in \mcal P(\mcal O):\sigma\vee\nu=\mu } \mrm{Mob}(0_{\mcal O},\sigma)\right) m^\mu,
\end{align*}
where $\sigma\vee\nu$ is the join of the partitions $\sigma$ and $\nu$, i.e. the smallest partition whose blocks contain those of $\sigma$ and $\nu$.
Now, for any $\mu\in \mcal P(\mcal O),$  by \cite[Sections 3.6 and 3.7]{Stan12} for the first and last equalities, one has
\begin{align*}
\sum_{\sigma\in \mcal P(\mcal O):\sigma\vee\nu=\mu }\mrm{Mob}(0_{\mcal O},\sigma)&=\sum_{\sigma\le \mu} \ \sum_{\sigma\vee \nu\le \xi\le \mu}\mrm{Mob}(\xi,\mu) \mrm{Mob}(0_{\mcal O},\sigma)\\
&=\sum_{\nu\le \xi\le \mu}\mrm{Mob}(\xi,\mu) \left(\sum_{\sigma\le \xi}\mrm{Mob}(0_{\mcal O},\sigma)\right)\\
&=\sum_{\nu\le \xi\le \mu}\mrm{Mob}(\xi,\mu)\delta_{\xi,0_{\mcal O}}=\delta_{\nu,0_{\mcal O}}\mrm{Mob}(0_{\mcal O},\mu),
\end{align*}
Hence we have obtained $ \Delta_\nu \big(p(m)\big) = \delta_{\nu, 0_{\mcal O}} p(m)$, that is $p(m)$ is reduced.

Let us now prove that every reduced graph polynomial $t$ satisfies $t=p(t)$. For any $\eta\in \mcal P(\mcal O)$ let us define $p_\eta(m) = \sum_{\pi \geq \eta} \mrm{Mob}(\eta, \pi) m^\pi$. Extended by linearity, the $p_\sigma$'s define a partition of the unity, that is  $t = \sum_{\eta\in \mcal P(O)}p_\eta(t)$ for any $t$. By the same computation as above, one sees that $t$ is reduced if and only if $p_\eta(t)=\delta_{\eta = 0_{\mcal O}}t$ for any $\eta$. Hence we obtain $t=p_{ 0_{\mcal O}}(t)=p(t)$ as expected.
\end{proof}

The second lemma tells that any $n$-graph polynomial in $\mbb C \mcal G\langle \bigsqcup_j \mcal A_j\rangle$ can be written as a linear combination of bigraph operations evaluated in alternated and reduced elements.

\begin{Def}\label{Def:ColBigraph} Let $J$ be an index set and, for each $j\in J$, let $\mcal A_j$ be an ensemble.
\begin{itemize}
	\item A colored bigraph operation with color set $J$ is a couple $(g,\gamma)$ where $g\in \bigcup_{n\geq 0}\mcal B^{(n)}$ is a bigraph operation with $L\geq 1$ inputs and $\gamma : [L] = \{1\etc L\}\to J$ is a map telling that the $\ell$-th input is of color $\gamma(\ell)$. With small abuse, we still denote $g$ instead of $(g, \gamma)$ the colored bigraph operation with implicit mention of $\gamma$. We say that $g$ is alternated if $\gamma$ associates distinct colors to the neighbours of a same connector. We denote by $\mcal B^{(n)}_{col}$ the set of colored bigraph operations with $n\geq 0$ outputs and by $\mcal B^{(n)}_{alt}$ the set of alternated colored bigraph operations.
	\item Let $t_1\etc t_L$ be graph polynomials of arbitrary ranks in $\bigsqcup_j \mcal A_j$. We say that the tensor product $\mbf t =(t_1\otimes \dots \otimes t_L)$ is $g$-colored if $t_\ell \in \mbb C \mcal G^{(d_\ell)} \langle \mcal A_{\gamma(\ell)} \rangle$ for any $\ell=1\etc L$.
\end{itemize}
\end{Def}

\begin{Lem}\label{Lem:DSD} Let $J$ be an index set, let $\mcal A_j$ be an ensemble for each $j\in J$, and let $\tau: \mbb C \mcal T\langle \bigsqcup \mcal A_j \rangle \to \mbb C$ be a unital linear form. Then we have the decomposition
	\eq
		\mbb C \mcal G^{(n)}\langle \bigsqcup_{j\in J} \mcal A_j \rangle 
		 & = & \mbb C \, (\cdot)  + \sum_{\substack{g\in \mcal B^{(n)}_{alt}}} \mcal W_{g}
	\qe
where $\mcal W_{g}$ is the space generated by $T_g(t_1\otimes \dots \otimes t_L)$, for any $(t_1\otimes \dots \otimes t_L)$ which is a $g$-colored tensor product of reduced elements with respect to $\tau$, and $\mbb C \, (\cdot)$ denotes the space generated by the graph monomial $(\cdot)$ with a single vertex and no edge in $\mbb C \mcal G^{(n)}\langle \bigsqcup_{j\in J} \mcal A_j \rangle$. 
\end{Lem}

 \begin{proof} Let us denote by $\mcal E^0$ the vector space on the right hand side, spanned by $(\cdot)$ and the $\mcal W_g$'s. For any $k\geq 1$, let us denote by $\mcal E_k$ the vector space generated by the graph polynomials $T_g(\mbf t) $, where $g\in \mcal B_{col}$ has a number of vertices less than or equal to $k$ and $\mbf t$ is $g$-colored. Let us prove by induction that  for any $k\ge 1,$ $\mcal E_k\subset \mcal E^0$. Since $\mbb C \mcal G^{(n)}\langle \bigsqcup_{j\in J} \mcal A_j \rangle  = \bigcup_{k\geq 0} \mcal E_k$, this shall conclude the proof.

 To begin with, note that for any $n\geq 1$ the only element of $\mcal E_1$ is $g=(\cdot)$ consists in a single connector vertex which is the common values of all outputs. Hence $\mcal E_1= \mbb C (\cdot) \subset \mcal E^0.$ If $n=0$, then $g=(\cdot)$ consists in a single input vertex and $W_{g,\gamma}$ is the linear space generated by the $\mbb C \mcal T\langle  \mcal A_j \rangle$, $j\in J$. Every element $T$ in this space can be written $T = \tau[T] (\cdot) + \big( T - \tau[T] (\cdot) \big) \in \mbb C \mbb I \oplus_j W_{(\cdot),j}$.

 Let us now assume the claim for $k\in\N$. For any $k'\geq 1$ and any $s\geq 0$ we denote by $\mcal E_{k'}^s$, the vector space spanned by the graph polynomials $T_g(\mbf t) $ of $\mcal E_{k'}$ where at most $s$ elements are non reduced in $\mbf t$. Note in particular that $\mcal E^0 = \bigcup_{k'\geq 0} \mcal E^0_{k'}$ and $\mcal E_{k'} = \bigcup_{s\geq 0} \mcal E_{k'}^s$. Let us prove by induction on $s\ge 0$ that $\mcal E^s_{k+1}\subset \mcal E.$
 
 We first assume that $\mcal E_{k+1}^s\subset \mcal E$ for some $s\geq 0$ and consider $T_g(\mbf t) $, a bigraph operation $g$ with $k+1$ vertices evaluated in a $g$-colored tensor product  $\mbf t$ with $s+1$ non reduced elements. Without loss of generality, we can assume the first graph $t_1$ is not reduced. We will denote $t_1\in \mbb C\mcal G^{(d_1)} \mcal \langle \mcal A_{\gamma(1)}\rangle $. If the rank $d_1$ of $t_1$ is one, then we can write $T_g(\mbf t) = T_g\left((t_1-\Phi(t_1))\otimes t_2 \ldots\otimes t_L \right) + \Phi(t_1) T_g\left((\cdot)\otimes t_2 \ldots\otimes t_L\right) =:a+b$ where $a\in \mcal E_{k+1}^s$ and  $b\in \mcal E_{k},$ so that $T_g(\mbf t)\in \mcal E$. If the rank of $t_1$ is greater than one, according to Lemma \ref{Mobius} we can write $t_1=r+\sum_{i=1}^m x_i$, where $r\in \mbb C\mcal G^{(d_1)} \langle \mcal A_{\gamma(1)}\rangle$ is a reduced graph polynomial and  $x_1,\ldots, x_m\in\mbb C\mcal G^{(d_1)} \langle\mcal A_{\gamma(1)}\rangle$ are graph monomials having at least two outputs equal to the same vertex. Then, for any $i=1\etc m,$   $T_g\left(x_i\otimes t_2\ldots\otimes t_L\right)\in \mcal E_{k}$ and  $T_g(r\otimes t_2\otimes \ldots \otimes t_L)\in \mcal E_{k+1}^s,$ so that $T_g(\mbf t)\in \mcal E.$
 \end{proof}
Below, $p$ denotes the operator defined in Lemma \ref{Mobius}. 

\begin{Cor}\label{Cor:Monom} In the setting of Lemma \ref{Lem:DSD}, the linear space $\mbb C \mcal G^{(n)}\langle \bigsqcup_{j\in J} \mcal A_j \rangle$ is generated by the $n$-graph polynomials of the form $ T_g\big( p(m_1) \otimes \dots \otimes p(m_L) \big)$, where $g\in \mcal B^{(n)}_{alt}$ and $m_1\otimes \dots \otimes m_L$ is a $g$-colored tensor product of monomials, such that outputs of the $m_\ell$'s are pairwise distinct and $T_g$ does not identify any pair of outputs of each input. 
\end{Cor}

\begin{proof} Let $\mbf t'=(t_1,\dots, t_L)$ be an arbitrary sequence of $g$-alternated, reduced graph polynomials and denote $t_\ell = \sum_i \alpha^{(\ell)}_i m_{i,\ell}$ where the $ m_{i,\ell}$'s are graph monomials. Then we have
	\eq
	  T_g(t_1\otimes \dots \otimes t_L)  & = & 
		 T_g\big(p(t_1)\otimes \dots \otimes p(t_L)\big) \\
		 & = & \sum_{i_1\etc i_L}  \Big( \prod_{\ell=1}^L \alpha^{(\ell)}_{i_\ell}   \Big) \times  T_g\big(p(m_{i_1,1})\otimes \dots \otimes p(m_{i_L,L})\big).
	\qe
By Lemma \ref{Lem:DSD}, we get that $\mbb C \mcal G^{(n)}\langle \bigsqcup_{j\in J} \mcal A_j \rangle$ is generated by the elements of the form $ T_g\big( p(m_1) \otimes \dots \otimes p(m_L) \big)$, where $m_1\otimes \dots \otimes m_L$ is a $g$-alternated tensor product of monomials. Moreover, if $m$ has two outputs that are equal, then $p(m)=0$. Hence one can assume that the outputs are pairwise distinct for each $m_\ell$.
\end{proof}

\subsubsection{Solidity, validity and primitivity}

This section contains most of the arguments of the proof of Theorem \ref{Equivalence Free product free independence} and it introduces tools that will be used later, in particular in Section \ref{Sec:PosFreeProd} to prove the positivity of the free product. 

In the first statement, we see how the reducedness of $n$-graph polynomials for $n\geq 2$ simplifies the computation of combinatorial traces (reducedness when $n=1$ plays a role at the last stage of the proof). We shall need the following definition.

\begin{Def}\label{Def:Solid} Let $\mcal A$ be an ensemble and let $T=T_g(m_1\otimes \dots \otimes m_L)$ be a test graph in $\mcal T\langle  \mcal A \rangle$, where $g$ is a bigraph operation and $(m_1\otimes \dots \otimes m_L)$ is a tensor product of graph monomials, such that outputs of a same $m_\ell$ are pairwise distinct and the operation $T_g$ does not identify any pair of outputs of each $m_\ell$. Consider the graphs of the $m_\ell$'s as subgraphs of $T$ and denote 
\begin{itemize}
	\item by $V$ the vertex set of $T$,
	\item by $\mcal O_\ell \subset V$ the set of outputs of $m_\ell$,
	\item by $\pi_{|\mcal O_\ell}$ the restriction of $\pi\in \mcal P(V)$ on $\mcal O_\ell$, namely $\{ B \cup \mcal O_\ell, B \in \pi\}$, 
	\item by $0_{\mcal O_\ell}$ the partition of $\mcal O_\ell$ made of singletons.
\end{itemize}
Consider a partition $\pi \in \mcal P(V)$. For each $\ell$ in $\{1\etc L\}$, we say that $m_\ell$ is \emph{solid} for $\pi$ whenever $\pi_{|\mcal O_\ell} = 0_{\mcal O_\ell}$. In other words, in $T^\pi$ there is no identification of outputs of the graph $m_\ell$. In a context where there is no confusion about $m_1\etc m_L$, we simply say that $\pi$ is solid, when $m_\ell$ is solid for $\pi$  for any $\ell=1\etc L$.\end{Def}

Beware that there is no uniqueness in the decomposition $T=T_g(m_1\otimes \dots \otimes m_L)$.

\begin{Lem}\label{Lem:RedEffect} Let $\mcal A$ be an ensemble and let $h = T_g(t_1\otimes \dots \otimes t_L) \in \mbb C \mcal T\langle    \mcal A \rangle$ be a $0$-graph polynomial, where $g \in \mcal B^{(0)}$ is a bigraph operation and 
\begin{itemize}
	\item $t_\ell = m_\ell$ is a monomial if $n_\ell=1$,
	\item $t_\ell = p(m_\ell)$ where $m_\ell$ is a monomial with pairwise distinct outputs if $n_\ell\geq 2$.
\end{itemize}
Let $T$ denote the test graph $T_g(m_1\otimes \dots \otimes m_L) \in  \mcal T\langle \mcal A \rangle$. Then the trace of $h$ is the sum of the quotient graphs of $T$ by solid partitions: with notations of Definition \ref{Def:Solid}, one has
$$\tau[h] = \sum_{\substack{\pi\in \mcal P(V) \\ \mrm{solid} }} \tau^0\big[T^\pi\big].$$
\end{Lem}

\begin{proof}

Without loss of generality, we can assume that the indices $\ell\in \{1\etc L\}$ such that $n_\ell\geq 2$ are $1\etc K$ for $K\leq L$. Let us denote $c_{\sigma_k} =  \mrm{Mob}(0_{\mcal O_k}, \sigma_k)$ for any $\sigma_k\in \mcal P(\mcal  O_\ell)$ and any $k=1\etc K$. Consider the graph $T_{\sigma}=T_g( m_1^{\sigma_1} \otimes \dots \otimes m_L^{\sigma_L} )$, with the convention that $m_\ell^{\sigma_\ell} = m_\ell$ if $\ell>K=1$. 
The definition of $p$ in Lemma~\ref{Mobius} allows to write
$$\tau[h]=\sum_{\substack{ \sigma_\ell\in \mcal P(\mcal O_\ell) \\ \forall \ell=1\etc K} } \Big(  \prod_{k=1}^K c_{\sigma_k} \Big) \tau[T_{\sigma}] . $$
Denoting by $V_{\sigma}$ the vertex set of $T_\sigma$, the linearity of $\tau$ and the definition of the injective trace lead to

\eqa\label{Eq:ProofCharIndep}
	\tau[h] & = & \sum_{\substack{ \sigma_\ell\in \mcal P(\mcal O_\ell) \\ \forall \ell=1\etc K} } \Big(  \prod_{k=1}^K c_{\sigma_k}   \Big)  \sum_{ \substack{ \pi\in \mcal P(V_{ \sigma}) }} \tau^0\big[T_{\sigma}^\pi \big].
\qea
Recall that for two partitions $\pi$ and $\pi'$ of some set, $\pi\leq \pi'$ means that the blocks of $\pi$ are included in blocks of $\pi'$. Given $\sigma_1 \etc \sigma_L$ as above, forming a graph $T_\sigma^{\mu}$ with a choice of a partition $\mu$ of $V_\sigma$ is equivalent to forming a graph $T^\pi$ with a choice of a partition $\pi$ of $V$ with the restriction below. 

\begin{enumerate}
	\item We consider firstly for each $\ell=1\etc L$ a partition $\pi_\ell$ of the vertex set $V_\ell$ of $m_\ell$. We assume that $\pi_{\ell}$ does more identifications of outputs of $m_\ell$ than $\sigma_\ell$: for any $\ell=1\etc K$, one has $ {\pi_\ell}_{|\mcal O_\ell}\geq \sigma_\ell $.
	\item Given a collection $\Pi=(\pi_1\etc \pi_L) \in \prod_{\ell=1}^L\mcal P(V_\ell)$ of partitions as in the previous point, we consider a partition $\pi$ of $V$ with same identification as the $\pi_\ell$ for vertices of the monomials:  for any $\ell=1\etc L$, one has $ \pi_{V_\ell} \geq \pi_\ell $. We denote by $\mcal P_\Pi(V)$ the set of partitions $\pi \in \mcal P(V)$ with this condition.
\end{enumerate}
 
We then obtain as expected, using the property of the M\"obius map \cite[Sections 3.6 and 3.7]{Stan12} in the third identity,
	\eq
		\tau[h] & = &  \sum_{\substack{ \sigma_\ell\in \mcal P(\mcal O_\ell) \\ \forall \ell=1\etc K} } \Big(  \prod_{k=1}^K c_{\sigma_k}   \Big) \sum_{ \substack{ \pi_\ell \in \mcal P(V_\ell) \\ \forall \ell=1\etc L \\ \mrm{s.t. \ }\sigma_\ell \leq  {\pi_\ell}_{|\mcal O_\ell} \\ \forall \ell=1\etc K}} \sum_{  \pi \in   \mcal P_\Pi(V)} \tau^0\big[T^{  \pi} \big]\\
		& = &   \sum_{ \substack{ \pi_\ell \in \mcal P(V_\ell) \\ \forall\ell=1\etc L}} \Big( \prod_{k=1}^K \ \sum_{\substack{ \sigma_k\in \mcal P(\mcal O_\ell) \\ \mrm{s.t. \ } \sigma_k \leq  {\pi_k}_{|\mcal O_k}} }   c_{\sigma_k}   \Big) \sum_{  \pi \in \mcal P_\Pi(V)} \tau^0\big[T^{  \pi} \big]\\
		& =&  \sum_{ \substack{ \pi_\ell \in \mcal P(V_\ell) \\  {\pi_\ell}_{|\mcal O_\ell} = 0_{\mcal O_{\ell}} \\ \forall \ell=1\etc L}} \ \sum_{  \pi \in \mcal P_\Pi(V)} \tau^0\big[T^{  \pi} \big] = \sum_{\substack{\pi\in \mcal P(V) \\ \pi_{|\mcal O_\ell} = 0_{\mcal O_\ell}\\ \forall \ell=1\etc L }} \tau^0\big[T^\pi\big].
	\qe

\end{proof}

The next lemma highlights an elementary property of the graph of colored components that we will use several times. We use the following terminology.

\begin{Def}\label{Def:Valid} We say that a partition $\pi$ of the vertex set of $T \in \mcal T\langle \bigsqcup \mcal A_j \rangle$ is \emph{valid} whenever $\mcal G \mcal C \mcal C(T^\pi)$ is a tree.
\end{Def}

\begin{Lem}\label{Lem:CycleGCC} Let $J$ be an index set and, for each $j\in J$, let $\mcal A_j$ be an ensemble. Let consider the data of
\begin{itemize}
	\item a test graph $T  \in \mcal T\langle \bigsqcup_{j}\mcal A_j\rangle$ such that $\mcal G \mcal C \mcal C(T)$ is not a tree,
	\item a valid partition $\pi$ of the vertex set of $T$,
	\item  a simple cycle $\mcal C: o_1, S_1, o_2, S_2 \etc o_K, S_K$ of $\mcal G \mcal C \mcal C(T)$, $K\geq 2$, where $S_{k}$ is a colored component of $T$ attached to the connectors $o_{k}$ and $o_{k+1}$, with indices $k$ modulo $K$ (the $o_k$'s and $S_k$'s are pairwise distinct).
\end{itemize}
Then, identifying connectors $o_k$'s with their image in $T$, there exist at least two indices $k, k'\in \{1\etc K\}$ such that $o_k\sim_\pi o_{k+1}$ and $o_{k'}\sim_\pi o_{k'+1}$, with indices modulo $K$. If moreover $K\geq 4$ then there exist non consecutive such $k,k'$, i.e. one can choose $|k-k'|\geq 2$ with distance in $ {\mbb Z}/{K\mbb Z }$.
\end{Lem}

In simple words, we cannot fold a cycle into a tree of colored components without pinching at least two colored components.

\begin{proof} Given $\pi\in \mcal P(V)$, the cycle $\mcal C$ on $\mcal G \mcal C \mcal C(T)$ induces a closed path on $\mcal G \mcal C \mcal C(T^\pi)$. Since $\mcal G \mcal C \mcal C(T^\pi)$ is a tree, the closed path visits a subtree of $\mcal G \mcal C \mcal C(T^\pi)$. This subtree has at least two leaves (vertices of degree one). They do not consist in connectors, since colors are alternated along the cycle. Hence each leaf corresponds to one or several graphs $S_k$ for which we have identified $o_k$ and $o_{k+1}$. When $K\geq 4$ it is clear that we can choose separated connectors.  Hence the result.\end{proof}

 We deduce the following corollary which  implies that traces of alternated bi-graph operations in reduced elements vanish, by a simple argument of linearity that is given explicitly in next section. .

\begin{Cor}\label{Cor:TraceTree} Let $j$ be an index set and, for each $j\in J$, let $\mcal A_j$ be an ensemble. Let $\tau: \mbb C \mcal T \langle \bigsqcup_j \mcal A_j \rangle \to \mbb C$ be a unital linear form such that $\tau$ is the free product of its restrictions on test graphs labeled in $\mcal A_j$, $j\in J$.  Let $h = T_g(t_1\otimes \dots \otimes t_L)$ in $\mbb C \mcal T \langle \bigsqcup_j \mcal A_j \rangle$ where $g\in \mcal B^{(0)}_{alt}$ and $(t_1\otimes \dots \otimes t_L)$ is $g$-colored and satisfies $t_\ell=m_\ell$ for $n_\ell=1$ and $t_\ell=p(m_\ell)$ for $n_\ell\geq 2$ as in Lemma \ref{Lem:RedEffect}. Then if $g$ is not a tree, $\tau[h]=0$, and otherwise 
	$$\tau[h] = \prod_{\ell=1}^L\tau[ t_\ell],$$
where in the above formula we extend $\tau$ as a linear map $\tau: \bigoplus_{n\geq 0} \mbb C \mcal G^{(n)}\langle \mcal A \rangle$ by forgetting the position of the outputs.
\end{Cor}

The proof of the corollary can be summarised as follows.  Let $h = T_g(t_1\otimes \dots \otimes t_L)$ and $T=T_g(m_1\otimes \dots \otimes m_L)$ be as in the above corollary. By Lemma \ref{Lem:RedEffect} and since $\tau$ is the free product of its restriction on the $\mcal A_j$'s, one has $\tau[h] = \sum_\pi \tau^0[T^\pi],$ where the sum is over valid and solid partitions $\pi$. By Lemma \ref{Lem:CycleGCC} the set of such maps is empty if $g$ is not a tree. The first part of the corollary is then a direct consequence of the lemmas. It will be enough to prove the following. 
\begin{Lem}\label{Lem:Primitive} We say that a partition $\pi$ of the vertex set of $T\in \mcal T \langle \bigsqcup_j \mcal A_j, j \in J \rangle$ is \emph{primitive} whenever it satisfies one of the following equivalent properties:
\begin{enumerate}
	\item the graph of colored components is preserved after a quotient by $\pi$: $\mcal G \mcal C \mcal C(T^\pi) = \mcal G \mcal C \mcal C(T)$;
	\item for any vertices $v,w$ of $T$  belonging to different colored components  such that $v\sim_\pi w$, the components of $v$ and $w$ in $T$ have exactly one connector $o$ in common and $v\sim_\pi o \sim_\pi w$;
	\item the colored components $m_1\etc m_L$ of $T$ are solid for $\pi$ and given its restriction $\Pi=(\pi_{|V_1} \etc \pi_{|V_L})$ on the vertex sets of the $m_\ell$'s, it is the smallest partition of the set $\mcal P_\Pi(V)$ constructed in the proof of Lemma \ref{Lem:RedEffect}.
\end{enumerate}
Let $T$ be a test-graph such that $\mcal G \mcal C \mcal C(T)$ is a tree and let $\pi$ be a valid partition which is solid for the colored components of $T$. Then $\pi$ is primitive.
\end{Lem}

The lemma implies that the trace of $h$ is the sum of the injective trace of quotient graphs of $T$ by primitive partitions. Denote by $T_\ell$ the test graph of $t_\ell$ and $V_\ell$ its vertex set. By multiplicativity with respect to the colored components in the definition of traffic independence, for any $\pi$ primitive we have $\tau^0[T^{\pi}] = \prod_{\ell = 1}^L \tau^0[T_\ell^{\pi_\ell}]$ where $\pi_\ell$ is the restriction of $\pi$ to $V_\ell$. Hence $\tau[h] =  \prod_{\ell = 1}^L  \sum_{  \pi_\ell  }  \tau^0[T_\ell^{\pi_\ell}]$ where the sums are over the solid partitions $\pi_\ell$ of $V_\ell$ with respect to $m_\ell$. By Lemma \ref{Lem:RedEffect} again, the sum of quotients of $T_\ell$ by solid partitions is $\tau\big[p(T_\ell) \big].$ Hence this last lemma implies the corollary.

\begin{proof}[Proof of Lemma \ref{Lem:Primitive}]
We prove that a solid partition $\pi$ which is not primitive is not valid: that is if $\pi$ does not identify outputs of a same $m_\ell$ but identifies vertices of different colored components in a non trivial way, then $\mcal G \mcal C \mcal C(T^\pi)$ is not a tree. So let $v$ and $w$ be two vertices in different colored components and denote by $T_{v\sim w}$ the graph obtained from $T$ by identifying $v$ and $w$. Then $T^\pi$ is a quotient of $T_{v\sim w}$, and we apply Lemma \ref{Lem:CycleGCC} to the graph $T_{v\sim w}$, the induced partition partition $\pi_{v\sim w}$ such that $T^\pi = (T_{v\sim w})^{\pi_{v\sim w}}$, and the cycle coming from the path between $v$ and $w$. All colored components of this cycle are not solid, but only those that are attached to $v$ and $w$, as we explain thanks to the enumeration below. Lemma \ref{Lem:CycleGCC} tells that $\pi_{v\sim w}$ is not valid if the cycle has length at least  4, which implies that $\pi$ is not valid. The remaining case ($K=2,3$) are considered after the description of the different possibilities for $\mcal G \mcal C \mcal C (T_{v\sim w})$, see Figure \ref{Fig:05}. 

Say that a vertex of $V_\ell$ that is not an output is an internal vertex. A vertex which is not internal is associated to a connector. We decompose five alternatives:

  \begin{figure}[h!]
    \begin{center}
     \includegraphics[width=125mm]{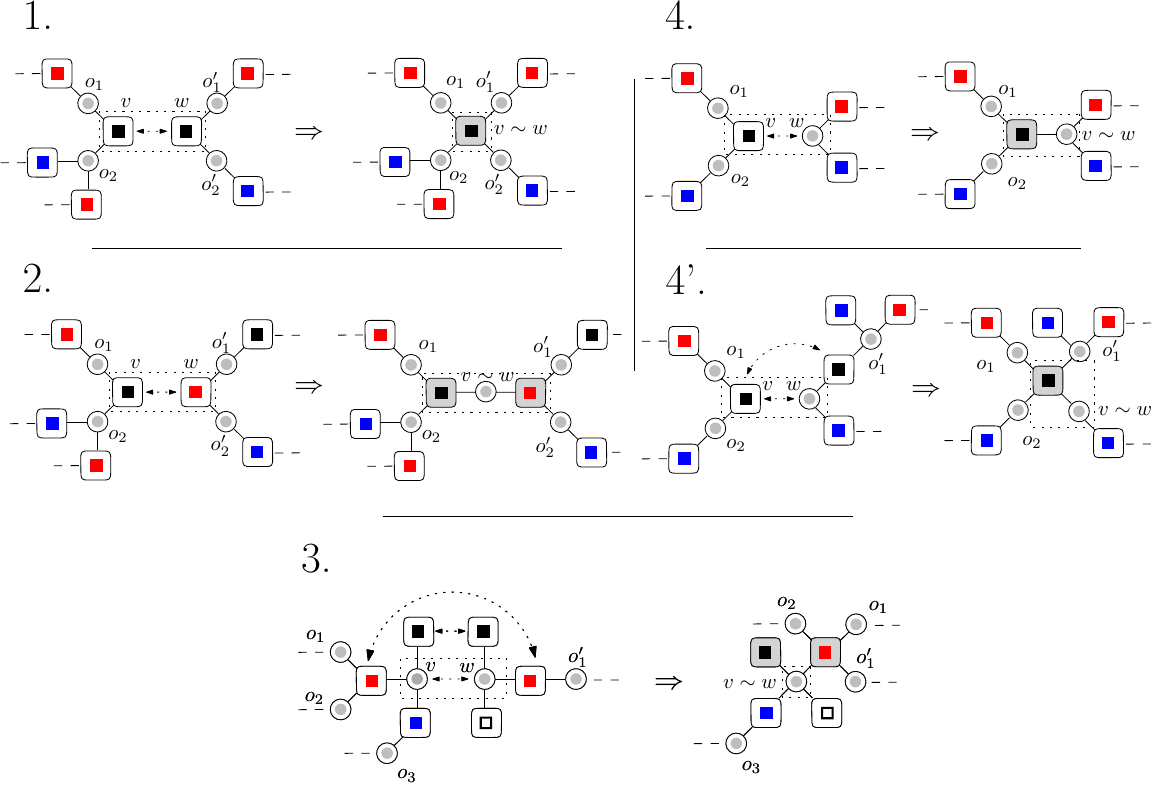}
    \end{center}
    \caption{For each of the five items of the figure, the left-most picture is a local detail of the graph $\mcal G \mcal C \mcal C(T)$. Square vertices represent inputs, circle vertices represent  connectors, and the different colors for inputs represent different labels $\mcal A_j$. They are two different parts of $\mcal G \mcal C \mcal C(T)$ (on the left and on the right) that contain respectively the vertex $v$ and $w$ (inside the dotted rectangle)  identified to give the right-most picture for each item. The right-most picture of each item is a detail of $\mcal G \mcal C \mcal C(T_{v\sim w})$. They are different cases, depending if $v$ and $w$ are input or output vertices and on the colors of the input vertices. An input vertex is in grey when it is involved in the identification (it is not a colored component of the original graph $T$).}
    \label{Fig:05}
  \end{figure}

\begin{enumerate}
	\item If $v$ and $w$ are internal vertices of components of the same color, then $\mcal G \mcal C \mcal C (T_{v\sim w})$ is obtained by identifying these components in $\mcal G \mcal C \mcal C (T)$.
	\item If $v$ and $w$ are internal vertices of components of different colors, then $\mcal G \mcal C \mcal C (T_{v\sim w})$ is obtained by creating a new connector between them in $\mcal G \mcal C \mcal C (T)$.
	\item If $v$ and $w$ are not internal vertices, then $\mcal G \mcal C \mcal C (T_{v\sim w})$ is obtained by identifying them in $\mcal G \mcal C \mcal C (T)$, then identifying the possible components of same colors attached to $v$ and $w$, and then reducing the number of edges attaching them to the connectors from two to one. In general for a partition $\pi$ of $\mcal P(V)$, there may exist a component attached both to $v$ and $w$ (which results in other operations), but this is not possible if $v$ and $w$ do not belong to a same component.
	\item If $v$ is an internal vertex and $w$ a connector that is not attached to a component of the same color as the one containing $v$, then $\mcal G \mcal C \mcal C (T_{v\sim w})$ is obtained by putting an edge between the component of $v$ and $w$ in $\mcal G \mcal C \mcal C (T)$.
	\item [4'] If $v$ is an internal vertex and $w$ a connector attached to a component of the same color as the one containing $v$, then $\mcal G \mcal C \mcal C (T_{v\sim w})$ is obtained by identifying these components in $\mcal G \mcal C \mcal C (T)$.
\end{enumerate}
Hence the path between $v$ and $w$ in $T$ induces a cycle $\mcal C: o_1,S_1\etc o_K,S_K$ on $\mcal G \mcal C \mcal C (T_{v\sim w})$. The components $S_1\etc S_K$ of $\mcal G \mcal C \mcal C (T_{v\sim w})$ are the original components of $\mcal G \mcal C \mcal C (T)$ except at most for two new components attached to a same connector (in grey in Figure \ref{Fig:05}). The partition $\pi$ cannot be valid.

 When $K=2$, then a partition of the vertices of $T_{v \sim w}$ is possibly valid only if the two connectors of the cycle are identified, and so are the two outputs of the colored components $S_1$ and $S_2$. At least one of the components is an original one, except if $v$ and $w$ are in the second situation in the above list: they are internal vertices of colored components of different colored, and a new connector appears in $T_{v\sim w}$. In that case the $\mcal G \mcal C \mcal C$ of a quotient of $T_{v\sim w}$ is a tree only if $v\sim w$ is identified with the connector between the initial components of $v$ and $w$ (this is the particular case). For $K=3$, getting a tree needs at least one identification that reduces the problem to the case $K=2$. This concludes the proof of the corollary.

\end{proof}

\subsubsection{End of the proof}

We can now achieve the proof of Theorem \ref{Equivalence Free product free independence}. To start with, we prove that the first two  properties are equivalent. Assume first that the $\mcal A_j, j \in J$, are independent and let us prove that every alternated $0$-bigraph polynomial in reduced elements is centered using the preliminary results of the section. By Lemma \ref{Lem:DSD} and Corollary \ref{Cor:Monom}, it is sufficient to consider $h = T_g(t_1\otimes \dots \otimes t_L) \in \mbb C \mcal T \langle \bigsqcup_j \mcal A_j\rangle$ where $g\in \mcal B^{(0)}_{alt}$ and $t_1\otimes \dots \otimes t_L$ $g$-colored such that $t_\ell = p(m_\ell)$ for a graph monomial $m_\ell$ for each $\ell=1\etc L$. Without loss of generality, assume that the indices $\ell$ for which $n_\ell=1$ are $1 \etc K$ for $K\leq L$. For any $\mbf i =(i_1\etc i_{K})$ in $\{0,1\}^{K}$, let $h_\mbf i$ be the graph polynomial $T_g(\tilde t_1\otimes \dots \otimes \tilde t_L)$ where 
\begin{itemize}
	\item $\tilde t_\ell = t_\ell$ if $\ell>K$,
	\item $\tilde t_\ell = m_\ell$ if $\ell\leq K$ and $i_\ell=0$,
	\item $\tilde t_\ell = -\tau[m_\ell] \times (\cdot)$ if $\ell\leq K$ and $i_\ell=1$, where $ (\cdot)$ is the $1$-graph monomial with a single vertex and no edges,
\end{itemize}
in such a way one has $h = \sum_{\mbf i \in \{0,1\}^{K}}  h_{\mbf i}.$ We apply Corollary \ref{Cor:TraceTree} to each $h_i$: one has
	$$\tau[h ] = \one\big( g \mrm{ \ is \ a \ tree}\big) \times \prod_{\ell=1}^L \tau[ p(T_\ell)],$$
where for $n_\ell=1$ we denote $p(T_\ell) = T_\ell - \tau[ T_\ell] \times (\cdot)$. Since a tree has leaves for which $\tau[ p(T_\ell)]=0$, we get $\tau[h]=0$ as expected.

Reciprocally, let $\tau$ be an unital linear form on $\mbb C \langle \bigsqcup_j \mcal A_j \rangle$. Assume it satisfies $\tau[h]=0$ for any $h$ given by an alternated bigraph operation in alternated and reduced elements. Then by Lemma \ref{Lem:DSD} and the previous paragraph, it coincides with the free product of the traffic distribution of the $\mcal A_\ell$'s on $\mbb C \mcal T\langle \bigsqcup_j \mcal A_j \rangle$. Hence the $\mcal G$-subalgebras $\mcal A_\ell$ are independent.

The second and fourth items (the same property for $2$-bigraph polynomials and w.r.t. the anti-trace $\Psi$) are equivalent since an element $h=T_g(t_1\otimes \dots \otimes t_L)$ of $\mbb C \mcal G^{(2)}\langle \mcal A\rangle$ is an alternated bigraph operation in reduced elements if and only if the element $\tilde h = T_{\tilde g}(t_1\otimes \dots \otimes t_L)$ of $\mbb C \mcal T\langle \mcal A\rangle$ is as well, where in $\tilde g$ we forget the position of the input and output. We recall that $\Psi(h) = \tau(\tilde h)$ by definition of $\Psi$.

The third item (the property for $2$-bigraph polynomials and w.r.t. the trace $\Phi$) implies the second one since if an element $h=T_g(t_1\otimes \dots \otimes t_L)$ of $\mbb C \mcal T\langle \mcal A\rangle$ is an alternated bigraph operation in reduced elements, then so is the element $\tilde h=T_{\tilde g}(t_1\otimes \dots \otimes t_L)$ of $\mbb C \mcal G^{(2)}\langle \mcal A\rangle$ obtained by declaring that a vertex is both the input and the outputs.

Assume now that the second item is satisfied and let us prove the third one. There we use again an argument of the previous section. Let $h = T_g(\mbf t) $ in $\mbb C \mcal G^{(2)}\langle \mcal A \rangle$ where $\mbf t$ is $g$-alternated reduced and given by monomials $t_\ell = p(m_\ell)$ as usual. If the two outputs $v$ and $w$ of $g$ are equal, then $g = \Delta(g)$ so $\Phi(h)=0$. Assume the outputs are distinct, so that $\Delta(g)$ is possibly not alternated at the position where $w$ and $v$ are identified (Figure \ref{Fig:05}). We apply Lemma \ref{Lem:CycleGCC} to the graph $T_{v\sim w}$, any partition, and a cycle given by a path between $v$ and $w$ in $T$. As in the proof of Lemma \ref{Lem:Primitive}, we get $\Phi(h)=0$.

\section{Products of traffic spaces}\label{Sec:PdtTrafSp}

This section is mainly devoted to the construction of the free product of traffic spaces, in particular under the context where we assume a positivity condition for the combinatorial trace. In the last subsection we also consider the tensor product of traffic spaces which will be used a couple of times in Part II. 

\subsection{The free product of algebraic traffic spaces}\label{Sec:FreeProdAlg}

Let us first consider an arbitrary ensemble $X$. The free $\mcal G$-algebra generated by $X$ is the space $\mbb C \mcal G\langle X \rangle$ generated by graph monomials whose edges are labeled by elements of $X$ . It is endowed with the natural structure of $\mcal G$-algebra given by the composition maps of the operad $\mcal G$ (Section \ref{Sec:AlgOp}): for any graph operation $g\in \mcal G_K$ and any graph polynomials $g_1\etc g_K\in \mcal G\langle X \rangle$ labeled in $X$,
	$$Z_g(g_1\otimes \dots \otimes g_K)=g(g_1\etc g_K),$$
where in the right hand side we identify the graph operation $g\in \mcal G_K$ with the associated graph monomial in $K$ variables. Hence $\mcal G$ is well a $\mcal G$-algebra. 

Let $\tau: \mbb C \mcal T\langle X \rangle \to \mbb C$ be an arbitrary linear map, unital in the sense that $\tau[ (\cdot)]=1$. Then it always induces a structure of algebraic traffic space on $\mbb C\mcal G\langle X \rangle$. To explain this fact, we first define a combinatorial trace $\tilde \tau: \mbb C \mcal T\big\langle \mbb C \mcal G\langle X \rangle \big\rangle \to \mbb C$ as follow. For any test graph $T$ labeled in $\mcal G\langle X \rangle $ with $K\geq 1$ edges denoted $e_1\etc e_K$ and labeled respectively by monomials $g_1 \etc g_K$ in $\mcal G\langle X \rangle$, we set $\tilde \tau[ T] = \tau[ T_{\mbf g}]$, where $T_{\mbf g}$ is the graph labeled in $X$ obtained from $T$ by replacing the edge $e_k$ by the graph $g_k$ for any $k=1\etc K$. Then we extend $\tilde \tau$ by multi-linearity with respect to the edges and set $\tilde \tau[ (\cdot)]=1$. 

\begin{Lem}\label{Lem:AssProdEns} The map $\tilde \tau: \mbb C \mcal T\big\langle \mbb C \mcal G\langle X \rangle \big\rangle \to \mbb C$ satisfies the associativity property, and so endows $\mbb C \mcal G \langle X \rangle$ with a structure of algebraic traffic space.
\end{Lem}

\begin{proof} Let $T \in  \mcal T\langle X \rangle$ whose edges are denoted $e_1\etc e_n$, where $e_1$ has label $Z_h\big( g_1 \otimes \dots \otimes g_K) = h(g_1\etc g_K)$ and $e_i$, $i\geq 2$, has label $g_{K+i-1}$ for graph monomials $g_1\etc g_{K+n-1}$ labeled in $X$. We have by definition $\tilde \tau[T] = \tau[\tilde T]$ where $\tilde T$ is the graph labeled in $X$ obtained by replacing $e_1$ by $h(g_1\etc g_K)$ and $e_i$, $i\geq 2$, by $g_{K+i-1}$. But we have $\tau[\tilde T]= \tilde \tau[T_h]$, where $T_h$ is the graph labeled in $\mcal G \langle X \rangle$ obtained by replacing in $T$ the edge $e_1$ by $h$. This implies the associativity property $\tilde \tau [T] = \tilde \tau[T_h]$.
\end{proof}

Let now $J$ be a labeling set and for each $j\in J$ let $X_j$ be an ensemble. Recall that we denote by $\bigsqcup_{j\in J}X_j$ the set of couples $(j,x)$ where $j\in J$ and $x\in X_j$. Assume that for each $j\in J$ we are given a unital linear map $\tau_j: \mbb C \mcal T\langle X_j \rangle \to \mbb C$, and denote by $\tau:\mbb C \mcal T\langle \bigsqcup_{j\in J}X_j \rangle \to \mbb C$ the free product of the $\tau_j$, $j\in J$. Denote by $\tilde \tau$ the combinatorial trace on $\mbb C \mcal T\big\langle \mbb C \mcal G\langle \bigsqcup_{j\in J}X_j \rangle\big\rangle \to \mbb C$ induced by $\tau$ and by $\tilde \tau_j$ the restrictions of $\tilde \tau$ to the subspaces $\mbb C \mcal T\big\langle \mbb C \mcal G\langle X_j \rangle\big\rangle$ generated by test graphs whose labels are graphs labeled in $X_j$, $j\in J$.

\begin{Lem}\label{Lem:jesaisplus} The map $\tilde \tau$ is the free product of the $\tilde \tau_j$'s, $j\in J$. Hence the $\mcal G$-subalgebras $\mbb C \mcal G\langle X_j \rangle, j\in J$ are traffic independent in $(\mbb C \mcal G\langle \bigsqcup_j X_j \rangle, \star_j \tau_j)$.
\end{Lem}

This fact is proved in \cite[Proposition 2.14]{Male2011}, based only on the definition of traffic independence in terms of the injective trace. The proof of Theorem \ref{Equivalence Free product free independence} is somehow a strengthening of this proof, and now the lemma is actually a direct consequence of the new characterization of traffic independence.
\begin{proof} 
Let $h = T_g(t_1 \otimes \dots \otimes t_L)$ be an alternated bigraph operation in reduced elements labeled in $\bigsqcup_j \mbb C \mcal G \langle X_j \rangle$ and let us prove that $\tilde \tau[h]=0$. Let $\tilde t_1\otimes \dots \otimes \tilde t_L$ be the tensor product of elements labeled in $\bigsqcup_j X_j $ obtained as follow: for each graph $t_\ell$, we replace each edge by the linear combination of the graphs that appear on their labels. By definition of $\tilde \tau$, we have $\tilde \tau[h] = \tau [ \tilde h ]$ where $\tilde h=T_g(\tilde t_1 \otimes \dots \otimes \tilde t_L) $. Moreover, $\tilde h$ is still an alternated bigraph operation in reduced elements. By Corollary \ref{Cor:TraceTree}, we hence get $\tilde \tau[h]=0$.
\end{proof}

We can now define the free product of $\mcal G$-algebras. The map $g\mapsto Z_g$ is extended for $g$ by linearity for linear combinations of graph operations. 

\begin{Def}\label{Def:QuotProd}
For any family of $\mcal G$-algebras $(\mathcal{A}_j)_{j\in J}$, we denote by $\ast_{j\in J} \mathcal{A}_j$ the vector space $\mathbb{C}\mathcal{G}\langle \bigsqcup_{j\in J}\mathcal{A}_j \rangle$, quotiented by the following relations: for any $i\in J$, any $a_1\etc a_k  \in \mcal A_i$, $a_{k+1} \etc a_n \in \bigcup_{j\in J} \mcal A_j$, any $g$ in $\mcal G_{n-k+1}$ and any linear combination of graph operations $h$ in $ \mcal G_k$,
	\eq
		\lefteqn{Z_g(\cdot \overset{Z_{h}(a_1 \otimes  \cdots \otimes  a_{k})}{\longleftarrow} \cdot \otimes  \cdot \overset{a_{k+1}}{\leftarrow} \cdot \otimes \dots \otimes  \cdot\overset{a_n}{\leftarrow} \cdot)}\\
		& \sim & Z_g (Z_{h}(\cdot \overset{a_{1}}{\leftarrow} \cdot \otimes  \cdots \otimes  \cdot \overset{a_{k}}{\leftarrow} \cdot)\otimes  \cdot \overset{a_{k+1}}{\leftarrow} \cdot\otimes \dots\otimes  \cdot\overset{a_n}{\leftarrow} \cdot).
	\qe
	
\end{Def}

This relation implies that, for $a_1\etc a_K$ in a same algebra $\mcal A_i$,
	$$ Z_h(\cdot \overset{ a_1 } \leftarrow \cdot \otimes \dots \otimes \cdot \overset{ a_K} \leftarrow  \cdot) \sim ( \cdot \overset{ Z_h(a_1 \otimes \dots \otimes a_K)} \longleftarrow \cdot),$$
and in particular, an edge labeled by the unit $(\cdot \overset{ 1_{\mcal A}}\leftarrow \cdot)$ is equal to the graph with no edge $(\, \cdot\, )$. The other relations involving several algebras make the $\mcal G$-algebra structure of $\mathbb{C}\mathcal{G}\langle \bigsqcup_{j\in J}\mathcal{A}_j \rangle$ compatible with this quotient (similar to the proof of Lemma \ref{Lem:AssProdEns}). This allows to consider the $\mcal G$-algebra homomorphisms $V_j:\mathcal{A}_j \to \ast_{j\in J} \mathcal{A}_j$ given by the image of $a\mapsto (\cdot\overset{a}{\leftarrow} \cdot)$ by the quotient map.

The $\mcal G$-algebra $\ast_{j\in J} \mathcal{A}_j$ is the free product of the $\mcal G$-algebras in the following sense.
\begin{Prop}Let $\mathcal{B}$ be a $\mcal G$-algebra, and $f_j:\mathcal{A}_j\to \mathcal{B}$ a family of $\mcal G$-morphisms. There exists a unique $\mcal G$-morphism $\ast_{j\in J} f_j:\ast_{j\in J} \mathcal{A}_j\to  \mathcal{B}$ such that $f_j=(\ast_{j\in J} f_j)\circ V_j$ for all $j\in J$. As a consequence, the maps $V_j$ are injective.
\end{Prop}

\begin{proof}The existence is given by the following definition of $\ast_{j\in J} f_j$ on $\mathbb{C}\mathcal{G}^{(2)}\langle \bigsqcup_{j\in J}\mathcal{A}_j \rangle$:
$$\ast_{j\in J} f_j\big(g(\cdot \overset{a_1}{\leftarrow} \cdot, \ldots,  \cdot\overset{a_n}{\leftarrow} \cdot) \big)=g\big(f_{j(1)}(a_1), \ldots , f_{j(n)}(a_n)\big)$$
whenever $a_1\in \mathcal{A}_{j(1)},\ldots,a_n\in \mathcal{A}_{j(n)}$. It obviously respects the relation defining $\ast_{j\in J} \mathcal{A}_j$.

The uniqueness follows from the fact that $\ast_{j\in J} f_j$ is uniquely determined on $\bigcup_j V_j(\mathcal{A}_j)$ (indeed, $\ast_{j\in J} f_j(a)$ must be equal to $f_j(b)$ whenever $a=V_j(b)$) and that $\bigcup_j V_j(\mathcal{A}_j)$ generates $\ast_{j\in J} \mathcal{A}_j$ as a $\mcal G$-algebra.
\end{proof}

We now construct the free product of algebraic traffic spaces.

\begin{Prop}\label{Prop:DefiningTrafficSpace} Let $(\mcal A_j,\tau_j)_{j\in J}$ be a family of algebraic traffic spaces.  Let $\tilde \tau:\mathbb{C}\mathcal T\big\langle \mbb C \mcal G \langle \bigsqcup_{j\in J}\mathcal{A}_j \rangle \big\rangle\to \C$ be the unital linear map induced by $*_{j\in J}\tau_j:\mathbb{C}\mathcal T\langle \bigsqcup_{j\in J}\mathcal{A}_j \rangle\to \C$ as in the first paragraph of the section. Then $\tilde \tau$ respects the quotient structure of $\ast_{j\in J}\mcal A_j$. Still denoting the quotient map $\star_{j\in J}\tau_j: \mbb C \mcal T\langle \ast_{j\in J}\mcal A_j \rangle\to \mbb C$, we then get an algebraic traffic space $(\ast_{j\in J}\mcal A_j , \star_{j\in J}\tau_j)$ called the free product of the algebraic traffic spaces. Furthermore, we have $\tau_i=(\star_{j\in J}\tau_j)\circ V_i$, where $V_i$ is the canonical injective algebra homomorphism from $\mathcal{A}_i$ to $\ast_{j\in J}\mcal A_j$, and the $\mcal A_i$, $i\in J$, are traffic independent in $(\ast_{j\in J}\mcal A_j , \star_{j\in J}\tau_j)$.
\end{Prop}

\begin{proof}
Let $T \in \mathbb{C}\mathcal T\big\langle \mbb C \mcal G \langle \bigsqcup_{j\in J}\mathcal{A}_j \rangle \big\rangle$ such that an edge $e$ has label $(\cdot \overset{Z_{h}(a_1 \otimes  \cdots \otimes  a_{k})}{\longleftarrow} \cdot ) $, where $h = \sum_i a_i h_i$ is linear combination of  graph operations labeled in a same $\mcal A_i$. It suffices to prove that $\tilde \tau[T] = \tilde \tau[T_h]$, where $T_h = \sum_i a_i T_{h_i}$ with $T_{h_i}$ the graph obtained by replacing $e_1$ by the graph $h_i$ evaluated in $(\cdot \overset{a_{1}}\leftarrow \cdot \etc \cdot \overset{a_{k}}\leftarrow \cdot)$. But when decomposing $T$ and $T_h$ on $\mathbb{C}\mathcal T\big\langle \mbb C \mcal G \langle \bigsqcup_{j\in J}\mathcal{A}_j \rangle \big\rangle$ according to the direct sum of Lemma \ref{Lem:DSD}, we get the same coefficient on $\mbb C (\cdot)$. Since the $\mcal G$-subalgebra are independent, $\tilde \tau[T]$ and $\tilde \tau[T_h]$ are equal to these constants and so they are equal. Hence $\tilde \tau$ respects the quotient structure defining $\ast_{j\in J}\mcal A_j$. 
\end{proof}

\subsection{Definition of positivity and traffic spaces}\label{Sec:DefPos}

We first define an analogue of $^*$-algebras. On the set of graph operations $\mcal G$, we define an \emph{involution} $t:g\to g^t$ , where $g^t$ is obtained from $g$ by reversing the orientation of its edges and interchanging the input and the output.

\begin{Def}\label{Def:G*Alg} A $\mcal G^*$-algebra is a $\mcal G$-algebra $\mcal A$ endowed with an anti-linear involution $*:\mathcal{A}\to \mathcal{A}$ which is compatible with the action of $\mcal G$, in the following sense: for all $K$-graph operation $g$ and $a_1,\ldots,a_K\in \mcal A$, $\big(Z_g(a_1\otimes \ldots\otimes a_K)\big)^*=Z_{g^t}(a_1^*\otimes \ldots\otimes a_K^*)$. A $\mcal G^*$-subalgebra is a $\mcal G$-subalgebra closed by adjoint. A $\mcal G^*$-morphism between $\mcal A$ and $\mcal B$ is a $\mcal G$-morphism $f:\mcal A \to \mcal B$ such that $f(a^*)=f(a)^*$ for any $a\in \mcal A$.

\end{Def}

Recall that for any $n\geq 1$, a $n$-graph monomial is a test graph with the data of a $n$-tuple of vertices. Let $g,g'$ be two $n$-graph monomials labeled in some set $\mathcal{A}$. We set $g|g'$ the test graph obtained by merging the $i$-th output of $g$ and $g'$ for any $i=1 \etc n$. We extend the map $(g, g') \mapsto t|t'$ to a bilinear application $\mbb C \mcal G^{(n)}\langle \mathcal{A} \rangle^2 \to \mbb C \mcal T\langle \mathcal{A} \rangle$. Note that one can also realize $g|g'$ as a bigraph operation evaluated in $g\otimes g'$, see Figure \ref{Fig:18}.

  \begin{figure}
    \begin{center}
     \includegraphics{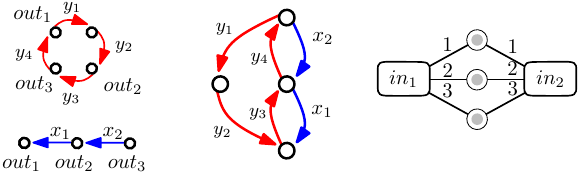}
    \end{center}
    \caption{Left: two 3-graph monomials $t$ and $t'$. Middle: the test graph $t'|t'$. Right: the bi-graph operation $g$ such that $t'|t' = T_g(t\otimes t')$.}
    \label{Fig:18}
  \end{figure}

Assume moreover that $\mcal A$ is endowed with an anti-linear involution $*:\mathcal{A}\to \mathcal{A}$. Given an $n$-graph monomial $g = (V,E,\gamma, \mbf v)$ we set $g^\dagger=(V,E^\dagger,\gamma^\dagger,\mbf v)$, where $E^\dagger$ is obtained by reversing the orientation of the edges in $E$ and with $\gamma^\dagger$ given by $e\mapsto \gamma(e)^*$. Note that for $n=2$, $g^\dagger(a_1\otimes \cdots \otimes a_n)\neq g^t(a_1^*\otimes \cdots \otimes a_n^*)$ since there is no inversion of the two outputs in the definition of $g^\dagger$ as in Definition \ref{Def:G*Alg}. We extend the map $g \mapsto g^\dagger$ to an anti-linear map on $\mbb C \mcal G^{(n)}\langle \mathcal{A} \rangle$.

\begin{Def}   

\label{Def:Positivity}A \emph{traffic space} is an algebraic traffic space $(\mcal A,\tau)$ such that:
\begin{itemize}
	\item $\mcal A$ is a $\mcal G^*$-algebra, 
	\item the combinatorial trace on $\mcal A$ satisfies the following \emph{positivity condition} :
for any $n\geq 1$ and any $n$-graph polynomials $g$ labeled in $\mcal A$, 
	\eqa\label{eq:NonNegCond}
		 \tau\big[ g|g^\dagger \big] \geq 0.
	\qea
	We call $\tau$ a combinatorial state.
\end{itemize}
A homomorphism between two traffic spaces is a $\mcal G^*$-morphism which is a homomorphism of algebraic traffic space. 
\end{Def}
Note that \eqref{eq:NonNegCond} for $n=2$ is equivalent to the positivity of the trace $\Phi$ induced by $\tau$ on the $*$-algebra $\mcal A$. Moreover, \eqref{eq:NonNegCond} for $n=1$ implies the positivity of the anti-trace $\Psi$ (Definition \ref{Def:TraceAndAntiTr}): indeed we have $\Psi[aa^*] = \tau\big[ g|g^\dagger \big]$ where $g$ is the 1-graph monomial with one simple edge whose source is the output.

 By consequence, every traffic space $(\mcal A, \tau)$ have two structures of $^*$-probability space $(\mathcal{A}, \Phi)$ and  $(\mathcal{A}, \Psi)$  (endowed with the product  $Z_{\cdot \overset{1}{\leftarrow} \cdot \overset{2}{\leftarrow} \cdot}$). Positivity of $\tau$ implies the Cauchy Schwarz inequality $|\tau[t_1|t_2]| \leq \sqrt{ \tau[t_1|t_1^\dagger] \tau[t_2 | t_2^\dagger]}$.

\begin{Ex}(Example \ref{Ex:MtSpa} continued)
The algebraic traffic space of random matrices is actually a traffic space since $\tau_N$ is positive. Indeed, recall that in Remark \ref{Rk:BigraphOp} for any $n$-graph monomial $t$ labeled in $J$ and a family $\mbf A_N=(A_j)_{j\in J}$ we have defined a random tensor matrix $t(\mbf A_N) \in (\mbb C^N)^{\otimes n}$. The positivity is clear since one has
	\eq
		 \tau_N \big[ (t| t^\dagger)(\mbf A_N)  \big] :=\esp \Big[ \frac 1 N \sum_{\mbf i\in [N]^n} t(\mbf A_N)_{\mbf i}\, \overline{t(\mbf A_N)}_{\mbf i}\Big]\geq 0
	\qe
\end{Ex}
\begin{Ex}\label{Ex:Graphs2}(Example \ref{Ex:MtSpaGraph} continued) The algebraic traffic space of a unimodular random graph is also a traffic space. As in the previous example, for a test graph $t$ and a family $\mbf A$ of infinite matrices we define an infinite tensor matrix $t(\mbf A)$ as in Remark \ref{Rk:BigraphOp1} but with summation over $k:V \to \mcal V$ with $k(r)=\rho$, for an arbitrary vertex $r$ of $V$ and with $(\xi_i)_{i\in \mcal V}$ the canonical basis of $\mbb C^\mcal V$. The positivity of $\tau$ follows as well since 
	$$\tau_\rho\big[ t(\mbf A)| t(\mbf A)^\dagger \big] :=\esp \big[  \sum_{\mbf i\in \mcal V^n:i_1=\rho} t(\mbf A)_{\mbf i}\, \overline{t(\mbf A)}_{\mbf i}\big]\geq 0.$$
\end{Ex}

We see now a consequence of the positivity, which will be an additional motivation for Part II. Let $(\mcal A, \tau)$ be a traffic space and let $a_1\etc a_n\in \mcal A$ such that $\Phi(a_1\dots a_n) \neq 0$. Denote by $T_1$ the oriented simple cycle with $n$ edges labeled $\cdots \overset{a_i}\leftarrow \cdot  \overset{a_{i+1}}\leftarrow \cdots  $ along the cycle. Let $t_1$ be a $1$-graph monomial with test graph $T_1$ and whose output is an arbitrary vertex. With $(\cdot)$ denoting the 1-graph monomial with no edge, we have 
	$$\Phi(a_1 \dots a_n) =\tau[T_1] = \tau[ t_1 | (\cdot) ]\neq 0.$$
Then, since $\tau$ is positive, the Cauchy-Schwarz inequality gives 
	$$\Big|\tau\big[ t_1 | (\cdot) \big]\Big|^2 \leq \tau[ t_1 | t_1^\dagger] \times \tau\big[ (\cdot) | (\cdot)\big] = \tau[ t_1 | t_1^\dagger].$$
Hence the test graph $T_2=t_1|t_1^\dagger$ satisfies $\tau[T_2]\neq 0$. It consists in two simple cycles that share exactly one vertex. We iterate, assuming we have a test graph $T_n$ such that $\tau[T_n]\neq 0$. Let $t_n$ be a $1$-graph monomial with test graph $T_n$ and output an arbitrary vertex. Then $T_{n+1} = t_n | t_n^\dagger$ satisfies $\tau[T_{n+1}]\neq 0$. We have proved the following.

\begin{Lem} Let $(\mcal A, \tau)$ be a traffic space such that $\tau$ is not constant to zero. Then $\tau$ is nonzero on an infinite number of \emph{cacti}, that are test graphs such that each edge belong to a unique cycle (see Part II).
\end{Lem}

In the second part of the monograph, given a non-commutative probability space $(\mcal A, \Phi)$ we construct a traffic space $(\mcal B, \tau)$ such that $\mcal B$ contains $\mcal A$ and the trace associated to $\tau$ and restricted on $\mcal A$ is $\Phi$. The lemma shows that the naive answer for this question, 
\begin{itemize}
	\item  $\tau[ T] = \Phi(a_1 \dots a_n)$ if $T$ is an oriented simple cycle with consecutive edges $a_1\etc a_n$,
	\item  $\tau[T]=1$ for the test graph with no edge,
	\item and $\tau[T]=0$ otherwise,
\end{itemize}
does not yield a positive combinatorial trace. There are no matrices converging to a traffic with such a simple distribution.

\subsection{Positivity of the free product}\label{Sec:PosFreeProd}

For each $j\in J,$ let $(\mcal A_j, \tau_j)$ be a traffic space. By Section \ref{Sec:FreeProdAlg}, we can consider the algebraic traffic space $(\ast_{j\in J}\mcal A_j , \star_{j\in J}\tau_j)$, the free product of the $(\mcal A_j, \tau_j)$'s. We shall now prove that $\tau:=\star_{j\in J}\tau_j$ satisfies the positivity condition \eqref{eq:NonNegCond}. Therefore, we give in Lemma \ref{Lem:Structure} a structural result for the canonical space $\mbb C \mcal G^{(n)}\langle \bigsqcup_{j\in J} \mcal A_j \rangle $, introduced in Definition \ref 
{Def:nGraphPol}. The ideas of the current section are inspired by  the counterpart of this construction for  the free product of unital algebras with identification of units (see \cite[Chapter 6]{NS} and \cite[Formula (6.2)]{NS}). The proofs build on  the preliminary 
material  presented in Section \ref{seq:Equivalence Free product free independence}.

\begin{Def} Let us consider for $n\geq 1$ a colored bigraph operation $g\in \mcal B_{col}^{(n)}$ (Definition \ref{Def:ColBigraph}). A bijection  of  the vertex set of $g$  is called an automorphism of $g$ if  it  preserves the adjacency, the bipartition, the ordered set of outputs and the coloring of $g$. Their set forms a group denoted  $Aut_{g}$ that acts on $ \mcal B_{col}^{(n)}$ and on the subspace ${\mcal B}^{(n)}_{alt}$ of  alternated colored bigraph operations with $n$ outputs.  The quotient space is denoted by  $\tilde{\mcal B}^{(n)}_{col}$ (resp. $\tilde{\mcal B}^{(n)}_{alt}$) and the equivalent class of  a  colored bigraph operations $g\in \mcal B_{col}^{(n)}$ is denoted by  $\bar g.$ 
\end{Def}
See figure \ref{Fig:09} for an example. Note that an automorphism does not necessarily respect the ordering of the inputs nor the ordering of the neighbor connectors. 
 
  Every $\sigma \in Aut_{g}$ and every $g$-alternated tensor product 
  $\mbf m=(m_1\otimes \dots \otimes m_L)$ of graph monomials induces a new $g$-alternated tensor product $\mbf m_\sigma=(m_{1,\sigma}\otimes\cdots \otimes m_{L,\sigma})$, such that $T_{\sigma}(\mbf m)=T_{g}(\mbf m_\sigma)$ by reordering the labels of the inputs and of neighbor connectors as follow (see Figures \ref{Fig:09} and \ref{Fig:10}):
  \begin{itemize}
  	\item if $\ell_v$ denotes the order of the input vertex $v$ of $g$, then $m_{\ell_v,\sigma} = m_{\ell_{\sigma^{-1}(v)}}$,
	\item the order of neighbor connectors of an input of $m_{\ell, \sigma}$ is the order of its pre-image by $\sigma$.
\end{itemize}
  
    We extend this definition by linearity for graph polynomials. Note that we have the property $(\mbf t_{\sigma_1})_{\sigma_2}=\mbf t_{\sigma_2 \sigma_1}$ for all $\sigma_1, \sigma_2 \in Aut_{g}$. For every alternated bigraph operation $g$, the space $W_{\bar g}$ spanned by $T_g(\mbf t)$ for $\mbf t$ reduced and $g$-colored does not depends on $g$ but only on the class $\bar g \in \tilde {\mcal B}_{alt}^{(n)}$.

  \begin{figure}
    \begin{center}
     \includegraphics{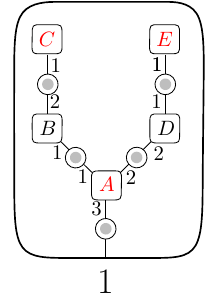}
    \end{center}
    \caption{The above colored bigraph operation $(g,\gamma)$ has a single non trivial automorphism $\sigma$, corresponding to vertical mirror symmetry.}
    \label{Fig:09}
  \end{figure}

  \begin{figure}
    \begin{center}
     \includegraphics{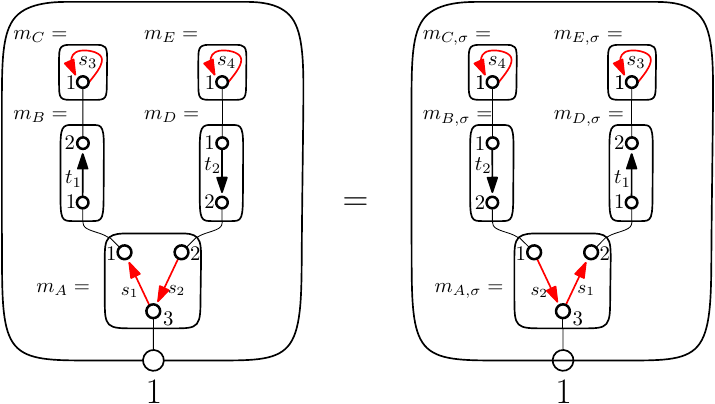}
    \end{center}
    \caption{Illustration of the equality $T_g(\mbf m) = T_g(\mbf m_\sigma)$. With $g$ the colored bigraph operation of Figure 7 and $(m_A \otimes \dots \otimes  m_E)$ a $g$-colored tensor product of graph monomials, then $(m_A \otimes \dots \otimes  m_E)_\sigma$ is obtained by exchanging $m_B$ and $m_D$, $m_C$ and $m_E$, and by permuting the outputs 1 and 2 of $m_A,m_B$ and $m_D$.}
    \label{Fig:10}
  \end{figure}

\begin{Lem}\label{Lem:Structure} Let $(\mcal A, \tau)$ be a traffic space and $\mcal A_j$, $j\in J$, be independent $\mcal G$-subalgebras.
\begin{enumerate}
	\item When considering the  non-negative Hermitian form $ \tau\big[ \, \cdot \, |\, \cdot \, ^\dagger \big]$ defined in  \eqref{eq:NonNegCond}, the space of graph-polynomials admits the orthogonal decomposition
		\eq
		\mbb C \mcal G^{(n)}\langle \bigsqcup_{j\in J} \mcal A_j \rangle 
		 & = & \mbb C \, (\cdot)  \oplus^\perp \bigoplus^{\ \ \ \perp}_{\substack{\overline g\in \tilde{\mcal B}^{(n)}_{alt}   }} \mcal W_{g}.
	\qe

	\item If $g$ is not a tree, then $W_{\bar g}$ is included in the kernel of $ \tau\big[ \, \cdot \, |\, \cdot \, ^\dagger \big]$, that is for any $h\in  W_{\bar g}$ and $h'\in \mbb C \mcal G^{(n)}\langle \bigsqcup_{j\in J} \mcal A_j \rangle, $  $\tau[h|h']=0$. 
	
	\item 
	If $g$ is a tree, then for any $h= T_g(t_1\otimes \dots \otimes t_L)$, 
	$h'= T_g(t'_1\otimes \dots \otimes t'_L)$ in $W_{\bar g}$, we have 
$$\tau[h|h']= \sum_{\sigma\in Aut_{g}}\tau[t_1|t'_{1,\sigma}] \cdots \tau[t_L|t'_{L,\sigma}].$$
\end{enumerate}
\end{Lem}

\begin{Ex} With $g$ consisting in a single path between two outputs, the only automorphism of $g$ is the identity, and we then get the following formula: for any $L,L'\geq 2$, any $j_1\neq  j_2 \neq \dots \neq j_L$ and $j'_1\neq  j'_2 \neq \dots \neq j'_{L'}$ in $J$, and any $a_{j_\ell}\in \mcal A_{j_\ell}$, $a'_{j'_\ell}\in \mcal A_{j'_{\ell'}}$, $\ell=1\etc L$, $\ell'=1\etc L'$, one has
	\eq
		\lefteqn{\Phi\Big( \big( a_{j_1} - \Delta(a_{j_1}) \big) \dots \big( a_{j_L} - \Delta(a_{j_L}) \big) \times  \big( a'_{j'_{L'}} - \Delta(a'_{j'_{L'}}) \dots \big( a'_{j'_1} - \Delta(a'_{j'_1}) \big) \Big) =}\\
		& = & \one(L=L', j_\ell = j'_\ell \, \forall \ell=1\etc L) \prod_{\ell=1}^L \Phi\Big(  \big( a_{j_\ell} - \Delta(a_{j_\ell}) \big)  \times  \big( a'_{j_\ell} - \Delta(a'_{j_\ell}) \big)  \Big).
	\qe
With $g$ the colored bigraph operation of Figure \ref{Fig:09}, the automorphisms of $g$ are the identity and the vertical mirror symmetry: hence for any $h=g\big(t_A \otimes \dots \otimes t_E\big)$ where $(t_1 \otimes \dots \otimes t_E)$ reduced and $g$-colored, one has 
	$$\tau[ h | h] = \tau\big[t_A| t_A] \dots \tau[t_E| t_E] +  \tau[t_A| \tilde t_A]  \tau[t_B| \tilde t_D] \tau[t_D| \tilde t_B]\tau[t_C|   t_E]^2 ,$$
	where $\tilde t_X$ is obtained from $t_X$ by permuting outputs 1 and 2 for $X\in \{A,B,D\}$.
\end{Ex}

\begin{proof}[Proof of Theorem \ref{Th:PosFreeProd}] Assuming  Lemma \ref{Lem:Structure} for now, let us deduce Theorem \ref{Th:PosFreeProd}.  By Corollary \ref{Cor:Monom}, it suffices to prove that $\tau[ h|h^\dagger]\geq0$ for each finite combination $h = \sum_{i} \beta_i T_{g_i}(\mbf t^i)$ for 
bigraph operations $g_i$ and tensor products of reduced polynomials $\mbf t^i = t_1^i\otimes \dots \otimes t_{L_i}^i,$ where $t_\ell^i = p(m_\ell^i)$ with a graph monomial $m_\ell^i$. Moreover the previous lemma allows to restrict 
our consideration to the case where all $g_i$ are in the equivalent class of one particular colored tree $g$ and the color of $m_\ell^i$ depends only on $\ell$, not on $i$.

   In particular, the automorphism group of colored graph $Aut_{g_i}$ is equal to $Aut_g$ for any $i$.  With this notation at hand, we can write
\eq
	 \tau\Big[h |h^\dagger\Big] & = & 
	 \sum_{ij}\beta_i\bar{\beta}_j\tau\big[T_{g}(\mbf t^i)\big|T_{g}({\mbf t^j}^\dagger)\big]\\
&=&
	\frac{1}{\sharp Aut_{g }}\sum_{\substack{ i,j \\ \sigma\in Aut_{g}}}\beta_i\bar{\beta}_j \tau\big[T_{g}(\mbf t_\sigma^i)\big|T_{g}({\mbf t_\sigma^j}^\dagger)\big],\\
&=&
	\frac{1}{\sharp Aut_{g }}\sum_{\substack{ i,j \\ \sigma,\sigma'\in Aut_{g }}}\beta_i\bar{\beta}_j   \tau    \big[t_{1,\sigma}^i  \big|{t^j_{1,\sigma'\circ \sigma}}^\dagger\big] \cdots \tau\big[ t_{L,\sigma}^i  \big|  {t^j_{L,\sigma'\circ 
\sigma}}^\dagger\big],\\
&=&
	\frac{1}{\sharp Aut_{g }}\sum_{\substack{ i,j \\ \sigma,\sigma'\in Aut_{g }}}\beta_i\bar{\beta}_j\tau \big[t_{1,\sigma}^i  \big|  {t^j_{1,\sigma'}}^\dagger\big] \cdots \tau\big[ t_{L,\sigma}^i \big| {t^j_{L,\sigma'}}^\dagger\big].\\
\qe
We shall now see that the r.h.s. is nonnegative. First, for any $\ell=1\etc L$, the matrices $\big(\tau[t_{\ell,\sigma}^i|{t^j_{\ell,\sigma'}}^\dagger]\big)_{(i,\sigma),(j,\sigma')}$  are non-negative  since $\tau$ is non-negative on each $\mcal G$-
subalgebra $\mcal A_j$. Moreover, their entrywise product $\big(\tau[t_{1,\sigma}^i|{t^j_{1,\sigma'}}^\dagger] \cdots \tau[ t_{L,\sigma}^i|{t^j_{L,\sigma'}}^\dagger]\big)_ {(i,\sigma),(j,\sigma')} $ is also non-negative (\cite[Lemma 6.11]{NS}). This 
yields the positivity of above right-hand-side.\end{proof}

\begin{proof}[Proof of Lemma \ref{Lem:Structure}] According to Lemmas \ref{Lem:DSD} and Corollary \ref{Cor:Monom}, in order to prove any of these three statements, it is enough to consider $\tau[h|h'],$ where  $h=T_g(\mbf t)$ and $h'=T_{g'}(\mbf t')$, with  
$g,g' \in 
\mcal B_{alt}^{(n)},$ $\mbf t=t_1\otimes \dots \otimes t_L$ a $g$-colored tensor product, $\mbf t' = t'_1\otimes 
\dots \otimes t'_{L'}$ a $g'$-colored tensor product, such that for each $\ell=1\etc L, \ell'=1\etc L',$ $t_\ell = p(m_\ell), t'_{\ell'} = p(m'_{\ell'}), $ where $m_\ell$ (respectively  $m'_{\ell'}$) is $n_\ell$-graph monomial (respectively a $n'_{\ell'}$-graph monomial) whose outputs are pairwise distinct. It suffices to prove that $\tau[h|h']=0$ if $g$ or $g'$ is not a tree and if $g$ and $g'$ do not belong to the same class of alternated colored bigraph operations, and to prove the formula of the third statement.

 Assume that the integers $\ell, \ell'$ such that $n_\ell, n_{\ell'} = 1$ are $\{1\etc K\}$ and $\{1\etc K'\}$ respectively. For any multi-index $(\mbf i, \mbf i') =(i_1\etc i_{K},i'_1 \etc i'_{K'} )$ in $\{0,1\}^{K+K'}$, let $h_{\mbf i, \mbf i'}$ be the graph polynomial $T_g(\tilde t_1\otimes \dots \otimes \tilde t_L)|T_g(\tilde t'_1 \otimes \dots \otimes \tilde t'_{L'})$ where
 \begin{itemize}
 	\item $\tilde t_\ell = t_\ell$ if $\ell>K$, 
	\item $\tilde t_\ell = m_\ell$ if $\ell\leq K$ and $i_\ell=0$, 
	\item $\tilde t_\ell =  (\cdot)$ if $i_\ell=1$,
\end{itemize}
and $\tilde t'_{\ell'}$ is defined similarly, so that
	\eqa\label{Eq:ProofPositivity}
	 h | h' = \sum_{(\mbf i, \mbf i')\in \{0,1\}^{K+K'}} \prod_{\substack{ \ell=l \etc K\\ \mrm{s.t.} \ i_k = 1}} \big(-\tau[t_k] \big)\times \prod_{\substack{k'=1 \etc K'\\ \mrm{s.t.} \ i_{k'} = 1}} \big( -\tau[t_{k'}]\big) \times h_{\mbf i, \mbf i'}.
	\qea
	
We can apply Lemma \ref{Lem:RedEffect} to each graph polynomial $h_{\mbf i, \mbf i'}$. Denote by $g_{\mbf i}$ and $g'_{\mbf i'}$ the colored bigraph operations obtained by erasing $1$-graph monomials such that $i_\ell=1$ and $i'_{\ell'}=1$ respectively, and $g_{\mbf i} | g'_{\mbf i'}$ the bigraph operation obtained by identifying the $i$-th outputs of $g_{\mbf i}$ and $g'_{\mbf i'}$ for any $i=1\etc n$. Denote by $(\tilde t_1\otimes \dots \otimes \tilde t_L \otimes\tilde t'_1\otimes \dots \otimes \tilde t'_{L'})_{\mbf i, \mbf i'}$ the tensor product where $1$-graph monomials are discarded when $\ell\leq K$ and $i_\ell=1$ and $\ell'\leq K'$ and $i'_{\ell'}=1$.
Then we have the identity
	$$h_{\mbf i, \mbf i'} = T_{g_{\mbf i}|g'_{\mbf i'}}(\tilde t_1\otimes \dots \otimes \tilde t_L \otimes  \tilde t'_1\otimes \dots \otimes \tilde t'_{L'})_{\mbf i, \mbf i'},$$
and this graph polynomial satisfies the assumptions of Lemma \ref{Lem:RedEffect}. Hence, we get
	\eqa\label{ProofPos1}
		 \tau\big[ h_{\mbf i, \mbf i'} \big] = \sum_{ \substack{ \pi \in \mcal P(V_{\mbf i, \mbf i'}) \\ \mrm{solid}}} \tau^0\big[ T_{\mbf i, \mbf i}^\pi\big],
	\qea
where we denote
\begin{itemize}
	\item the graph monomial 	$T_{\mbf i, \mbf i'} = T_g(\tilde m_1\otimes \dots \otimes \tilde m_L)|T_g(\tilde m'_1 \otimes \dots \otimes \tilde m'_{L'}),$ with $\tilde  m_\ell$ defined as $\tilde t_\ell$ with $m_\ell$ instead of $t_\ell$ in the first case,
	\item $V_{\mbf i, \mbf i'}$ the vertex set of $T_{\mbf i, \mbf i'}$, 
	\item $\mcal O_\ell$ and $\mcal O'_{\ell'}$ are the sets of outputs of $m_\ell$ and $m'_{\ell'}$ respectively, seen in $V_{\mbf i, \mbf i'}$ for $\ell>K, \ell'>K'$,
	\item $\pi_{|\mcal O_{\ell}}$ and $ \pi_{|\mcal O_{\ell'}}$ the restriction of $\pi$ to these sets. 
\end{itemize}

Solidity is with respect to the graph monomials $m_1\etc m_L, m'_1\etc m'_L$ (out of $1$-graph monomials such that $i_\ell=i'_{\ell'}=1$). Note that these graphs are not the colored components of $T$, because of possible identifications between inputs of $g_\mbf i$ and $g'_{\mbf i'}$ that are neighbors of the outputs when forming $g_\mbf i | g_{\mbf i'}$, as in the third example of Figure \ref{Fig:05} of the previous section. 

We first assume that $g$ or $g'$ is not a tree and prove that a solid partition is not valid, so we will conclude that $\tau[h|h']=0$ for any $h'$, as we expect. Note that for any $\mbf i, \mbf i'$, we have that $g_{\mbf i}$ or $g_{\mbf i'}$ is not a tree. We apply Lemma \ref{Lem:CycleGCC} to $T_{\mbf i, \mbf i'}$, any partition $\pi$ solid w.r.t. the $m_\ell$'s and $m'_{\ell'}$'s, and a cycle $C$ on $\mcal G \mcal C \mcal C( T_{\mbf i, \mbf i'})$ coming from a simple cycle of $g_{\mbf i}$. Note that $\mcal C$ is indeed simple since identifications with inputs of $g'$ do not change the cycle, see Figure \ref{Fig:11}. Solidity of the $m_\ell$'s implies that there is no possible identifications of connectors neighbouring  a same input on the cycle $\mcal C$. Hence $\pi$ cannot be valid. From now on, we shall assume that $g$ and $g'$ are trees.

  \begin{figure}
    \begin{center}
     \includegraphics{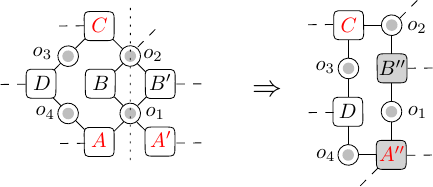}
    \end{center}
    \caption{ Left: a local detail of the bigraph operation $g_\mbf i|g'_{\mbf i'}$, with the vertical dotted line separating $g_\mbf i$ and $g'_{\mbf i'}$. Right: the bigraph operation $g_{alt, \mbf i, \mbf i'}$. The sequence $A, o_1, B, o_2, C, o_3, D, o_4$ forms a simple cycle in $g$. Going from $g_\mbf i|g'_{\mbf i'}$ to $g_{alt, \mbf i, \mbf i'}$, the inputs $A$ and $A'$ are identified. Yet, if a partition does not identify $o_1$ and $o_4$ in the leftmost picture, then it does not identify $o_1''$ and $o_4$ in the rightmost one.}
    \label{Fig:11}
  \end{figure}

Let us use now the centering of $1$-graph polynomials. Let $k=1\etc K$ be an index such that $i_k=0$ ($m_k$ is in $g_\mbf i$) and let $\pi$ be a partition of $V_{\mbf i, \mbf i'}$. We say that $m_k$ is \emph{isolated} by $\pi$ whenever no vertex of $m_k$ is identified with a vertex of another colored component except in the trivial way for a vertex of a neighboring component identified with the connector linking them. We say that $\pi$ is not isolating whenever no $m_k$ nor $m'_{k'}$ is isolated, for $k=1\etc K$ and $k'=1\etc K'$. By the multiplicativity property w.r.t. the colored components in the definition of traffic independence, for any valid partition $\pi$
	$$\tau^0[T_{\mbf i, \mbf i}^\pi] = \Bigg(\prod_{\substack{\ell =1\etc L  \\ \mrm{s.t. \ } m_\ell  \mrm{ \ isolated}}}  \tau^0[ T_{\ell}^{\pi_{|V_{\ell}}}] \times \prod_{\substack{\ell' =1\etc {L'}  \\  \mrm{s.t. \ } m'_{\ell'} \mrm{ \ isolated}}}  \tau^0[ {T'_{\ell'}}^{\pi_{|V'_{\ell'}}}] \Bigg)\times \tau^0[T_{\mbf j, \mbf j'}^{\pi_{|V_{\mbf j, \mbf j'}}}],$$
where $(\mbf j, \mbf')\in \{0,1\}^{K+K'}$ is defined by $j_\ell = 1$ if and only if $i_\ell=1$ or $m_\ell$ is isolated. Note that $\pi_{|V_{\mbf j, \mbf j'}}$ is not isolated. Hence, with the notations 
\begin{itemize}
	\item $\eps(\mbf i, \mbf i'):=\prod_{k=1}^{K}(-1)^{i_k}  \prod_{k'=1}^{K}(-1)^{i'_{k'}}$,
	\item $\alpha(\mbf i, \mbf i')  = \prod_{ k=1\etc K } \tau[t_{k}]^{i_k}\prod_{k'=1\etc K' } \tau[t'_{k'}]^{i'_{k'}}$,
\end{itemize}
we have
	\eqa
		\tau[h] & = & \nonumber
		\sum_{(\mbf i, \mbf i') }  
		\eps(\mbf i, \mbf i')  \times 
		\Bigg(
		\sum_{ \substack{ (\mbf j, \mbf j') \\ j_k \geq i_k \, \forall k \\ j'_{k'}  \geq i'_{k'}  \forall k'}}
		\alpha(\mbf i, \mbf i') \times \Big(
		\sum_{ \substack{ \pi \in \mcal P(V_{\mbf j, \mbf j'}) \\ \mrm{ solid}\\ \mrm{non \ isolating}}}\tau^0[T_{\mbf j, \mbf j'}] \Big) \Bigg)
		\\
		& =&   \nonumber
		\sum_{(\mbf j, \mbf j') }  
		\Bigg(  \sum_{ \substack{ (\mbf i, \mbf i')  \\ i_k \leq j_k \, \forall k \\ i'_{k'}  \leq j'_{k'} \, \forall k'}}
		 \eps(\mbf i, \mbf i')\Bigg)   
		 \times \alpha(\mbf j, \mbf j') \sum_{ \substack{ \pi \in \mcal P(V_{\mbf j, \mbf j'}) \\ \mrm{ solid}\\ \mrm{non \ isolating}}}\tau^0[T_{\mbf j, \mbf j'}^\pi]\\
		& = & \sum_{ \substack{ \pi \in \mcal P(V) \\ \mrm{ solid}\\ \mrm{non \ isolating}}}\tau^0[T^\pi], \label{ProofPos2}	
\qea
where $T=T_{g|g}(m_1\otimes \dots \otimes m_L \otimes m'_1\otimes \dots \otimes m'_{L'})$ and $V$ its vertex set. In words, the trace of $h$ is the sum of the injective traces of quotients of $T$ by solid and non isolating partitions.

On the other hand, we claim that the valid partitions of $T$ solid w.r.t. the $m_\ell$'s and $m'_{\ell'}$'s satisfy half of the primitivity property of Lemma \ref{Lem:Primitive}: two vertices $v$ and $w$ of $T_{\mbf i, \mbf i'}$ that come from $g_{\mbf i}$ (respectively from $g'_{\mbf i'}$) can be identified by a valid partition solid w.r.t. the $m_\ell$'s only in the trivial situation: they belong to a same colored component, or they belong to neighboring components and are identified with the vertex that belong to both the components. Indeed, let us assume conversely that $\pi$ is a solid partition that identifies $v$ and $w$. We apply as usual Lemma \ref{Lem:CycleGCC} to the graph $(g_{\mbf i})_{v \sim w}$, the induced partition, and the cycle given by a path between $v$ and $w$ in $g$. Solidity of the $m_\ell$'s implies that there is no possible identifications of connectors which are neighbors of  a same input, except possibly around $v\sim w$, see Figure \ref{Fig:12}. So $\pi$ is not valid except in the trivial case.

  \begin{figure}
    \begin{center}
     \includegraphics{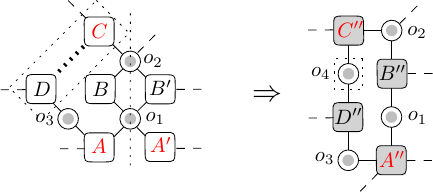}
    \end{center}
    \caption{ Left: the bigraph operation $g_\mbf i|g'_{\mbf i'}$ with the dot rectangle representing the identification of two vertices. Right: the graph of colored components of $(S_{\mbf i, \mbf i'})_{v\sim w}$. Identifications $o_4\sim_\pi o_3$ and $o_4 \sim_\pi o_3$ are possible since $o_4$ appeared while identifying $v$ and $w$, but other identifications $o_2 \sim_\pi o_1$ and $o_1 \sim_\pi o_3$ are not possible if they are not allowed in the leftmost graph.}
    \label{Fig:12}
  \end{figure}
  
We are now ready to prove that if $\tau[h] \neq 0$ then $g$ and $g'$ are isomorphic. Recall that we assume $g$ and $g'$ are trees. Let $\pi$ be a valid and solid partition which does not isolate $1$-graph monomials, as in Formula \eqref{ProofPos2}. Because of the argument of the previous paragraph, each $1$-graph monomial of $g$ must be identified with a single $1$-graph monomial of $g'$, which defines a bijection $\sigma$ between the leaves of $g$ and $g'$. We now show that 
\begin{itemize}
	\item for any $i_1, i_2=1\etc n$, the unique path from the $i_1$-th to the $i_2$-th outputs of $g$ is isomorphic to the unique path between the same outputs in $g'$.
	\item if $S$ and $S'$ are two leaves of $g$ and $g'$ such that $\sigma(S) = \sigma(S')$, then for any $i=1\etc n$, the unique path from $S$ to the $i$-th output of $g$ is isomorphic to the unique path from $S'$ to the same output in $g'$,
\end{itemize}
This clearly implies that $g$ and $g'$ must be isomorphic. For the first point, consider a simple path $\mcal D: o_1, S_1 \etc  S_Q, o_{Q+1}$ between two outputs in $g$ and the simple path $\mcal D':o'_1, S'_1 \etc  S'_{Q'}, o'_{Q'+1}$ in $g$ between the same outputs. We apply Lemma \ref{Lem:CycleGCC} to $T_{\mbf i, \mbf i'}$, a solid and valid partition $\pi$, and a cycle $\mcal C$ formed by identifying extremities of the paths. Denote by $j_q$ the color of $S_q$ and $j'_{q'}$ the one of $S'_{q'}$. The inputs $S_q, S'_{q'}$ different from $S_1,S'_1, S_Q,S'_{Q'}$ are components of $g$ or $g'$ and they are therefore solid in $\pi$. The inputs $S_1$ and $S'_1$ are identified $\mcal C$ if and only if $j_1 = j'_1$, and the same holds for the last inputs $S_Q$ and $S'_{Q'}$. Hence necessarily they are the only pairs of identification. If $Q\geq 2$ and $Q'\geq 2$, we can iterate this reasoning on the graph obtained from $S_{\mbf i, \mbf i'}$ with these two identifications, see Figure \ref{Fig:06}. Hence the two colored paths $\mcal D$ and $\mcal D'$ are isomorphic: one has $Q=Q'$ and $j_q=j'_q$ for any $q=1\etc Q$, and moreover the partition $\pi$  identifies pairwise $o_k \sim_\pi o'_k$ for any $k=1\etc K$. For the second point, the proof is the same with paths form the colored components to the outputs.
   
  \begin{figure}
    \begin{center}
     \includegraphics[width=100mm]{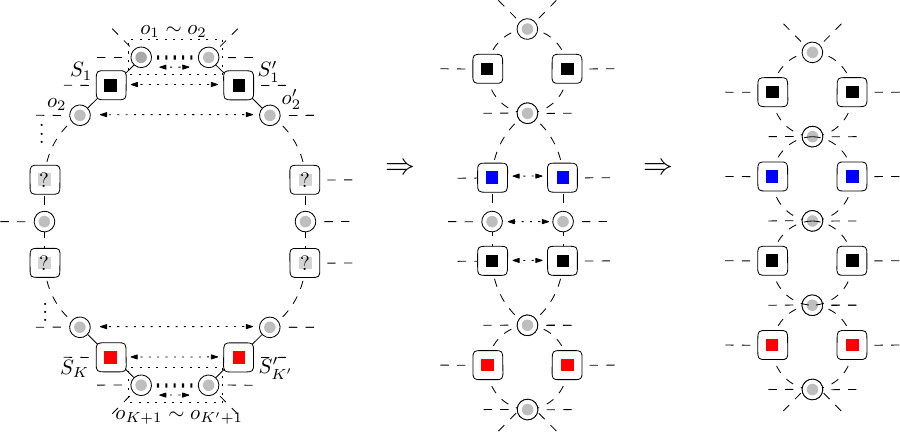}
    \end{center}
    \caption{Left: two paths from the same outputs that form a simple cycle in $g|g'$. The pair of extremal inputs must be of a same color if a quotient whose $\mcal G \mcal C \mcal C$ is a tree exists. Such a quotient must be a quotient of the graph with identifications $o_2\sim o_2'$ and $o_{Q}\sim o_{Q'}$, so we can iterate the reasoning.}
    \label{Fig:06}
  \end{figure}

Valid partitions identifying pairwise a connectors of $g$ with connectors of $g'$, the multiplicativity property in the definition of traffic independence yields the expected formula.
\end{proof}

\subsection{The tensor product of traffic spaces}

Let $J$ be an integer and for each $j=1\etc J$, let $(\mcal A_j, \tau_j)$ be an algebraic traffic space. We construct a traffic space $(\bigotimes_j \mcal A_j, \bigotimes_j \tau_j)$, that contain each traffic space $\mcal A_j$ and such that the $\mcal A_j$ commute. Their algebraic tensor product $\bigotimes_j \mcal A_j$ is indeed a $\mcal G$-algebra with action of $K$-graph operations 
	\eq
		\lefteqn{ Z_g\big( (a_{1,1} \otimes  \cdots \otimes   a_{1,J}) \otimes \cdots \otimes (a_{K,1} \otimes  \cdots \otimes   a_{K,J})\big) }
		\\ &=& Z^{(1)}_g(a_{1,1} \otimes \cdots \otimes a_{K,1}) \otimes \cdots \otimes  Z^{(J)}_g(a_{1,J} \otimes \cdots \otimes a_{K,K}),
	\qe
 for any $g\in \mcal G_K$ and any $a_{k,j} \in \mcal A_j$, where $Z^{(j)}$ denotes the action of graph operations on $\mcal A_j$, $j=1\etc J$. The tensor product of the combinatorial traces is defined, for any $T\in \mcal T\langle \bigotimes_j \mcal A_j \rangle$ whose edges are labeled by pure tensor products, by $\bigotimes_j\tau_j[T] = \tau_1[T_1] \cdots  \tau_J[T_J]$, where $T_j$ is obtained from $T$ by replacing a label $a_1\otimes \cdots \otimes \cdots a_J$ by $a_j$. 

We will need later the following lemma.

\begin{Lem}\label{Lem:TensorInj} The injective version of $\bigotimes_j\tau_j$ is given as follow. For any test graph $T\in \mcal T\langle \bigotimes_j\mcal A_j\rangle$, denote by $\Lambda_T$ the set of $J$-tuples $(\pi_1\etc \pi_J) \in \mcal P(V)^J$ such that if two elements belong to a same block of $\pi_i$ then they belong to different blocks of $\pi_j$ for some $j\neq i$. Then 
	$$\big(\bigotimes_j\tau_j\big)^0[T]  =  \sum_{(\pi_1\etc \pi_J) \in \Lambda_T} \tau_1^0[T_1^{\pi_1} ] \cdots \tau_J^0[T_J^{\pi_J}].$$
\end{Lem}

\begin{proof} We clearly have 
	\eq
		  \sum_{\pi \in \mcal P(V)} \Big( \sum_{(\pi_1\etc \pi_J) \in \Lambda_{T^\pi}} \prod_{j=1}^J \tau_j^0[T_j^{\pi_j}] \Big) & = &  \prod_{j=1}^J \sum_{\pi_j \in \mcal P(V)} \tau_j^0[T_j^{\pi_j}]  =  \bigotimes_j \tau_j[T].
		\qe
This implies the expected result by uniqueness of Mobi\"us transform.
\end{proof}

If the spaces are traffic spaces, i.e. if the maps $\tau_j$'s are positive, then their tensor product is also a traffic space by the usual argument of positivity of the Hadamard product \cite[Lemma 6.11]{NS}. 

\section{Conclusion of Part I and perspectives}\label{Sec:ConcluI}

Our initial motivation was to prove the positivity of traffic independence. This is indeed important since many statements about traffics involve the consideration of independent variables, for instance the law of large numbers and the central limit theorem. Doing so, we have discovered a natural characterization of traffic independence and understood more about the nature of this notion. In particular, Corollary \ref{Corfreeind} gives a tool to use traffic independence in order to prove the free independence of variables.

In this section we would like to emphasis an aspect which is present in this analysis we hardly mentioned: the notion of traffics, made initially for the analysis of large matrices, can be generalized to cover the case of tensor matrices of arbitrary order. Whereas a matrix $A=(A_{i,j})_{i,j=1}^N \in \mrm M_N(\mbb C) \sim \mbb C^N \otimes \mbb C^N$ is a collection of complex numbers labeled by two variables, a tensor matrix $A=(A_{\mbf i})_{\mbf i\in [N]^n} \in (\mbb C^N)^{\otimes n}$, $n\geq 1$, is labeled by a multi-index $\mbf i \in [N]^n$. 

\subsection{Operadic aspects: bigraph operations and wiring diagrams}

Recall that $\mcal B = \bigcup_{n\geq 0} \mcal B^{(n)}$ denotes the set of bigraph operations. It is actually a so-called \emph{colored operad}, encoding operations on objects of different kinds (the rank of the graph polynomials in the preceding sections). It is in bijection with the so-called operad of (undirected, connected) \emph{wiring diagrams}, initially introduced by Rupel and Spivak \cite{RuSp13} and developed in several variants in \cite{Sp13,Sp15,VaSpEu15}. They have an important potential to model and design complex systems in many different disciplines, such as computer science and cognitive neuroscience. They are applied to study open dynamical systems, certain differential equations, databases recursions, plug-and-play circuits, etc. See \cite{Yau15} for an extended presentation of wiring diagrams. Formulating the notion of traffic independence is then a new application of wiring diagrams.

  \begin{figure}
    \begin{center}
     \includegraphics{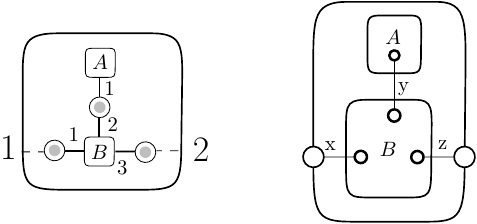}
    \end{center}
    \caption{On the left, a bi-graph operation $t$ , and on the right, the associated wiring diagram $t'$. The inputs of $t$ (resp. its outputs) are represented by internal boxes (resp. by the external box). Connectors of $t$ are represented by bonds between the boxes. Note that general wiring diagram are not assume to be connected.}
    \label{Fig:16}
  \end{figure}

As for graph operations, we can define a notion of algebra over the operad $\mcal B$ of bigraph operations (or equivalently of wiring diagrams), in short $\mcal B$-algebra. Since the operad is colored, a $\mcal B$-algebra is a graded vector space $\mcal A = \bigoplus_{n\geq 0} \mcal A^{(n)}$ endowed with an action of the elements of $\mcal B$ as follow. For any bigraph operation $g\in \mcal B^{(n)}_{L,\mbf d}$, i.e. with $L$ inputs and sequence of input degree $\mbf d=(d_1\etc d_L)$, there is a linear map $T_g : \mcal A^{(d_1)} \otimes \dots \otimes \mcal A^{(d_L)} \to \mcal A^{(n)}$ satisfying the following properties:
\begin{enumerate}
	\item For the bigraph operation $id_n\in \mcal B^{(n)}$ with $n$ distinct outputs, a single input of degree $n$ and such that the $i$-th neighbor of the input is the $i$-th output, one has $T_g= id_{\mcal A^{(n)}}$, 
	\item $T_g\big( T_{g_1} \otimes \ldots \otimes T_{g_L})=T_{g \,  (g_1,\ldots, g_L)}$, for any $g\in \mcal B^{(n)}_{L, \mbf d}$ and $g_\ell\in \mcal B^{(d_\ell)}, \ell=1\etc L$,
	\item $T_g(a_1 \otimes \dots \otimes a_L) = T_{g_\sigma}(a_{\sigma^{-1}(1)} \otimes \dots \otimes a_{\sigma^{-1}(L)})$ for any permutation $\sigma$ that only permute elements of same degree, where $g_\sigma$ is as $g$ with the $i$-th input becoming the $\sigma^{-1}(i)$-th input.
\end{enumerate}

The definition of Example \ref{Rk:BigraphOp} defines a structure of $\mcal B$-algebra on the space of random variables with values on $\bigoplus_{n\geq 0} (\mbb C^N)^{\oplus n}$.

Another operadic aspects of bigraph operations to be commented is about its relations with the work of Jones \cite{Jones} on planar algebras. Firstly, we can mention that the action of graph operations on tensor matrices appears in this paper as a slight generalization of \cite[Example 2.6]{Jones}. Recall that a \emph{combinatorial map}, i.e. a graph embedded into a surface such that the connected components of the complementary of the graph are isomorphic to discs. Then by a result of Heffter \cite{Heffter} a combinatorial map is a bigraph operation with no output such that each connector as degree two. It would be interesting to know if this set of observables is associated to some distributional symmetry of random tensor matrices.

\subsection{Definition of traffics of arbitrary ranks and their independence}
Most of the analysis provided for traffic spaces can be generalized \emph{in extenso} in the context of $\mcal B$-algebra. We briefly describe in this section this generalization. 

\begin{Def} An algebraic \emph{traffic space of arbitrary rank} is a pair $(\mcal A, \tau)$ where $\mcal A = \bigoplus_{n\geq 0} \mcal A^{(n)}$ is a $\mcal B$-algebra and $\tau$ is a linear form on $\mcal A^{(0)}$ such that $\tau\big[(.)\big] = 1$.
\end{Def}

The term \emph{algebraic} can be dropped provided $\tau$ satisfies the positivity property of Section \ref{Sec:DefPos}. Contrary to the rank two case, there is no need for associativity and multi-linearity properties for $\tau$ since there are encoded in the structure. An element of $\mcal A_n$ is called a traffic of rank $n$. Hence the theory of traffics is the theory of traffics of rank two. By Example \ref{Rk:BigraphOp}, the space of tensor matrices is a traffic space of arbitrary rank.

Notions of $\mcal B$-subalgebras and reduced elements are defined as for $\mcal G$-algebras. 

\begin{Def}\label{Def:FreeBalg} The $\mcal B$-subalgebras $\mcal A_1\etc \mcal A_L\subset \mcal A$ are said to be independent whenever any bigraph operation in alternated reduced elements is centered. 
\end{Def}

With minor modifications of the proof, we can prove the following extension of the asymptotic traffic independence theorem. Let $\mbf A_1^{(N)} \etc \mbf A_L^{(N)}$ be independent families of tensors, namely for each $\ell=1\etc N$, $\mbf A_\ell^{(N)}= ( A_{\ell,j})_{j\in J_\ell}$ where $A_j \in (\mbb C^N)^{\otimes n_j}$ for some $n_j\geq 1$. Assume the following properties:
\begin{enumerate}
	\item For each $\ell=1\etc L$, the family $\mbf A_\ell^{(N)}$ is permutation invariant in law in the sense that for any permutation $\sigma$ of $[N]$,
		$$\mbf A_\ell^{(N)} = ( A_{\ell,j})_{j\in J_\ell} \overset{\mcal L}=  \Big( A_{\ell,j}\big(\sigma(i_1) \etc \sigma(i_{n_j}) \big)_{\mbf i\in [N]^{n_j}} \Big)_{j\in J_\ell}.$$
	\item For any bigraph operation $g$ of rank 0, the quantity $\esp\big[ T_g\big( \mbf A_{\ell_1,j_1} \otimes \dots \otimes \mbf A_{\ell_K,j_K}\big)\big]$ converges as $N$ goes to infinity for any compatible $\mbf A_{\ell_k,j_k}$. 
	\item On has the asymptotic factorization property: 
		\eq
			\lefteqn{ \prod_{i=1}^n \esp\Big[ T_{g_i}\big(   A_{\ell_{1,i},j_{1,i}} \otimes \dots \otimes   A_{\ell_{K_i,i},j_{K_i,i}}\big)\Big]}\\
			& & - \esp\Big[ \prod_{i=1}^n T_{g_i}\big(   A_{\ell_{1,i},j_{1,i}} \otimes \dots \otimes   A_{\ell_{K_i,i},j_{K_i,i}} \big)\Big] \limN 0,
		\qe
for any $g_1\etc g_n$ and any tensors such that the evaluation makes sense.
\end{enumerate} 
Then the families $\mbf A_1^{(N)} \etc \mbf A_L^{(N)}$ are asymptotically independent. This can be proved in two steps analogous to the random matrices case. The first determines   the limit in terms of an analogue for the injective trace, following the argument of \cite{Male2011} for random matrices. The second   characterises  traffic independence in terms of the latter injective trace, replacing graph polynomials by elements of a $\mcal B$-algebra (let us stress emphasis that this proof does not rely on the $\mcal G$-algebra structure of traffic spaces).

\subsection{Potential perspectives}

Traffics of arbitrary ranks may be interesting to study new objects, e.g. simplicial complexes, in a similar fashion we study the non-commutative distributions of large random graphs.  Moreover, it could also open new perspectives,   for instance in the setting of Voiculescu's notion of bi-freeness \cite{Janus} . Let $(\mcal A, \Phi)$ be a $^*$-probability space such that $\Phi$ is a faithful state. Denote by $H$ the closure of $\mcal A$ for the Hilbert norm $a \mapsto \sqrt{ \Phi(aa^*)}$. There are two commuting actions of $\mcal A$ on $H$ given by left and right multiplications. Recall that the theory of bi-freeness is about the relations of free operators from these two points of view. Let now $(\mcal A, \tau)$ be a  traffic space of arbitrary rank. Assume that $\tau$ is positive and for any $a \in \mcal A^{(n)}$ one has $\|a\|_2:=\sqrt{\tau[a|a^\dagger]} =0$ implies $a=0$ (the definition of $a|a^*$ is the same as for $g|g^*$ in traffic spaces as it only involves bigraph operations). Consider the closure of $\mcal A^{(n)}$ by $\| \cdot \|_2$. Then we have now $n$ commuting actions of $\mcal A$, one in the direction of each output, and so we can develop a theory of \emph{multi-freeness} dedicated in the relations between these actions.

\newpage
\part{On the three types of traffics associated to non-commutative independences}\label{Sec:CanonicalExtension}
\section*{Presentation of Part 2}

In \cite{Male2011} three types of traffics were identified, one for each notion of the three non-commutative notions independence.
\begin{Def}Let $(\mcal A, \tau)$ be an algebraic traffic space and let $\mbf a=(a_j)_{j\in J}$ a family of elements of $\mcal A$. We said that $\mbf a$ is of
\begin{itemize}
	\item \emph{free type} if it is \emph{unitarily invariant}, in the sense that $\mbf a$ has the same traffic distribution as $u\mbf a u^*=(ua_ju^*)_{j\in J}$, where $(u,u^*)$ is traffic independent of $\mbf a$ and limit of $(U_N,U_N^*)$ for a Haar unitary random matrix $U_N$;
	\item \emph{Boolean type} if, for any $T\in \mcal T\langle \mcal Y\rangle$, one has $\tau[T] = 0$ if $T$ is not a tree;
	\item \emph{tensor type} if the traffics are diagonals, in the sense that $a_j=\Delta(a_j)$ for all $j\in J$.
\end{itemize}
\end{Def}
The precise link with the usual notions of independences is given by \cite[Theorem 5.5]{Male2011}:
\begin{itemize}
\item the traffic independence of traffics of \emph{free type} implies the \emph{free independence} with respect to the trace $\Phi$ of $(\mcal A, \tau)$;
\item the traffic independence of traffics of \emph{Boolean type} implies the \emph{Boolean independence} with respect to the anti-trace $\Psi$ of $(\mcal A, \tau)$;
\item the traffic independence of traffics of \emph{tensor type} implies the \emph{tensor independence} with respect to the trace $\Phi$ of $(\mcal A, \tau)$;
\end{itemize}

This section is mostly devoted to the study of traffics of free types. We give an explicit description of the injective distribution of traffics of this type, and we prove Theorems \ref{Th:Matrices} and \ref{Th:MatricesBis} about unitarily invariant random matrices.

More generally, we give, for each type of traffic, a characterization as a particular symmetry of the traffic distribution,  a characterization with respect to the injective distribution, and for any non-commutative probability space we construct a canonical traffic space containing the initial space as a subalgebra of traffics of each type, under mild assumptions. The whole picture is contained in Section~\ref{Sec:three_types_traffics}.

\section{Generalities on unitarily invariant traffics}\label{Sec:GeneralUnitTraffic}

The canonical construction of free type consists in proving that any tracial non-commutative probability space $(\mcal A, \Phi)$ can be realized as an algebra of unitarily invariant traffics. A crucial step is an explicit description of the distribution of traffics of this type, which is given in the two next sections.

\subsection{Cacti and non crossing partitions}

Recall that we call simple cycle of a graph a closed path visiting pairwise distinct vertices (orientation of the edges is ignored).

\begin{Def} A cactus is a finite connected graph such that each edge belongs exactly to one simple cycle. A well oriented cactus is a cactus such that the simple cycles of the graph are oriented. 
\end{Def}

  \begin{figure}[h!]
    \begin{center}
     \includegraphics[width=35mm]{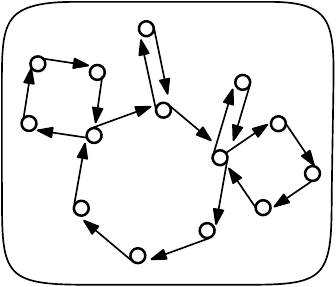}
    \end{center}
    \caption{A well oriented cactus.}
    \label{Fig:03}
  \end{figure}

 Well oriented cacti are related to non crossing partitions in the following way. Let $T$ be a test graph consisting in a simple cycle with consecutive edges  $(\cdot \overset{1}\leftarrow \cdot \dots \cdot \overset{n}\leftarrow \cdot)$. Let $\sigma$ be a non crossing partition of the set $E:=\{1\etc n\}$ of edges of $T$. Let us denote by $V=\{1' \etc n'\}$ the set of vertices of $T$, so that $i'$ is neighbor of $i$ and $i+1$ with notation modulo $n$. The \emph{Kreweras} complement $\hat \sigma$ of $\sigma$ is the largest partition of $V$ such that the partition $\sigma \sqcup \hat \sigma$ of $E\sqcup V$ is non crossing (with the convention $1< 1'< 2 <  \dots < n< n'$).
 
\begin{Lem}\label{Lem:CactiNCP} For any partition $\pi$ of $V$, the quotient $T^\pi$ is a well oriented cactus if and only if there exists $\sigma$ a non crossing partition of $E$ such that $\pi=\hat \sigma$.
\end{Lem}

  \begin{figure}[h!]
    \begin{center}
     \includegraphics[width=60mm]{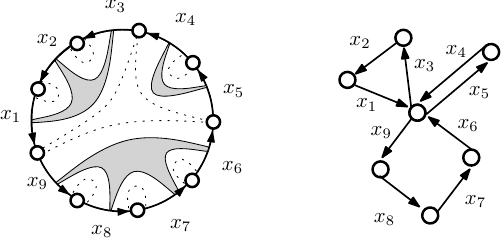}
    \end{center}
    \caption{Left: A cycle of length nine, a non crossing partition $\nu$ of its edges (grey) and the Kreweras complement $\pi$ (dotted) of $\nu$. Right: the quotient of the cycle by $\pi$.}
    \label{Fig:04}
  \end{figure}
The Kreweras operation $\sigma \mapsto \hat \sigma$ is a bijection $\mrm{NC}(n) \to \mrm{NC}(n')$. Hence the content of the lemma is unchanged if we replace the sentence "$\exists\, \sigma \in\mrm{NC}(n)$ such that $\pi = \hat \sigma$" by  "$\pi\in\mrm{NC}(n')$". There we emphasis the role of $\sigma$ since this is how non crossing partitions play a role in free probability theory.

\begin{proof} For any partition $\pi$ of $V$, denote $\sigma(\pi)$ the partition of the edges of $T$ such that $i\sim_\sigma i'$ if and only if $i$ and $i'$ belong to a same simple cycle of $T$. Note first that $T^\pi$ is a cactus if and only if there exists at least one \emph{isolated} simple cycle, that is a subgraph attached to the rest of the graph by a single vertex, and the graph without this simple graph is a cactus. Indeed, let $\mcal G$ be the (undirected) graph whose vertices are the simple cycles of $T$ with an edge between two cycles for each vertex they have in common. Then $T^\pi$ is a cactus if and only if $G$ is a tree. A leaf of this tree is a simple cycle with the expected property. On the other hand, $\sigma$ is a non crossing partition if and only if there is at least one block of $\sigma$ interval $I$ of $[n]$ and the restriction of $\sigma$ to $[n]\setminus I$ is non crossing (the proof is similar, see for instance \cite[Property 17.9]{GUI}). Since isolated simple cycles of $T$ correspond to intervals of $\sigma$, we get the expected property.
\end{proof}

\begin{Cor}\label{Cor:Cacti} Let $(\mcal A, \tau)$ be an algebraic traffic space with trace $\Phi$ and let $\mbf a$ be a family of elements of $\mcal A$.
Assume that the injective distribution of $\mbf a$ is supported on well oriented cacti and is multiplicative w.r.t. their cycles, that is: for any test graph $T\in \mcal T\langle \mbf a \rangle$, 
		$$\tau^0[T] = \one\big( T \mrm{ \ is  \ a \ well \ oriented \ cactus} \big) \times \prod_{ C} \tau^0[C],$$
where the product is over the simple cycles of $T$. Then for any simple cycle $C$ with consecutive edges $(\cdot \overset{a_1}\leftarrow \cdot \dots \cdot \overset{a_n}\leftarrow \cdot)$ we have $\tau^0[C] = \kappa_n(a_1\etc a_n)$ where $\kappa_n$ is the $n$-th free cumulant function relative to the trace $\Phi$. 
\end{Cor}

\begin{proof} Let $T$ denotes a simple cycle with consecutive edges $(\cdot \overset{a_1}\leftarrow \cdot \dots \cdot \overset{a_n}\leftarrow \cdot)$. Then the definition of $\Phi$, the formula for $\tau^0$, and the lemma yield
	\eq
		\Phi( a_1 \dots a_n) & = & \tau[T] = \sum_{\pi \in \mcal P(V)} \one(T^\pi \mrm{ \ well \ oriented \ cactus}) \prod_{ C} \tau^0[C]\\
		& = & \sum_{\sigma \in NC(n)} \prod_{ C\in \sigma} \tau^0[C],
	\qe
with in the second line the abuse of notation that a cycle $C$ with consecutive edges $(\cdot \overset{a_{i_1}}\leftarrow \cdot \dots \cdot \overset{a_{i_\ell}}\leftarrow \cdot)$ of a cactus $T^{\hat \sigma}$ is identified with the corresponding block $\{i_1  \etc i_\ell\}$ of $\sigma$. Since $\tau^0$ is multi-linear when seen as function of the labels of its edges, this property characterizes the free cumulants functions by M\"obius inversion formula stated in Section \ref{Sec:Mobius}.\end{proof}

The motivation to introduce this notion is that the traffic distribution of $\mbf a$ is completely determined by its non-commutative distribution in $(\mcal A, \Phi)$ since free cumulants are determined by $\Phi$. This is the starting point of the canonical construction which is developed in Section \ref{section:Canonical_construction}. Before stating this, we first present properties and example of such traffics.

\subsection{Unitarily invariant traffics}

Let us temporarily say that a family of traffics is \emph{of cactus type} when its injective distribution is supported on well oriented cacti and multiplicative w.r.t. their cycles, as in Corollary \ref{Cor:Cacti}. We characterize this ensemble of traffics in terms on the following distributional symmetry.

\begin{Def}\label{Def:UI} Let $(\mcal A, \tau)$ be an algebraic traffic space and $\mbf a=(a_j)_{j\in J}$ be a family of elements of $\mcal A$. We say that $\mbf a$ is \emph{unitarily invariant} if and only if it has the same traffic distribution as $u\mbf a u^*=(ua_ju^*)_{j\in J}$, where $(u,u^*)$ is traffic independent of $\mbf a$ and limit of $(U_N,U_N^*)$ for a Haar unitary random matrix $U_N$. 
\end{Def}

\begin{Prop}\label{EqCactusUI} A family of traffics is unitarily invariant if and only if it is of cactus type.
\end{Prop}

The proof of the proposition is given in Section \ref{Sec:UISCMC} and requires an analysis of the geometry of cacti and graph of colored components.

\begin{Cor} Let $(\mcal A, \tau)$ be an algebraic traffic space and let $\mbf a$ be a family of traffics of cactus type. Then the unital algebra spanned by $\mbf a$ is of cactus type.
\end{Cor}

\begin{proof} Let $\mbf b = \big(P_j(\mbf a) \big)_{j\in J}$ for some non-commutative polynomials $P_j$, $j\in J$. Then $u \mbf b u^*= \big(P_j(u \mbf a u^*) \big)_{j\in J}$ has the same traffic distribution as $\mbf b$, so it is of cactus type.
\end{proof}

For all $N\geq 1$, let $\mbf A_N$ be a family of random matrices in $\mrm{M}_N(\mbb C)$. We recall that under the assumptions of Theorem \ref{Th:Matrices} (the convergence in $^*$-distribution and the asymptotic factorization of $^*$-moments), $\mbf A_N$ converges in traffic distribution toward a unitarily invariant family.

\begin{Th}\label{Th:MatricesBis} Under the above setting, the asymptotic factorization property holds for the traffic distribution: for all test graphs $T_1 \etc T_k$, we have the following convergence
	\eqa
		\lefteqn{ \lim_{N\to \infty} \mathbb{E}\left[\frac{1}{N}\mathrm{Tr}\left(T_1(\mbf A_N)\right)\cdots \frac{1}{N}\mathrm{Tr}\left(T_k(\mbf A_N)\right)\right]} \label{Eq:Facto}\\
		 &= &\lim_{N\to \infty} \mathbb{E}\left[\frac{1}{N}\mathrm{Tr}\left(T_1(\mbf A_N)\right)\right]\cdots \lim_{N\to \infty} \mathbb{E}\left[\frac{1}{N}\mathrm{Tr}\left(T_k(\mbf A_N)\right)\right].\nonumber
	\qea
\end{Th}

The proof of Theorems \ref{Th:Matrices} and \ref{Th:MatricesBis} is given in Section \ref{subsec:conv} and is based on Weingarten calculus. Factorization property of $^*$-moments is required to get the multiplicativity of the injective distribution with respect to the cycles of cacti as in Corollary \ref{Cor:Cacti}. Let us give now some examples of large random matrices converging to traffics of free types.

\begin{Ex}\label{Ex:UIMatrices}
\begin{enumerate}
 \item A Haar unitary matrix $U_N$ converges to a unitarily invariant traffic $u$ in some traffic space $(\mcal A, \tau)$, and we can assume that $u$ is unitary ($u^*u=uu^*=1$), see \cite{Male2011}. Denote by $\Phi$ the trace associated to $\tau$. It is known that in the non-commutative probability space $(\mcal A, \Phi)$, $u$ is a \emph{Haar unitary}, characterized by $\Phi\big(u^k(u^*)^\ell\big)=\one(k=\ell)$ for any $k, \ell\geq 0$. Recall that the only nonzero free cumulants of $u$ are
	$$\kappa_{2n}(u,u^* \etc u,u^*) = \kappa_{2n}(u^*,u \etc u^*,u) = c_{n-1}(-1)^{n-1}$$
 where $c_n = \frac{ 2n!} {(n+1)! n!}$ are the Catalan numbers. In particular, the injective traffic distribution of $u$ is supported on well oriented cacti whose cycles have even size and whose labels are alternated.
\item Let $X_N = \big( \frac{x_{i,j}}{\sqrt n } \big)_{i,j=1\etc N}$ be a complex Wigner matrix (the $x_{i,j}$ are independent and identically distribution along and out of the diagonal, the distribution of $x_{i,j}$ does not depend on $N$ and admit moments of all orders). Assume the entries are centered, invariant in law by complex conjugation ($x_{i,j} \overset{law}= \overline{x_{i,j}}$) and that $\esp[ |x_{i,j}|^2]=1$, $\esp[ x_{i,j}^2]=0$. Then $X_N$ converges to a unitarily invariant traffic $x$ in some traffic space $(\mcal A, \tau)$, and we can assume that $x$ is self-adjoint ($x^*=x$), see \cite{Male2011}. It is known that in the non-commutative probability space $(\mcal A, \Phi)$, $x$ is a \emph{semicircular variable}, characterized $\Phi(a^k) = \one(k \mrm{ \ even}) c_{k/2}$ for any $k$, where $c_n$ are the $n$-th Catalan numbers. The only nonzero free cumulant of $x$ is $\kappa_2(x,x) =1$. In particular, the injective traffic distribution of $x$ is supported on cacti whose cycles have size two (called the double trees in \cite{Male2011}).
	\item An interest of the notion of unitarily invariant traffics is that it is not restricted to the limit of unitarily invariant matrices, as we have seen in the previous example with Wigner matrices. Matrices which are asymptotically unitarily invariant can even be more structured. For instance, convergence to a unitarily invariant semicircular traffic remains true when Wigner matrix models is generalized to Wigner matrices with intermediated exploding moments (like diluted Erd\"os-Re\'nyi graphs) \cite{Male122}, for uniform regular graphs with large degree \cite{MP14} (when restricting the traffic distribution to \emph{cyclic test graphs}), periodic band Wigner matrices and band Wigner matrices with slow growth \cite{Au16}. Hence, the properties of unitarily invariant traffic we state below are asymptotically true for these models.
\end{enumerate}
\end{Ex}

\subsection{Relation with freeness and large random matrices}\label{Sec:RelFree}

\subsubsection{Abstract statement}

The following proposition motivates that unitarily invariant traffics are referred as traffics of \emph{free type}.

\begin{Prop}\label{Prop:FreenessFreeTraffics} Let $(\mcal A, \tau)$ be an algebraic traffic space with trace $\Phi$. For each $j\in J$ let $\mbf a_j$ be a family of traffics in $\mcal A$ and set $\mbf a = \cup_j \mbf a_j$. Let  $\mbf b$ be an arbitrary family of traffics in $\mcal A$. 
\begin{enumerate}	
	\item If $\mbf a_j$ is unitarily invariant for each $j\in J$ and the $\mbf a_j$'s are traffic independent then $\mbf a $ is unitarily invariant and the $\mbf a_j$'s are free independent in $(\mcal A, \Phi)$.
	\item Reciprocally if $\mbf a$ is unitarily invariant and the $\mbf a_j$'s are free independent in $(\mcal A, \Phi)$ then they are traffic independent in $(\mcal A, \tau)$. 
	\item If $\mbf a$ is unitarily invariant and is traffic independent from $\mbf b$ then $\mbf a$ and $\mbf b$ are freely independent in $(\mcal A, \Phi)$.
\end{enumerate}
\end{Prop}

\begin{Rk}\label{Rk:FreeCacti} For the first and third parts of the statement, it is sufficient to assume, instead of the unitary invariance of the $\mbf a_j$'s that for any test graph $T$ with no cutting edge, $\tau^0[T]=0$ whenever $T$ is not a cactus. 
\end{Rk}

A proof of the proposition is given in \cite[Section 5.2]{Male2011} based on the property of unitary invariance (Definition \ref{Def:UI}). For completeness, we give a proof using the cactus property. 

\begin{proof} 1. Let $T \in \mcal T \langle \bigsqcup_j \mbf  a_j\rangle$. Under the assumptions of the proposition, we can write, using \emph{w.o.} as a shortcut for \emph{well oriented},
	\eq
		\tau^0[T] & = &\one\big( \mcal G \mcal C \mcal C(T) \mrm{ \ is \ a  \ tree} \big) \prod_{S \in \mcal C \mcal C (T)} \one(S \mrm{ \ w.o. \ cactus}) \prod_{ C \mrm{ \ cycle \ of \ } S} \tau^0[C].
	\qe
Let us say that a cactus $T$ in variables $\mbf a =  \bigsqcup_j \mbf a_j$ is well colored (in short w.c.) whenever each cycle of $T$ is labeled by variables in a same family $\mbf a_j$. Note that $T$ is well colored if and only if $\sigma$ is a non-mixing non crossing partitions, that is each of its blocks contain variables in a same family $\mbf a_j$. Since the graph of colored components of a cactus is a tree if and only if it is well colored, we then get
	\eqa\label{Eq:ProofLinkFree}
		\tau^0[T] & =&  \one ( T \mrm{ \ w.o.w.c. \ cactus})  \prod_{ C \mrm{ \ cycle \ of \ } S} \tau^0[C].
	\qea
Moreover, let $C \in \mcal T \langle \bigsqcup_j \mbf  a_j\rangle$ be a simple cycle with consecutive edges $(\cdot \overset{a_{i_1}}\leftarrow \cdot \dots \cdot \overset{a_{i_n}}\leftarrow \cdot)$, where the $a_i$ are elements of the $\mbf a_j$'s. The above formula yields
	\eq
		\Phi(a_{i_1} \dots a_{i_n}) = \tau[C] =  \sum_{\substack{\sigma \in \mrm{NC}(n) \\ \mrm{non-mixing}}}    \prod_{ \{i_1< \dots < i_\ell\} \in \sigma} \kappa_\ell(a_{i_1} \etc a_{i_\ell}),
	\qe
which characterizes free variables. Moreover, this implies the correspondance between injective traces of well-oriented cycles and free cumulants. Hence, coming back to Equation \eqref{Eq:ProofLinkFree} for general $T$ we can write
	\eqa\label{Eq:ProofLinkFree2}
		\tau^0[T] & =&  \one ( T \mrm{ \ w.o. \ cactus})  \prod_{ C \mrm{ \ cycle \ of \ } S} \tau^0[C].
	\qea
 since for test graphs $T$ that are not well colored, there are mixed cumulants along some cycles. Hence $\mbf a$ is of cactus type, so it is unitarily invariant.

2. Reciprocally, let us assume that $\mbf a$ is unitarily invariant and that the $\mbf a_j$'s are free independent. Let us prove that they are traffic independent. Since $\mbf a$ is of cactus type, for any test graph $T \in \mcal T\langle \bigsqcup_j \mbf a_j \rangle$, Equation \eqref{Eq:ProofLinkFree2} is satisfied. Freeness of the $\mbf a_j$'s implies vanishing of mixed cumulants, so that $\tau^0[C] = 0$ for some cycle if $T$ is not a well colored cactus. But $T$ is a w.o.w.c. cactus if and only $\mcal G \mcal C \mcal C(T)$ is a tree and the colored components are cacti. This yields the formula \eqref{Eq:ProofLinkFree} and by the above computation that the $\mbf a_j$'s are traffic independent.

3. Let now $\mbf a$ be a unitarily invariant family of traffics independent of an arbitrary family $\mbf b$, and let us prove that $\mbf a$ and $\mbf b$ are free independent in $(\mcal A, \Phi)$. Without loss of generality, we can assume that the families of matrices contain the identity. By \cite[Theorem 14.4]{NS}, it suffices to prove that for any $a_1\etc a_n $ in $\mbf a$ and any $b_1\etc b_n$ in $\mbf b$, the following is satisfied
	$$\Phi(a_1 b_1 \dots a_n b_n) = \sum_{\sigma \in \mrm{NC}(n)} \kappa_\sigma(a_1\etc a_n) \times \Phi_{\hat \sigma} (b_1\etc b_n),$$
where $\hat \sigma$ is the Kreweras complement of $\sigma$, that is the largest non crossing partition of $\{1' \etc n'\}$ such that $\sigma \sqcup \hat \sigma$ is a non crossing partition of $\{1,1' \etc n, n'\}$, and 
	$$\kappa_\sigma(a_1\etc a_n) = \prod_{\{i_1< \dots < i_\ell\} \in \sigma} \kappa_\ell(a_{i_1} \etc a_{i_\ell}),$$
with a similar definition for $\Phi_{\hat \sigma}$.

Let $T$ be a simple cycle with consecutive edges $(\cdot \overset{a_1}\leftarrow \cdot \overset{b_1}\leftarrow \cdot \dots\cdot \overset{a_n}\leftarrow  \cdot \overset{b_n}\leftarrow \cdot)$. Then by definition of traffic independence and the cactus property of $\mbf a$, denoting by $V$ the vertex set of $T$ one has
	\eq
		\lefteqn{\Phi(a_1 b_1  \dots  a_n b_n) = \tau[ T]}\\
		 & = & \sum_{\pi\in \mcal P(V)} \one\big( \mcal G \mcal C \mcal C(T^\pi)\mrm{ \ is \ a \ tree}\big) \\
		 & & \ \ \ \times \Bigg( \prod_{S\in \mcal C \mcal C_\mbf b (T^\pi)} \tau^0[S] \times   \prod_{S\in \mcal C \mcal C_\mbf a (T^\pi)} \Big(\one( S \mrm{ \ w.o. \ cactus}) \prod_{C \mrm{ \ cycle \ of \ } S} \tau^0[C] \Big) \Bigg),
	\qe
where $\mcal C \mcal C_\mbf a(T)$ is the set of colored components of $T$ labeled in $\mbf a$, and $\mcal C \mcal C_\mbf b(T)$ is defined similarly. 

The arguments of the proof are those used in \cite{Male122} (replacing the so-called \emph{fat trees} by the cacti). Given $\pi \in  \mcal P(V)$, denote by $S_{\mbf a, \pi}$ the graph obtained from $T^\pi$ by identifying the source and target of each edge labeled in $\mbf b$ and suppressing these edges. If $\mcal G \mcal C \mcal C(T^\pi)$ is a tree and $\mcal C \mcal C_\mbf a$ is a set of cacti, then $S_{\mbf a, \pi}$ is a cactus. By Lemma \ref{Lem:CactiNCP}, $\pi$ induces a non crossing partition $\sigma_{\mbf a, \pi}$ of the set $E_\mbf a$ of edges of $T$ labeled $\mbf a$, whose blocks are associated to variables labeled $\mbf a$ in a same cycle of $T$ (the cyclic order of $E_\mbf a$ is the one around the cycle $T$). 

Reciprocally, consider a non crossing partition $\sigma_{\mbf a}$ of $E_\mbf a$ and then a cactus $S(\sigma_{\mbf a})$ labeled $\mbf a$. Let $\sigma_{\mbf b} = \hat \sigma_{\mbf a}$ be the Kreweras complement of $\sigma_{\mbf a}$, which is a partition of the set $E_{\mbf b}$ of edges of $T$ labeled $\mbf b$, and let $S(\sigma_\mbf b)$ denotes the cactus associated to $\sigma_{\mbf b}$. Once more we consider the Kreweras complement of $\sigma = \sigma_{\mbf a} \sqcup \sigma_{\mbf b}$, which is now a partition $\pi_0\in\mcal P(V)$ of the vertex set of $T$. By Lemma \ref{Lem:CactiNCP}, $T^{\pi_0}$ and $S(\sigma \mbf b)$ are cacti. Moreover, the partitions $\pi \in \mcal P(V)$ such that $\mcal G \mcal C \mcal C(T^\pi)$ is a tree, $\mcal C \mcal C_\mbf a$ is a set of cacti and $S_{\mbf a, \pi} = S(\sigma_\mbf a)$ are those that only identifies vertices in a same cycle of $T^{\pi_0}$ labeled $\mbf b$, which are the cycles of $S(\sigma_\mbf b)$, see Figure \ref{Fig:13}. Then we have, using that $\mbf a$ is of cactus type and the definition of $\tau^0$ in the second line,
	\eq
		\tau[ T]  & = & \sum_{\sigma_{\mbf a} \in NC(E_{\mbf a})} \prod_{C_{\mbf a} \mrm{ \ cycle \ of \ } S_{\sigma_\mbf a} }\tau^0[C_{\mbf a}] \times \prod_{C_{\mbf b} \mrm{ \ cycle \ of \ } S_{\hat\sigma_\mbf a} } \ \sum_{\pi \in \mcal P\big(V(C_{\mbf b})\big) }\tau^0[C_\mbf b^\pi]\\
		& = & \sum_{\sigma_{\mbf a} \in NC(E_{\mbf a})} \prod_{C_{\mbf a} \mrm{ \ cycle \ of \ } S(\sigma_\mbf a) }\kappa(C_{\mbf a}) \times \prod_{C_{\mbf b} \mrm{ \ cycle \ of \ } S(\hat \sigma_\mbf a)} \Phi(C_\mbf b),
	\qe
where $V(C_{\mbf b})$ denoting the vertex set of $C_{\mbf b}$, $\kappa(C)$ means the free cumulants $\kappa(x_1\etc x_\ell)$ for a cycle with consecutive edges $(x_1\etc x_\ell)$, and $\Phi(C)$ is defined similarly. With this notation, this is the desired formula.
\end{proof}

  \begin{figure}[h!]
    \begin{center}
     \includegraphics[width=120mm]{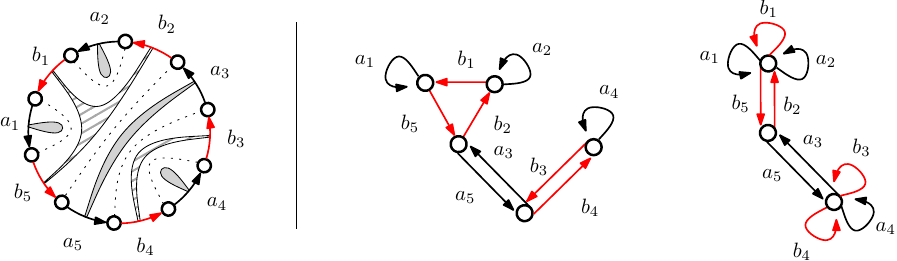}
    \end{center}
    \caption{Left: the cycle $T$ with a non crossing partition $\sigma_\mbf a$ (full grey blocks), its kreweras complement $\sigma_\mbf b$ (striped grey blocks), and the Kreweras complement $\pi_0$ of $\sigma_\mbf a \sqcup \sigma_\mbf b$ (dotted lines). Center and right: the quotient graph $T^{\pi_0}$, and another quotient graphs $T^\pi$ such that $S_\pi$ is the cactus of $\sigma_\mbf a$.}
    \label{Fig:13}
  \end{figure}

\subsubsection{Asymptotic freeness of random matrices}

The previous proposition implies a universal property of free independence for asymptotically unitarily invariant matrices. 

\begin{Cor}\label{Cor:Freeness1} Let $\mbf A_j^{(N)}$, $j\in J$, be independent families of random matrices such for each $j\in J$, \begin{enumerate}
	\item [(H0)]  $U\mbf A_j^{(N)}U^*$ has the same law as $\mbf A_j^{(N)}$ for any permutation matrices $U$.
	\item [(H1)] $\mbf A_j^{(N)}$ converges in traffic distribution to a unitarily invariant family traffics.
	\item [(H2)]  $\mbf A_j^{(N)}$ satisfies the factorization property \eqref{Eq:Facto}.
\end{enumerate}
Then the $\mbf A_j^{(N)}$'s are asymptotically freely independent with respect to $\esp\big[ \frac 1 N \Tr \big]$ and $\cup_j \mbf A_j^{(N)}$ is asymptotically freely independent from any auxiliary independent family of random matrices converging in traffic distribution and satisfying (H2).
\end{Cor}

The \emph{universal} aspect of this statement is that it holds for any auxiliary matrices, without assumptions on the form of their limiting traffic distribution.

\begin{proof} The three first assumptions implies the asymptotic traffic independence by  \cite{Male2011}, so the corollary follows directly from Proposition \ref{Prop:FreenessFreeTraffics}.
\end{proof}

One can work under a slightly weaker assumption than the convergence in traffics distribution of the matrices, since the conclusion is about non-commutative distribution. This is allowed by the modification of the asymptotic traffic independence theorem of \cite{Male122}. Let us say that a test graph $T$ is \emph{cyclic} if there exists a cycle visiting each edge once in the sense of the orientation. For a family $\mbf B_N$ of matrices, we denote by $\| \mbf B_N\|$ the supremum of the operation norm (square-root of the largest singular value) of the matrices of $\mbf B_N$.

\begin{Cor} Let $\mbf A_j^{(N)}$, $j\in J$, be independent families of random matrices such for each $j\in J$, $\mbf A_j^{(N)}$ satisfies (H0) and the following modifications of the previous hypotheses:
\begin{enumerate}
	\item [(H1')] $\esp\big[ \frac 1 N \Tr \, T(\mbf A_N)\big]$ converges for any cyclic test graph $T$ and the limit satisfies the cactus formula.
	\item [(H2')]  $\mbf A_N$ satisfies the factorization property on cyclic test graphs, and furthermore it satisfies the tightness condition of \cite{Male122}, for instance $\|\mbf A_N\|$ is uniformly bounded as $N$ goes to infinity.
\end{enumerate}
Then the $\mbf A_j^{(N)}$'s are asymptotically free independent and $\cup_j \mbf A_j^{(N)}$ is asymptotically free independent from any independent family of random matrices converging in traffic distribution on cyclic test graphs and satisfying (H2').
\end{Cor}

For instance, a normalized adjacency matrix $\mbf A_N$ of a regular large graph with large degree may converges to a unitarily invariant traffics $  a$ on cyclic graphs (see \cite{MP14}). It cannot converges to $  a$ on all test graphs since $deg(\mbf A_N)$ is a constant matrix whereas $deg(  a)$ is a non trivial random variable for a nonzero unitarily invariant traffic $a$.

\begin{proof} The first assumptions implies the asymptotic traffic independence when distribution are restricted to cyclic graphs by  \cite{Male122}. Computation of trace of cyclic test graphs involves only computation of injective trace of cyclic graphs and reciprocally. Moreover the trace depends only on combinatorial traces of such graphs. Hence all the computation of the section is valid with this restriction.
\end{proof}

\section{Equivalence between unitary invariance and cactus type}\label{Sec:UISCMC}

This section is dedicated to the proof of Proposition \ref{EqCactusUI}.

\subsection{On the geometry of cacti}

 \begin{Def}\label{Def:Cutt}
 \begin{itemize}
 	\item  A cutting edge of a finite graph is an edge whose removal increases the number of connected component. A two-edge connected (t.e.c.) graph is a connected graph with no cutting edges.
	\item The cut number between two vertices is the minimal number of edges whose removal separate them. 
	\item  Two vertices of a graph form a 3-connection whenever there exist three edge-distinct paths joining them.
\end{itemize}
\end{Def}

We will use Menger's theorem \cite{Menger1927}:

\begin{Th}\label{Th:Menger} Let $v$ and $w$ two distinct vertices of a connected graph. The cut number between $v$ and $w$ is equal to the maximum number of edge-disjoint paths from $v$ to $w$.
\end{Th}

In particular, a 3-connection consists in vertices with cut number at least three. A t.e.c. graph is a graph whose vertices have cutting numbers at least two. We can then deduce the following characterization of cacti. 

\begin{Prop}\label{Prop:2CutDist} A finite graph is a cactus if and only if the cut number between two vertices is constant, equal to two.
\end{Prop}

\begin{proof} Let $T$ be a finite graph with cut number constant to two. It is connected since the cut number is finite. There is no vertices $v$ and $w$ with cut number equal to one, so every edge $e=(v,w)$ is contained in a simple cycle. Moreover, if an edge $e$ of a graph belongs to more than two distinct simple cycles, the union of these cycles with $e$ remove is still t.e.c. so one can find a 3-connection in the graph. Hence $T$ is a cactus.

Let now $T$ be a cactus. The cut number between two vertices is greater than one since the graph is connected and each edge belong to a cycle. Moreover, the cut number between two vertices $v$ and $w$ is alway two. Indeed, consider a simple path between $v$ and $w$ and remove one of its edges. In the cycle with this edge removed, we can remove an edge to separate $v$ and $w$: the two vertices cannot be in a same connected component of the graph outside this cycle because the path between them is simple.
\end{proof}

Let $T$ be a test graph and let $\pi$ be a partition of its vertices. Since edge-disjoint paths on $T$ induce edge-disjoint paths on the quotient graph $T^\pi$, the cut number of two vertices $v$ and $w$ in $T$ cannot decrease if $v$ and $w$ are not identified in $T^\pi$. This implies the following lemma.

\begin{Lem}\label{Lem:3co} Let $T$ be a connected finite graph, let two vertices $v,w$ forming a 3-connection, and let $\pi$ a partition of the vertex set of $T$. If the quotient graph $T^\pi$ is a cactus then $v\sim_\pi w$.
\end{Lem}

We now deduce from this lemma three properties characterizing unitarily invariant traffics that we use in next section.

\begin{Cor}\label{Lem:Fact1} Let $\mbf a$ be a  family of traffics of cactus type in an algebraic traffic space $(\mcal A, \tau)$. Let $T$ be a test graph labeled in $\mbf a$ with two vertices $v$ and $w$ forming a 3-connection. Let $T_{v\sim w}$ be the test graph obtained by identifying $v$ and $w$ in $T$. Then one has $ \tau[T]= \tau[T_{v\sim w}]$.
\end{Cor}

In particular, by iterating this procedure, we get that $\tau[T] = \tau[\tilde T]$ where $\tilde T$ is obtained by identifying all pairs of 3-connections. The cut number of pairs of vertices of $\tilde T$ is always smaller than or equal to two.

\begin{proof} If a quotient graph $T^\pi$ of $T$ is a cactus then $v\sim w$. Since $\tau^0$ is supported on cacti Lemma \ref{Lem:3co} implies
	$$\tau[ T]  = \sum_{  \pi \in \mcal P(V)  } \tau^0\big[ T^\pi\big] =\sum_{\substack{ \pi \in \mcal P(V) \\ v \sim_\pi w}} \tau^0\big[ T^\pi\big] = \tau[T_{v\sim w}].$$
\end{proof}
 
\begin{Cor}\label{Lem:Fact2}  Let $\mbf a$ be a family of traffics of cactus type in an algebraic traffic space $(\mcal A, \tau)$. Let $T$ be a test graph that can be obtained by identifying two vertices of test graphs $S$ and $S'$, where $S$ is t.e.c. Then
	$$\tau[T] = \tau[S] \times \tau[S'].$$
\end{Cor}
In particular, by iterating this procedure, we get that if $T$ is a cactus then 
	$$\tau[T ] = \prod_{ C \mrm{ \ cycle \ of \ } T} \tau[ C].$$

\begin{proof} 
Denote by $o$ the vertex of $T$ that belong both to $S$ and $S'$. Let $v$ (resp. $v'$) be a vertex of $S$ (resp. $S'$), seen in $T$ and different from $o$. Let $\pi$ be a partition of $T$ such that $v\sim_\pi v'$ and $T^\pi$ is a cactus. Then $T^\pi$ is a quotient of $T_{v\sim_\pi v'}$ for which $(v,o)$ forms a 3-connection and so by Lemma \ref{Lem:3co} one has $v\sim_\pi o \sim_\pi v'$. Hence, each partition $\pi$ such that $T^\pi$ is a cactus is the union $\pi = \sigma \cup \sigma'$ of a partition $\sigma$ of the vertices of $S$ and a partition $\sigma'$ of those of $S'$. For such a partition $\pi = \sigma \cup \sigma'$, by definition of cactus type traffics we have $\tau^0[T^\pi] = \tau^0[S^\sigma] \times \tau^0[(S')^{\sigma'}]$. Hence we get, denoting by $V_S$ and $V_{S'}$ the vertex sets of $S$ and $S'$ respectively, 
	$$\tau[T]  = \sum_{\sigma\in \mcal P(V_S)} \tau^0[S^\sigma] \times \sum_{\sigma'\in \mcal P(V_S)} \tau^0[(S')^{\sigma'}] = \tau[S] \times \tau[S'].$$ 
\end{proof}

It remains to show how to handle test graphs with cutting edges.

\begin{Lem}\label{Lem:Fact3} Let $\mbf a$ be a  family of traffics of cactus type in an algebraic traffic space $(\mcal A, \tau)$. Let $T$ be a test graph labeled in $\mbf a$ and denote by $\mcal O$ the set of vertices of $T$ with odd degree (the degree is the number of neighbors, here we forget the orientation of the edges). For any partition $\sigma$ of $\mcal O$, let us denote
	$$p_\sigma(T) = \sum_{\sigma'\geq \sigma} \mrm{Mob}_{\mcal P(\mcal O)}( \sigma, \sigma') T^{\sigma'},$$
where $\mrm{Mob}_{\mcal P(\mcal O)}$ denotes the M\"obius function of the poset of partitions of $\mcal O$ and $T^{\sigma'}$ the graph obtained by identifying vertices in a same block of $T$. Then one has
	$$\tau[T ] = \sum_{\substack{\sigma \in \mcal P(\mcal O) \\ |B| \mrm{ \ is \ even } \, \forall B \in \sigma}} \tau\big[ p_\sigma(T)\big].$$
\end{Lem}

In particular, we get that $\tau[T ]$ can be written as a linear combination of $\tau[S]$ where $S$ has no cutting edges. Since the cutting number can always increase when taking quotients, together with the above lemmas, one gets an expression of $\tau[T]$ in terms of linear combinations of products of $\tau[C]$ where $C$ are simple cycles.

\begin{proof} Cacti have only vertices of even degree. Hence, if $\pi$ is a partition such that $T$ is a cactus then it must re-group the vertices in $\mcal O$ in blocks of even size. Hence
	$$\tau[T] = \sum_{\substack{\sigma \in \mcal P(\mcal O) \\ |B| \mrm{ \ is \ even } \, \forall B \in \sigma}} \sum_{\substack{ \pi \in \mcal P(V) \\ \mrm{s.t. \ }\pi_{\mcal O} = \sigma}} \tau^0[T^\pi].$$
By the same proof as Lemma \ref{Def:Solid} applied to $T^\sigma$, the second sum in nothing else than $\tau\big[ p_\sigma(T)\big]$.
\end{proof}

\subsection{Proof of the equivalence}

Let $\mbf a = (a_j)_{j\in J}$ be an arbitrary family of traffics, independent from the limit $(u, u^*)$ of a Haar unitary matrix and its conjugate, and denote $\mbf b = (ua_ju^*)_{j\in J}$. We shall prove that $\mbf a$ and $\mbf b$ have the same traffic distribution.

For a test graph $T$ labeled in $\mbf b$ we denote by $\tilde T$ the graph labeled in $\mbf a, u, u^*$ obtained by replacing each edge $(\cdot \overset{ b}\leftarrow \cdot)$ by the sequence of edges $(\cdot \overset{ u}\leftarrow\cdot \overset{ a}\leftarrow\cdot \overset{ u^*}\leftarrow \cdot)$ and by $\tilde V$ the vertex set of $\tilde T$. In this section, we say that a partition $\pi$ of $\tilde V$ is \emph{valid} whenever $ \mcal G \mcal C \mcal C(\tilde T^\pi)$ is a tree and the colored components of $\tilde T^\pi$ labeled in $(u,u^*)$ are well oriented cacti whose edges along each cycle alternate between $u$ and $u^*$.

Lemma \ref{Lem:3co} is replaced by the following.
 
\begin{Lem} If $\pi$ is a valid partition then for any 3-connection $(v,w)$ of $\tilde T$ one has $v\sim_\pi w$.
\end{Lem}

\begin{proof} Let $(v,w)$ be such a pair in $\tilde T$ and $\pi$ a valid partition. Assume moreover $v \not \sim_\pi w$ and let us find a contradiction. Let $S_1 \etc S_n$ be the path in $\mcal G \mcal C \mcal C(\tilde T)$ between the colored components $S_1$ and $S_n$ containing $v$ and $w$ respectively. If $n\geq 2$, then one of these components is labeled in $(u,u^*)$ and it is traversed by at least three edge disjoint paths, so it cannot be a cactus. If $n=1$, then $\tilde T^\pi$ has a cycle with an odd number of variables in $u$ and $u^*$, so there are colored components labeled in $(u,u^*)$ which are not cacti whose edges are labeled alternatively by $u$ and $u^*$.
\end{proof}

We now prove that the tree properties stated in Corollary \ref{Lem:Fact1},  Corollary \ref{Lem:Fact2} and Lemma \ref{Lem:Fact3} hold for $\mbf b$. By independence of $\mbf a$ and $(u,u^*)$ and by the formula for the traffic distribution of $(u,u^*)$, we have
	\eq
		\tau[T] = \tau[\tilde T] & = &  \sum_{\pi \in \mcal P(\tilde V)} \one\big( \mcal G \mcal C \mcal C(\tilde T^\pi) \mrm{ \ is \ a \ tree}) \prod_{ S \in \mcal C \mcal C_{\mbf a}(\tilde T^\pi) } \tau^0[ S] \\
		& & \ \ \ \times \prod_{ S \in \mcal C \mcal C_{(u,u^*)}(\tilde T^\pi) }\one( S \mrm{\ w.o. \ cactus}) \prod_{C \mrm{ \ cycle \ of \ } S} \tau^0[S],
\qe
where $\mcal C \mcal C_{\mbf a}$ and $\mcal C \mcal C_{(u,u^*)}$ denote the set of colored components labeled in $\mbf a$ and $(u,u^*)$ respectively. 

The 3-connections of $ T$ correspond to those of $\tilde T$. Hence for any such pair $(v,w)$ in $T$, we get $\tau[T] = \tau[\tilde T] = \tau[\tilde T_{v,w}] = \tau[T_{v,w}]$, so the first property is clear. The proof of the second property is similar: if $T$ is a t.e.c. graph that can be obtained by identifying two vertices of disjoint test graphs $S$ and $S'$, then $\tilde T$ can be obtained by identifying two vertices $\tilde S$ and $\tilde S'$ and the proof is unchanged, using the above lemma instead of of Lemma \ref{Lem:3co}.

 Let now $T$ be an arbitrary test graph and denote by $\mcal O$ the set of vertices of odd degree. Corresponds in $\tilde T$ a set $\tilde {\mcal O}$. Moreover, denote by $\tilde {\mcal O}'$ the set of vertices of $\tilde T$ both adjacent to an edge labeled in $\mbf a$ and in $(u,u^*)$. If a partition $\pi$ of the vertices of $\tilde T$ is valid, it must regroup the vertices of $\tilde {\mcal O} \cup \tilde {\mcal O}'$ in blocks of even size (since vertices of considered cacti are of even degree). But if $\pi$ identifies a vertex of $\tilde {\mcal O}$ and a vertex of $\tilde {\mcal O}'$, then $T^\pi$ has a cycle with an odd number of edges in $(u,u^*)$, so $\pi$ is not valid. We hence get
	\eq
		\tau[\tilde T ] & = & \sum_{ \substack{\sigma \in \mcal P( \tilde {\mcal O}) \\ |B| \mrm{ \ even \ } \forall B \in \sigma} } \tau\big[ p_\sigma[\tilde T] \big]  =  \sum_{ \substack {\sigma \in \mcal P(   {\mcal O}) \\ |B| \mrm{ \ even \ } \forall B \in \sigma}} \tau\big[ p_\sigma[  T] \big].
 	\qe
	
Hence $\tau[T]$ has the same expression as if labels were in $\mbf a$. Since for a simple cycle $C$ labeled $b_{j_1} \etc b_{j_n}$ one has $\tau[S] = \Phi( b_{j_1} \etc b_{j_n}) = \Phi(a_{j_1} \etc a_{j_n} ) $, we get as expected that $\mbf a$ and $\mbf b$ have the same traffic distribution.

\section{Asymptotically unitarily invariant random matrices}\label{Sec:AsymUImatrices}

\label{subsec:conv}

The purpose of this section is to prove Theorems \ref{Th:Matrices} and \ref{Th:MatricesBis}. Namely. for any unitarily invariant families of matrices $\mbf X_N$ satisfying the assumptions, for any test graphs $T_1 \etc T_m$, 
$$\tau_ {\mbf X_N}(T_1\etc T_m):= \frac 1 {N^m} \esp\Big[ \prod_{i=1}^m \Tr \,  T_i(\mbf X_N) \Big]$$
converges to 
	\eqa\label{Eq:Matrices}
		\prod_{i=1}^m \sum_{\pi \in \mcal P(V_i)} \Bigg(  \one( T_i^{\pi} \mrm{ \ well \ oriented \ cactus}) \prod_{C\in \mrm{Cycle}(T_i^{\pi})}\Nlim \Big( \tau^0_{\mbf X_N}[C] \Big) \Bigg),
	\qea
where $V_i$ denotes the vertex set of $T_i$.

Before reviewing some results about the free cumulants, some results about the Weingarten function, and the links between those two objects in large dimension, let us mention two applications of this result.

\subsection{Applications}

\begin{Lem}\label{Lem:FreeofTrans} If $\mbf A_N$ is a family of matrices converging in traffic distribution to a unitarily invariant family, then $\mbf A_N, \mbf A_N^t$ and $\big(\mrm{deg}(\mbf A_N),\mrm{deg}(\mbf A_N^t)\big)$ are asymptotically freely independent.
\end{Lem}

This generalize a recent result of Mingo and Popa~\cite{MingoPopa2014} stating the asymptotic free independence of $\mbf A_N$ and $\mbf A_N^t$ for unitarily invariant matrices. There we only assume that unitary invariance holds asymptotically.

\begin{proof}Let $(\mcal A, \tau)$ be an algebraic traffic space with trace $\Phi$ and let $\mbf a = (a_j)_{j\in J}$ be a unitarily invariant family of traffics. It is sufficient to prove that the families $\mbf a$, $\mbf a^t = (a_j^t)_{j\in J}$ and $\big(\mrm{deg}(\mbf a),\mrm{deg}(\mbf a^t)\big) =\big( \mrm{deg}( a_j), \mrm{deg}( a_j^t)\big)$ are free independent in $(\mcal A, \Phi)$.

We first prove that $\mbf a $ and $\mbf a^t $ are free. Let us consider $2n$ elements $c_1,\ldots, c_{2n}$ alternatively in $\C\langle \cdot \overset{a}{\leftarrow} \cdot:a\in \mathcal{A}\rangle$ and $\C\langle \cdot \overset{a}{\rightarrow} \cdot:a\in \mathcal{A}\rangle$ such that $\tau_\Phi(c_1)=\ldots=\tau_\Phi(c_{2n})=0$. We want to prove that $\tau_\Phi(\Delta(c_1\ldots c_{2n}))=0$. Using Proposition~\ref{interestingproposition} in order to regroup consecutive edges which are oriented in the same direction, we can assume that the $c_i's$ are written as $\cdot \overset{a_i}{\leftarrow} \cdot$ with $a_i\in \mathcal{A}$ such that $\Phi(a_i)=0$, and $c_{i}$ and $c_{i+1}$ not oriented in the same direction.

Consider now a partition $\pi$ such that $\tau_\Phi^0(\Delta (c_1\ldots c_{2n})^\pi)\neq 0$. Then, take a leaf of the oriented cactus $\Delta (c_1\ldots c_{2n})^\pi$. This leaf is a cycle of only one edge, because if not, the cycle cannot be oriented, since two consecutive edges in $\Delta (c_1\ldots c_{2n})$ are not oriented in the same way. This produces a term $\tau_\Phi^0(\Delta(c_{i}))=0$ in the product $\tau_\Phi^0(\Delta (c_1\ldots c_{2n})^\pi)$, which leads at the end to a vanishing contribution. Finally, $\tau_\Phi(c_1\ldots c_{2n})=0$ and we have the freeness wanted.

Now, let us prove that $\C\langle \ ^{\uparrow a}_{\cdot}\  :a\in \mathcal{A}\rangle$ is free from $\C\langle \cdot \overset{a}{\leftarrow} \cdot, \cdot \overset{a}{\rightarrow} \cdot:a\in \mathcal{A}\rangle$. By the same argument as above, we can consider that we have a cycle $\Delta(c_1\ldots c_{n})$ which consists in an alternating sequence of $c_i's$ written as $\cdot \overset{a_i}{\leftarrow} \cdot$ with $a_i\in \mathcal{A}$ such that $\Phi(a_i)=0$, $\cdot \overset{a_i}{\rightarrow} \cdot$ with $a_i\in \mathcal{A}$ such that $\Phi(a_i)=0$, and $c_i \in \C\langle \ ^{\uparrow a}_{\cdot}\  :a\in \mathcal{A}\rangle$ such that $\tau_\Phi(c_i)=0$. We want to prove that $\tau_\Phi(\Delta(c_1\ldots c_{n}))=0$. If there is no term $c_i\in \C\langle \ ^{\uparrow a}_{\cdot}\  :a\in \mathcal{A}\rangle$, we are in the case of the previous paragraph. Let us assume that there exists at least one such term, say $c_1$. By linearity, we can consider that the term $c_1\in \C\langle \ ^{\uparrow a}_{\cdot}\  :a\in \mathcal{A}\rangle$ is written as $^{\uparrow b_1}_{\cdot}\cdots ^{\uparrow b_k}_{\cdot}-\tau_\Phi(^{\uparrow b_1}_{\cdot}\cdots ^{\uparrow b_k}_{\cdot})$, where $^{\uparrow b_1}_{\cdot}\cdots ^{\uparrow b_k}_{\cdot}$ is some vertex input/output from which start $k$ edges labelled by $b_1,\ldots, b_k \in \mathcal{A}$. Let us prove that $\tau_\Phi(\Delta((^{\uparrow b_1}_{\cdot}\cdots ^{\uparrow b_k}_{\cdot})c_2\ldots c_{n}))$ and $\tau_\Phi(^{\uparrow b_1}_{\cdot}\cdots ^{\uparrow b_k}_{\cdot})\tau_\Phi(\Delta(c_2\ldots c_{n}))$ are equal, which implies by linearity that $\tau_\Phi(\Delta(c_1\ldots c_{n}))=0$. Decomposing into injective trace, we are left to prove that for all partition $\pi$ of the vertices of $\Delta((^{\uparrow b_1}_{\cdot}\cdots ^{\uparrow b_k}_{\cdot})c_2\ldots c_{n})$ which do not respect the blocks $(^{\uparrow b_1}_{\cdot}\cdots ^{\uparrow b_k}_{\cdot})$ and $\Delta (c_2\ldots c_{n})$, $\tau_\Phi^0(\Delta((^{\uparrow b_1}_{\cdot}\cdots ^{\uparrow b_k}_{\cdot})c_2\ldots c_{n})^\pi)=0$. The same argument as previous paragraph works again. If one of the vertex of $(^{\uparrow b_1}_{\cdot}\cdots ^{\uparrow b_k}_{\cdot})$ is identified by $\pi$ with one of the vertex of $\Delta (c_2\ldots c_{n})$, and $\Delta((^{\uparrow b_1}_{\cdot}\cdots ^{\uparrow b_k}_{\cdot})c_2\ldots c_{n})^\pi$ is a cactus there exists a cycle not oriented or a leaf labelled by one $a_i$, which leads to a vanishing contribution.
\end{proof}

\begin{Lem} Let $\mbf A^{(N)}_1 \etc \mbf A^{(N)}_L$ (resp. $\mbf B^{(M)}_1 \etc \mbf B^{(M)}_L$) be independent families of $N$ by $N$ (resp. $M$ by $M$) random matrices, that converge in traffic distribution to unitarily invariant variables and satisfy the factorization property as $N$ (resp. $M$) goes to infinity. Let $U_1 \etc U_L$ be independent uniform permutation matrices of size $N\times M$. Assume that $(\mbf A^{(N)}_\ell)_{\ell=1\etc L}$, $(\mbf B^{(M)}_\ell)_{\ell=1\etc L}$ and $(U_\ell)_{\ell=1\etc L}$ are independent. Then the families $U_1 (\mbf A^{(N)}_1 \otimes \mbf B^{(M)}_1) U_1 \etc U_L (\mbf A^{(N)}_L \otimes \mbf B^{(M)}_L) U_L$ are asymptotically freely independent with respect to $\frac 1 {NM} \Tr$. If moreover $\mbf C_{NM}$ is a family of $NM$ by $NM$ random matrices that converge in traffic distribution and satisfies the factorization property, then it is asymptotically free independent from the previous families.
\end{Lem}

\begin{proof} For each $\ell=1\etc L$, the family of matrices $\mbf A^{(N)}_\ell \otimes \mbf B^{(M)}_\ell$ converges in traffic distribution to the tensor product $\mbf a_\ell \otimes \mbf b_\ell$ of the limits of each factor, and it satisfy the factorization property. Hence by the asymptotic traffic independence theorem, the families of matrices $U_\ell (\mbf A^{(N)}_\ell \otimes \mbf B^{(M)}_\ell) U_\ell$, $\ell=1\etc L$, are asymptotically traffic independent.

On the other hand, let us compute the limiting distribution of each family. Let $T$ be a test graph in $\mcal T\langle \mbf a_\ell\otimes \mbf b_\ell \rangle$. Assume that it has not cutting edge, which is sufficient to characterize the limiting $^*$-distribution. Recall that we denote by $\Lambda_T$ the set of pairs $(\pi_1,\pi_2) \in \mcal P(V)^2$ such that if two elements belong to a same block of $\pi_i$ then they belong to different blocks of $\pi_j$, $i\neq j \in \{1,2\}$. We have by Lemma \ref{Lem:TensorInj}
	\eq
		(\tau_1\otimes \tau_2)^0[ T] 
		& = & \sum_{(\pi_1,\pi_2) \in \Lambda_T} \tau^0\big[ T_1^{\pi_1}\big] \times \tau^0\big[ T_2^{\pi_2}\big],
	\qe
where $T_i$ is the graph whose edges are labeled the variables of the $i$-th factor.

If $(\pi_1,\pi_2)\in \Lambda_T$ contributes in the above term then $T^{\pi_1}$ and $T^{\pi_2}$ are cacti with oriented cycles. Since $T$ is t.e.c., it is a cactus if and only if it has a 3-connection. But if $T$ has a 3-connection, it must be identified to produce a cactus. Hence there is no $(\pi_1,\pi_2)$ in $\Lambda_T$ such that both $T_1^{\pi_1}$ and $T_2^{\pi_2}$ are cacti.

By Lemma \ref{Prop:FreenessFreeTraffics} and Remark \ref{Rk:FreeCacti}, we then get that the matrices are asymptotically freely independent. 
\end{proof}

\subsection{The Weingarten function.}We need to integrate against the $\UN$-Haar measure.  Expressions for these integrals appeared in \cite{Weingarten1978} and were first proven in \cite{Collins2004} and given in terms of a function on symmetric group called the Weingarten function. We recall here its definition  and some of its properties.  For any $n \in\N^*$ and  any permutation $\sigma\in \mcal S_n$,  let us set $$\Omega_{n,N}(\sigma)=N^{\# \sigma}, $$
where $\# \sigma$ is the number of cycles of $\sigma$.
When $n$ is fixed and  $N\to \infty,$  $N^{-n}\Omega_{n,N}\to  \delta_{\Id_n}$. For any pair of functions $f,g: \mcal S_n\to \C$ and $\pi \in \mcal S_n$, let us define the convolution product 
 $$f\star g (\sigma)= \sum_{\pi\preccurlyeq \sigma}f(\pi)g(\pi^{-1}\sigma),$$
 Hence, for $N$ large enough, $\Omega_{n,N}$ is invertible  in the algebra of function on $\mcal S_n$ endowed with convolution as a product. We denote by $\Wg_{n,N}$ the unique function on $\mcal S_n$  such that $$\Wg_{n,N}*\Omega_{n,N}= \Omega_{n,N}*\Wg_{n,N}=\delta_{\Id_n}. $$
Then, \cite[Corollary 2.4]{Collins2004} says that, for any indices $i_1,i'_1,j_1,j'_1\ldots,i_n,i'_n,j_n,j'_n\in \{1,\ldots, N\}$ and $U=\big(U(i,j)\big)_{i,j=1\etc N}$ a Haar distributed random matrix on $\UN$,
\begin{equation}
\label{IntegrationHaar} \esp[U(i_1,j_1)\ldots U(i_n,j_n)\overline{U}(i'_1,j'_1)\ldots \overline{U}(i'_n,j'_n)]= \sum_{\substack{\alpha,\beta\in \mcal S_n\\ i_{\alpha (k)}=i'_k, j_{\beta(k)}=j'_k}}\Wg_{n,N}(\alpha\beta^{-1}). 
\end{equation}

\subsection{Free cumulants and the M\"obius function $\mu$.}As explained in \cite{Biane1997}, it is equivalent to consider lattices of non-crossing partitions or sets of permutations endowed with an appropriate distance. For our purposes, it is more suitable to define the free cumulants using sets of permutations. Let us endow $\mcal S_n$ with the metric $d$, by setting for any $\alpha,\beta\in \mcal S_n$,  $$d(\alpha,\beta)= n- \#(\beta\alpha^{-1}),$$
where $\#(\beta\alpha^{-1})$ is the number of cycles of $\beta\alpha^{-1}$.
We endow the set $\mcal S_n$ with the partial order given by the relation $\sigma_1\preceq\sigma_2$ if $d(\Id_n,\sigma_1)+ d(\sigma_1,\sigma_2)=d(\Id_n,\sigma_2)$, or similarly if $\sigma_1$ is on a geodesic between $\Id_n$ and $\sigma_2$.

Given a state $\Phi: \C\langle x_j,x_j^*\rangle_{j\in J}\to \C,$ we define the free cumulants $(\kappa_n)_{n\in \N}$ recursively on $\C\langle x_j,x_j^*\rangle_{j\in J}$ by the system of equations: $\forall y_1,\ldots, y_n\in \C\langle x_j,x_j^*\rangle_{j\in J}$
	\eqa\label{Def:FreeCum}
		\Phi(y_1\cdots y_n)=\sum_{\sigma \preccurlyeq (1\cdots n)} \prod_{\substack{(c_1\ldots c_k)\\\text{ cycle of }\sigma}}\kappa(y_{c_1},\ldots, y_{c_k}).
	\qea
Let us fix $y_1,\ldots, y_n\in \C\langle x_j,x_j^*\rangle_{j\in J}$ and denote by respectively $\phi$ and $k$  the functions from $\mcal S_n$ to $\C$ given by
$$\phi(\alpha)= \prod_{\substack{(c_1\ldots c_k)\\\text{ cycle of }\sigma}}\Phi(y_{c_1}\ldots y_{c_k})\ \ \text{and}\ \ k(\alpha)= \prod_{\substack{(c_1\ldots c_k)\\\text{ cycle of }\sigma}}\kappa(y_{c_1},\ldots ,y_{c_k}),$$
which are such that $\phi((1\cdots n))=\sum_{\pi\preccurlyeq (1\cdots n)} k(\pi).$ In fact, we have more generally the relation $$\phi(\alpha)=\sum_{\pi\preccurlyeq \sigma} k(\pi).$$
Note that $\phi=k\star \zeta,$
where $\zeta$ is identically equal to one. The identically one function $\zeta$ is invertible for the convolution $\star$ (see \cite{Biane1997}), and its inverse $\mu$ is called M\"obius function. It allows us to express the free cumulants in terms of the trace:
\begin{equation}\label{cumtrace}k=\phi\star \mu.\end{equation}

\subsection{Asymptotics of the Weingarten function.}One can observe that, for any pair of functions $f,g: \mcal S_n\to \C$ and $\pi \in \mcal S_n$,
$$\sum_{\pi\in \mcal S_n}N^{d(\Id_n,\sigma)-d(\Id_n,\pi)-d(\pi,\sigma)}f(\pi)g(\pi^{-1}\sigma)=f\star g (\sigma)+o(1).$$
Defining the convolution $\star_N$ as
	\eq
		f\star_N g & = & N^{n}\Omega_{n,N}^{-1}((N^{-n}\Omega_{n,N} f)*(N^{-n}\Omega_{n,N} g))\\
		& = & \sum_{\pi\in \mcal S_n}N^{d(\Id_n,\sigma)-d(\Id_n,\pi)-d(\pi,\sigma)}f(\pi)g(\pi^{-1}\sigma),
\qe
it follows that $\star$ is the limit of $\star_N$. Because $\Wg_{n,N}$ is the inverse of $\Omega_{n,N}$ for the convolution $\ast$, we have $(N^{2n}\Omega_{n,N}^{-1}\Wg_{n,N})\star_N \zeta =N^{-n}\Omega_{n,N},$
from which we deduce that $(N^{2n}\Omega_{n,N}^{-1}\Wg_{n,N})\star \zeta=\delta_{\Id_n}+o(1)$, or similarly that $$N^{2n}\Omega_{n,N}^{-1}\Wg_{n,N}=\mu+o(1).$$
More generally, if $f,f_N: \mcal S_n\to \C$ are such that $f_N=f+o(1)$, then
\begin{equation}\label{WeingartenCumulantsLibres}
N^{n}\Omega_{n,N}^{-1}((\Omega_{n,N}f_N)*\Wg_{n,N})=(f_N)\star_N (\Wg_{n,N})=f\star \mu +o(1) .\end{equation}

\begin{proof}[Proof of Theorem \ref{Th:Matrices}]  Let $\mbf 
X_N = (X_j)_{j\in J}$ a family of unitarily invariant random matrices which converges in $^*$-distribution, as $N$ 
goes to infinity, to $\mbf x = (x_j)_{j\in J}$ family of some non-commutative probability space $(\mathcal A,\Phi)$. We fix $m\geq 1$ and test graph $T_i=(V_i,E_i,j_i)\in \mbb C \mcal T\langle J \rangle, i=1\etc m,$ and show the convergence stated in \eqref{Eq:Matrices}.

By taking the real and the imaginary parts, 
we can assume that the matrices of $\mbf X_N$ are Hermitian and so we do not consider adjoint of the matrices. We 
shall denote by $T=( V, E, j)$ the labeled graph obtained from the disjoint unions of $T_1\etc T_m,$ where the label map is given by restriction: $j_{| V_i}= j_i $  for $i=1\etc m.$  
\par We consider a random unitary matrix $U$, distributed according to the Haar distribution, and independent of $\mbf X_N$.  By assumption $\mbf Z_N :=U\mbf X_N U^*\in M_N(\C)$ has the same distribution as $\mbf X_N$. We denote respectively by $\underline{e}$ and $\overline{e}$ the origin vertex and the goal vertex of $e$. Then
 	\eq
		\lefteqn{\tau_{\mbf X_N}[T_1\etc T_m] }\\
		& = & \frac 1 {N^m}  \sum_{\phi: V \to [N]}\esp\left[ \prod_{e\in E} Z_{j(e)}\big( \phi(\underline{e}),\phi(\overline{e})\big)\right] \\
	& = & \frac 1 {N^m} \sum_{\substack{\phi: V \to [N] \\ \varphi,\varphi': E \to [N]}}   \esp\left[\prod_{e\in E}U\big(\phi(\underline{e}),\varphi(e)\big)  \overline{U}\big(\phi(\overline{e}),\varphi'(e)\big)  \right]\esp\left[ \prod_{e\in E} X_{j(e)}\big(\varphi(e),\varphi'(e)\big)\right].
	\qe 
 
\noindent In the integration formula (\ref{IntegrationHaar}), the number $n$ of occurrence of each term $U(i,j)$ is the cardinality of $E$ and the sum over permutations of $\{1,\ldots,n\}$ is replaced by a sum over the set  $\mcal S_E$ of permutations of the edge set $E$. By identifying $E$ with the set of integers $\{1,\ldots, |E|\}$, we consider that $\Wg_{n,N}$ is defined on $\mcal S_E$ instead of $\mcal S_n$.  Then, one has
	\eq
		\lefteqn{	\tau_{\mbf X_N}[T_1\etc T_m] }\\
	 & = & \frac 1 {N^m} \sum_{\alpha,\beta\in \mcal S_E}\Wg_{n,N}(\alpha\beta^{-1})\sum_{\substack{\phi: V \to \{1,\ldots, N\}\\ \varphi,\varphi': E \to \{1,\ldots, N\}\\ \phi\left(\underline{\alpha(e)}\right)=\phi(\overline{e}) ,\varphi(\beta(e))=\varphi'(e)}}\esp\left[ \prod_{e\in E} X_{j(e)}\big(\varphi(e),\varphi'(e)\big)\right].
	\qe
For any permutation $\alpha\in \mcal S_E$, let $\pi(\alpha)$ be the smallest partition of $V$ such that, for all $e\in E$, $\overline{e}$ is in the same block as $\underline{\alpha(e)}$. Summing over $\phi$ in the previous expression yields  

\begin{align*}
	\lefteqn{ \tau_{\mbf X_N}[T_1\etc T_m]}\\
	&= &\sum_{\alpha,\beta\in \mcal S_E}N^{\# \pi(\alpha)-m}\Wg_{n,N}(\alpha\beta^{-1})\sum_{\substack{ \varphi,\varphi': E \to \{1,\ldots, N\}\\ \varphi(\beta(e))=\varphi'(e)}}\esp\left[ \prod_{e\in E} X_{j(e)}\big(\varphi(e),\varphi'(e)\big)\right]\\
 &=&\sum_{\alpha,\beta\in \mcal S_E}N^{\# \pi(\alpha)-m}\Wg_{n,N}(\alpha\beta^{-1})\esp\left[\prod_{\substack{(e_1\ldots e_k)\\\text{ cycle of }\beta}}\Tr(X_{j(e_1)}X_{j(e_2)}\ldots X_{j(e_k)})\right]
\end{align*}
To conclude we will need the following
\begin{Lem} i) For any permutation $\alpha\in \mcal S_E$, $\#\pi(\alpha)+\#\alpha\le \# E+m$  and the equality implies that the graph of $T^{\pi(\alpha)}$ is the disjoint union of $m$ oriented cacti, with resp. set of edges $E_1\etc E_m,$ and that $\alpha$ fixes the sets $E_1\etc E_m.$ 

\noindent ii)  The map   \begin{align*}
\pi:\{\alpha: \#\pi(\alpha)+\#\alpha=  \# E+m\}&\longrightarrow \{\pi : \text{the graph of }T^\pi \text{ is the disjoint union } \\
&\text{of } m \text{ oriented cacti with resp. edges set } E_1\etc E_m\}
\end{align*} is a bijection whose inverse $\gamma$ is given, for all $\pi\in \mcal P(V)$ such that $T^\pi$ is a disjoint union of $m$ oriented cacti with resp. edges $E_1\etc E_m$, by the permutation $\gamma(\pi)$ whose cycles are the biconnected components of $T^\pi$.\label{lemmabijection}
\end{Lem}
\begin{proof}[Proof Lemma~\ref{lemmabijection}]i) Let $\alpha\in  \mcal S_E$. Let us define a connected graph 
$G_\alpha$ whose vertices are the cycles of $\alpha$ altogether with the blocks of $\pi(\alpha)$, and whose 
edges are defined as follow. There is an edge between a cycle $c$ of $\alpha$ and a block $b$ of $\pi(\alpha)$ if 
and only if there is an edge $e$ of $T$ such that $e\in c$ and $\overline{e}\in b$. This way, the edges of 
$G_\alpha$ are in bijective correspondence with the edges of $T$. Therefore, $\#\pi(\alpha)+\#\alpha\le \# E+m$  
with equality if and only $G_\alpha$ is the disjoint unions of  two trees.

In fact, each cycle of $\alpha$ yields a cycle in $T^{\pi(\alpha)}$, and in the case where $G_\alpha$ is acyclic, 
there exist no other cycle in $T^{\pi(\alpha)}$.  What is more, since $T_1\etc T_m$ are connected, if $T^\pi$ 
has $m$ connected components, the latter cannot use edges of several sets among $E_1\etc E_m$. Hence, the biconnected 
component of $T^{\pi(\alpha)}$ are exactly the 
cycles of $\alpha$, that cannot use edges from several sets $E_1\etc E_m$, and $T^{\pi(\alpha)}$ is therefore the disjoint union of $m$ oriented cacti, with $\alpha$ fixing each set $E_i$, $i=1\etc m$.

ii) $\pi\circ\gamma$ and $\gamma\circ \pi$ are the identity functions: $\pi$ is one-to-one and its inverse is $\gamma$.
\end{proof}

For all $\alpha\in \mcal S_E$, set $$\phi_N(\alpha)=N^{-\# \alpha}\esp\left[\prod_{\substack{(e_1\ldots e_k)\\
\text{ cycle of }\sigma}}\Tr(X_{\gamma(e_1)}X_{\gamma(e_2)}\ldots X_{\gamma(e_k)})\right]$$ and $$ 
\phi(\alpha)=\prod_{\substack{(e_1\ldots e_k)\\\text{ cycle of }\sigma}}\Phi(x_{\gamma(e_1)}x_{\gamma(e_2)}\ldots 
x_{\gamma(e_k)})$$ in such a way that that $\phi_N=\phi+o(1)$. Let us fix $\alpha\in \mcal S_E$. On the one hand 
we have $$N^{\#\pi(\alpha)+\#\alpha-\# E-m}=\one_{\#\pi(\alpha)+\#\alpha=  \# E+m}+o(1).$$ On the other hand, 
according to \eqref{WeingartenCumulantsLibres}, the quantity
	\eq
 	\sum_{\beta\in \mcal S_E}N^{\# E-\#\alpha}\Wg_{n,N}(\alpha\beta^{-1})\esp\left[\prod_{\substack{(e_1\ldots e_k)\\\text{ cycle of }\beta}}\Tr(X_{\gamma(e_1)}X_{\gamma(e_2)}\ldots X_{\gamma(e_k)})\right] 	\qe
is equal to $((\phi_N)\star_N \Wg_{n,N})(\alpha) = 
(\phi\star \mu)(\alpha)+o(1).$ Let us write  $\alpha_1\times \dots \times \alpha_m$ for the permutation whose restriction to $E_1\etc E_m$  is given by $\alpha_i\in \mcal S_{E_i},$ for $i=1\etc m.$ It follows that 
	$$\tau_{\mbf X_N}(T_1\etc T_m) 
	=\sum_{\substack{\alpha_i \in \mcal S_{E_i}, i =1\etc m \\\#\pi(\alpha_i\times \dots \times \alpha_m)+\#\alpha_1\times\dots \times \alpha_m=  \# E+m}}(\phi\star \mu)(\alpha_1\times\dots \times \alpha_m)+o(1).$$
From \eqref{cumtrace}, we know that $(\phi\star \mu)(\alpha)=k(\alpha)=\prod_{\substack{(e_1\ldots e_k)\\\text{ cycle of }\alpha}}\kappa(x_{\gamma(e_1)},\ldots, x_{\gamma(e_k)})$.  Let us write now $\pi_1\sqcup\pi_2,$ the partition of $E$ that is finer than $\{E_1\etc E_m\}$ and whose restriction of these $m$ sets is fixed, when $\pi_i\in \mcal P(V_i)$, $i=1\etc m$. Thanks to Lemma \ref{lemmabijection}, we can now write
	\begin{align*}
	\tau_{\mbf X_N}(T) 
	&=\sum_ {\substack{\pi_i\in \mcal P(V_i), i=1\etc m \\ T_i^{\pi_i} \text{ is an oriented cactus}}}\ \prod_{\substack{(e_1\ldots e_k)\\\text{ cycle of }\gamma({\pi_1\sqcup \dots \sqcup \pi_m})}}\kappa(x_{\gamma(e_1)},\ldots, x_{\gamma(e_k)})+o(1)\\
&=\sum_{\substack{\pi_i\in \mcal P(V_i), i=1\etc m\\ T_i^{\pi_i } \text{ is an oriented cactus}}}\ \prod_{\substack{(e_1\ldots e_k)\\\text{ simple cycle of one graph }T_i^{\pi_i} }}\kappa(x_{\gamma(e_1)},\ldots, x_{\gamma(e_k)})+o(1).\end{align*}
In order to pursue the computation, let $t_i$ be the test graph $(V_i,E_i, \lambda_i(e))\in \mbb C \mcal T\langle \mathcal{G}(\mathcal{A}) \rangle$ such that $\lambda_i(e)=x_{j_i(e)}$, for $i=1,2.$ By Definition of unitarily invariant traffics, we get
	\begin{align*}
		\tau_{\mbf X_N}(T_1\etc T_m) &=\sum_{\pi_i\in \mcal P(V_i), i=1\etc m } \prod_{i=1}^m\tau_\Phi^0[t_i^{\pi_i}]+o(1)\\
&=\prod_{i=1}^m \tau_{X}[t_i]+o(1)\end{align*}
so that $\tau_{\mbf X_N}(T_1\etc T_m)$ converges towards the expected limit.
\end{proof}
\begin{Rk} From the above proof, it is tempting to believe that expansions of moments of the evaluation of  test graphs in powers of $N^{-1}$   should actually be  expansions in powers of $N^{-2}$, so that for any   $*$-test 
graph $T=(V,E,j\times \epsilon)\in \mbb C \mcal T\langle J\times \{1,\ast\} \rangle$,  the fluctuations of $\frac{1}{N}\Tr( T(\mbf X_N))-\esp\big[\frac{1}{N}\Tr( T(\mbf X_N))\big]$ should be of order $N^{-1}.$ This is nonetheless wrong as shows the following simple example.   Consider a random $N\times N$ matrix $A,$ whose 
law is invariant  by unitary conjugation and the test graph $T$ with one simple edge labeled by $A$ and one extremity equal both to the input and output.  For the associated traffic distribution as in Example \ref{Ex:MgAlg}, $\Tr(T(\mbf X_N))= \sum_{1\le i,j\le N} A_{i,j}.$   In the setting of the central limit theorem where entries 
of $A$ have variance of order $\frac 1 N$, the fluctuations of $\frac{1}{N}\Tr( T(\mbf X_N))$  are of order $O(\frac{1}{\sqrt{N}}).$ 
\end{Rk}

\section{Canonical construction of spaces of free type}
\label{section:Canonical_construction}

The purpose of this section is to prove the Theorem \ref{MainThfree}, which states that any tracial $^*$-probability space can be enlarged into a traffic space. 

\subsection{Free $\mcal G$-algebra generated by an algebra}

We first describe how an algebra can be canonically extended into a $\mcal G$-algebra.

\begin{Def}
Let $\mathcal A$ be an algebra. We denote by $\mcal G(\mathcal{A})$ the $\mcal G$-algebra $\mathbb{C}\mathcal{G}\langle \mathcal{A} \rangle$ of graph polynomials labeled in $\mcal A$, quotiented by the following relations: for all $g\in \mcal G_{n-k+1}$, $a_1,\ldots, a_n\in \mcal A$ and $P$ non-commutative polynomial in $n$ variables, we have
\begin{equation}Z_g(\cdot \overset{P(a_1,\ldots,a_k)}{\longleftarrow} \cdot\otimes \cdot \overset{a_{k+1}}{\leftarrow} \cdot\otimes\ldots\otimes \cdot\overset{a_n}{\leftarrow} \cdot)=Z_g\big(P(\cdot \overset{a_1}{\leftarrow} \cdot,\ldots,\cdot \overset{a_k}{\leftarrow} \cdot)\otimes \cdot \overset{a_{k+1}}{\leftarrow} \cdot\otimes\ldots\otimes \cdot\overset{a_n}{\leftarrow} \cdot\big)\label{Firstrelation}\end{equation}
which allows to consider the algebra homomorphism $V:\mathcal{A} \to \mcal G(\mathcal{A})$ given by $a\mapsto (\cdot\overset{a}{\leftarrow} \cdot)$.
\end{Def}
As for the free product of $\mcal G$-algebra of Section \ref{Sec:FreeProdAlg}, the space $\mcal G(\mcal A)$ is a $\mcal G$-algebra. Moreover, it is the free $\mcal G$-algebra generated by the algebra $\mathcal{A}$ in the following sense.
\begin{Prop}Let $\mathcal{B}$ be a $\mcal G$-algebra and $f:\mathcal{A}\to \mathcal{B}$ a algebra homomorphism. There exists a unique $\mcal G$-algebra homomorphism $f':\mcal G(\mathcal{A})\to  \mathcal{B}$ such that $f=f'\circ V$. As a consequence, the algebra homomorphism $V:\mathcal{A} \to \mcal G(\mathcal{A})$ is injective.
\end{Prop}
\begin{proof}
The existence is given by the following definition of $f'$ on $\mcal G(\mathcal{A})$:
$$f'(Z_g(\cdot \overset{a_1}{\longleftarrow} \cdot\otimes\ldots\otimes \cdot\overset{a_n}{\leftarrow} \cdot))=Z_g(f(a_1)\otimes\ldots \otimes f(a_n))$$
for all $a_1,\ldots,a_n\in \mathcal{A}$; which obviously respects the relation defining $\ast_{j\in J} \mathcal{A}_j$.

The uniqueness follows from the fact that $f'$ is uniquely determined on $ V(\mathcal{A})$ (indeed, $f'(a)$ must be equal to $f(b)$ whenever $a=V(b)$) and that $V(\mathcal{A})$ generates $\mcal G(\mathcal{A})$ as a $\mcal G$-algebra.
\end{proof}

For example, the free $\mcal G$-algebra generated by the variables $\mbf x=(x_i)_i\in J$ and $\mbf x^*=(x_i^*)_i\in J$ is the $\mcal G$-algebra $\mbb C \mcal G\lara$ of graphs whose edges are labelled by $\mbf x$ and $\mbf x^*$.

\subsection{Algebraic construction}

Let $(\mcal A, \Phi)$ be a non-commutative probability space such that $\Phi$ is a trace. We want to equip the $\mcal G$-algebra $\mcal G(\mcal A)$ with a combinatorial distribution that is of cactus type and whose induced distribution on $\mcal A \subset \mcal G(\mcal A)$ is $\Phi$. We firstly define $\tau : \mbb C \mcal T\langle \mcal A \rangle \to \mbb C$ by the cactus formula, namely for any test graph $T$ labeled in $\mcal A$, 
	$$\tau^0[T] = \one( T \mrm{ \ is \ a \ w.o. \ cactus \ }) \prod_{C \mrm{ \ cycle \ of \ } T} \kappa(C),$$
where as usual $\kappa$ is the free cumulant function with respect to $\Phi$ of the variable along the oriented cycle. Then, as in Section \ref{Sec:FreeProdAlg}, we consider the map $\tilde \tau : \mbb C \mcal T\big\langle \mbb C \mcal G\langle \mcal A \rangle \big\rangle \to \mbb C$ defined as follow: for any test graph $T$ with edges $e_1\etc e_k$ labeled respectively by graph monomial $g_1\etc g_K$, we set $\tilde \tau[T] = \tau[T_g]$ where $T_g$ is obtained by replacing the egde $e_k$ by the graph $g_k$ for any $k=1\etc K$. We extend the definition by multi-linearity with respect to the edges and set $\tilde \tau\big[ (\cdot) \big] = 1$. By Lemma \ref{Lem:AssProdEns}, $\tilde \tau$ satisfies the associativity property and then endows $\mbb C \mcal G\langle \mcal A \rangle$ with a structure of algebraic traffic space. It remains to prove that it induces a same structure on $\mcal G(\mcal A)$.

\begin{Prop}\label{interestingproposition}
The linear form $\tilde \tau$ is invariant under the relations~\eqref{Firstrelation} defining $\mcal G(A)$, and consequently yields to an algebraic traffic space $(\mcal G(A) , \tilde \tau)$. Furthermore, the trace induced by $\tilde \tau$ coincides with $\Phi$ on $\mcal A$, seen as a subalgebra of $\mcal G(\mcal A)$.
\end{Prop}

\begin{proof}It is sufficient to prove the following:
\begin{enumerate}
\item For any test graph $T$ having an edge $e$ labeled $a_1 + \alpha a_2$, where $a_1,a_2 \in \mcal A$ and $\alpha \in \mbb C$, one has $\tau[T] = \tau[T_1] + \alpha \tau[T_2]$ where $T_i$ is obtained from $T$ by putting label $a_i$ on $e$. 
\item For any test graph $T$ having an edge $e$ labeled $1_\mcal A$, one has $\tau[T] = \tau[T_\bullet]$ where $T_\bullet$ is obtained by identifying source and target of $e$ and suppressing this edge.
\item For any test graph $T$ having an edge $e$ labeled $a_1 a_2$, where $a_1,a_2 \in \mcal A$, one has $\tau[T] = \tau[T_\times]$ where $T_\times$ is obtained by replacing $e$ by two consecutive edges $(\cdot \overset{a_1}\leftarrow\cdot \overset{a_2}\leftarrow \cdot)$. 
\end{enumerate}
The first property is an immediate consequence of the linearity of the cumulants. Let us prove the others properties at the level of the injective trace.
\begin{Lem}\label{Lem:gplus}With notations as above, we have the following formulas:
\begin{enumerate}
\item Whenever $e$ has label $1_\mcal A$, one has $\tau^0[T] = \tau^0[T_\bullet]$ if the goal and the source of the edge $e$ are equal in $T$, and $\tau^0[T]=0$ otherwise.
\item Whenever $e$ has label $a_1a_2$, denote by $V$ the vertex set of $T$ and by $v_0$ the new vertex in $T_\times$. Then for any partition $\pi$ of $V$,  one has $\tau^0[T^\pi] = \sum_{\substack{\sigma\in \mathcal{P}(V\cup\{v_0\})\\\sigma\setminus \{v_0\}=\pi}} \tau^0[ T_{\times}^\sigma]$.
\end{enumerate}
\end{Lem}

This implies the proposition as we see now. When $e$ has label $1_{\mcal A}$, we get 
\eq
	\tau[T] = \sum_{\pi \in \mcal P(V)} \tau^0[T^\pi] =\sum_{\substack{ \pi \in \mcal P(V_\bullet)}} \tau^0[T_\bullet^\pi] = \tau[T_\bullet].
\qe
Moreover, when $e$ has label $a_1a_2$, one has
	\eq
		\tau[T] & = & \sum_{\pi \in \mcal P(V) } \tau^0[T^\pi] =   \sum_{\pi \in \mcal P(V) }\sum_{\substack{\sigma\in \mathcal{P}(V\cup\{v_0\})\\\sigma\setminus \{v_0\}=\pi}} \tau^0[T_\times^\sigma]\\
		& = & \sum_{\sigma \in \mcal P(V \cup\{v_0\})} \tau^0[T_\times^\sigma] = \tau[T_\times].
	\qe
Finally, for any $a\in \mcal A$, seen as an element of $\mcal A$, its trace associated to $\tau$ is given by $\tau(\circlearrowleft_{a})=\tau^0 (\circlearrowleft_{a})=\kappa(a)=\Phi(a)$ as expected. This finishes the proof of the proposition.
\end{proof}

\begin{proof}[Proof of Lemma~\ref{Lem:gplus}]The first item follows from the fact that a cumulant involving $1_{\mcal A}$ is equal to $0$, except $\kappa(1_{\mcal A})=1$ (see~\cite[Proposition 11.15]{NS}). As a consequence, for a cactus $T$ having a loop labeled $1_\mcal A$, we can remove the loop without changing the value of the invective trace.

Let us prove the second item, and consider a test graph $T$ with an edge $e$ labeled $a_1a_2$ and $T_\times$ defined as before. Let $\pi$ be a partition of the vertex set of $T$. If $T^\pi$ is not a cactus, then the two side of the equation are equal to zero. Assume that $T^\pi$ is a cactus. We denote by $c$ the cycle of $\cdot \overset{a_1a_2}{\leftarrow} \cdot $ in $T^{\pi}$ and $a_1a_2,b_2,\ldots, b_{k-1}$ the elements of the cycle $c$ starting at $a_1a_2$.

Let us consider a partition $\sigma\in \mathcal{P}(V\cup\{v_0\})$ such that $T_\times^\sigma$ is a cactus and $\pi=\sigma\setminus \{v_0\}$. Then, we have two cases:
\begin{enumerate}
\item $v_0$ is of degree $2$ (this occurs for only one partition $\sigma$ given by $\pi\cup\{\{v_0\}\}$). Denoting by $c^+$ the cycle of $T_\times^\sigma$ which contains $v_0$, we have $c^+=(a_2,b_2,\ldots, b_{k-1},a_1)$. The cycles of $T^\pi$ and $T_\times^\sigma$ different from $c$ and $c^+$ are the same, and by consequence 
	$$\tau^0[T^\pi]/k(a_1a_2,b_2,\ldots, b_{k-1})=\tau^0[T_\times^\sigma]/k(a_2,b_2,\ldots, b_{k-1},a_1).$$
\item $v_0$ is of degree $>2$. We denote by $c_1$ the cycle of $\cdot \overset{a_2}{\leftarrow} \cdot $ in $T_\times^\sigma$, $c_2$ the cycle of $\cdot \overset{a_1}{\leftarrow} \cdot $ in $T_\times^\sigma$ (of course, $c_1$ and $c_2$ are not equal, because if it is the case, $T^\pi$ would be disconnected, which is not possible). The cycles of $T^\pi$ different from $c$ are exactly the cycles of $T_\times^\sigma$ different from $c_1$ or $c_2$. We have $c_1=(a_2,b_2,\ldots , b_l)$ and $c_2=(b_{l+1},\ldots , b_k,a_1)$ with $l$ the place of  the vertex which is identified with $v_0$ in $T_\times^\sigma$. By definition, we have
$$\tau^0[T^\pi]/k(a_1a_2,b_2,\ldots, b_{k-1})=\tau^0[T_\times^\sigma]/(k(a_2,b_2,\ldots , b_l)\cdot k(b_{l+1},\ldots , b_k,a_1)).$$
Conversely, for each vertex $v_1$ in the cycle $c$, we are in the above situation for $\sigma=\pi_{|v_0\simeq v_1}$.
\end{enumerate}
Finally, using \cite[Theorem 11.12]{NS} for computing $k(a_1a_2,b_2,\ldots, b_{k-1})$, we can compute
	\begin{align*}\tau^0[T^\pi]=	&\tau^0[T^\pi]/k(a_1a_2,b_2,\ldots, b_{k-1})\cdot k(a_1a_2,b_2,\ldots, b_{k-1})\\
		=&\tau^0[T^\pi]/k(a_1a_2,b_2,\ldots, b_{k-1})\\
		&\ \ \cdot  \left(k(a_2,b_2,\ldots, b_{k-1},a_1)+\sum_{1\leq l \leq k}k(a_2,b_2,\ldots , b_l)\cdot k(b_{l+1},\ldots , b_k,a_1)\right)\\
		=&\tau^0[T_\times^{\pi\cup\{\{v_0\}\}}]+\sum_{\substack{\sigma\in \mathcal{P}(V\cup\{v_0\})\setminus \{\pi\cup\{\{v_0\}\}\}\\\sigma\setminus \{v_0\}=\pi}}\tau^0[T_\times^{\sigma}]\\
=&\sum_{\substack{\sigma\in \mathcal{P}(V\cup\{v_0\})\\\sigma\setminus \{v_0\}=\pi}}\tau^0(T_\times^\sigma).
\end{align*}
\end{proof}

\subsection{Positivity}\label{Sec:PositivityCactusDistr}
Let $(\mathcal{A},\Phi)$ be a $^*$-probability space. We define $\tau: \mbb C \mcal T \langle \mcal A \rangle \to \mbb C$ by the cactus formula with respect to $\Phi$ and then $(\mathcal G(\mathcal{A}), \tilde \tau)$ as in Proposition~\ref{interestingproposition}. It remains to prove that $\tilde \tau$ satisfies the positivity condition~\eqref{eq:NonNegCond}, and it is actually sufficient to prove that $\tau$ is positive.

In the four steps of the proof, we will prove successively that $\tau\big[t|t^*\big]\geq 0$ for $n$-graph polynomials $t = \sum_{i=1}^L \alpha_i t_i$ with an increasing generality:
\begin{enumerate}
\item the $t_i$ are $2$-graph monomials without cycles and the leaves are outputs, that is chains of edges with possibly different orientations;
\item the $t_i$ are trees whose leaves are the outputs;
\item the $t_i$ are such that $t_i|t_i^*$ have no cutting edges (see Definition \ref{Def:Cutt});
\item the $t_i$ are $n$-graph monomials.
\end{enumerate}

\paragraph{Step 1} By Proposition~\ref{interestingproposition}, the trace associated to $\tau$ coincides with $\Phi$ on $\mcal A \subset \mcal G(\mcal A)$. We still denote it by $\Phi$. Hence we get the positivity if all the $t_i$'s consist in chains of edges all oriented in the same direction. Indeed, we can write $t_i=\cdot \overset{a_i}{\leftarrow} \cdot$ for all $i$ (or $t_i=\cdot \overset{a_i}{\rightarrow} \cdot$ for all $i$) and so, we get
	$$\tau\big[t|t^*\big]=\tau\Big[\sum_{i,j=1}^L\alpha_i\bar \alpha_j t_it_j^*\Big]=\Phi\Big(\sum_{i,j=1}^L\alpha_i\bar \alpha_j a_ia_j^*\Big)\geq 0,$$
by positivity of $\Phi$ on $\mcal A$.  We deduce that $\Phi$ is positive on the subalgebras $\C\langle \cdot \overset{a}{\leftarrow} \cdot:a\in \mathcal{A}\rangle$ and $\C\langle \cdot \overset{a}{\rightarrow} \cdot:a\in \mathcal{A}\rangle$ of $\mcal G(\mcal A)$. From Lemma~\ref{Lem:FreeofTrans} these subalgebras are freely independent, so $\Phi$ is also positive on the mixed algebra $\C\langle \cdot \overset{a}{\leftarrow} \cdot,\cdot \overset{a}{\rightarrow} \cdot:a\in \mathcal{A}\rangle$ (the free product of positive trace is positive \cite[Lecture 6]{NS}). Finally, if the $t_i$'s consist in chains of edges labeled by elements of $\mathcal{A}$, we know that
$$\tau\big[t|t^*\big]=\Phi\big[\sum_{i,j=1}^L\alpha_i\bar \alpha_j t_it_j^*\big]\geq 0.$$
\paragraph{Step 2} Assume that the $t_i$'s are trees whose leaves are the outputs. Let us prove by induction on the number $D$ of all edges of the $t_i$'s that we have $\tau\big[t|t^*\big]\geq 0$.

If the number of edges of the $t_i$'s is $0$, we have $\tau_\Phi\big[t|t^*\big]=\sum_{i,j}\alpha_i\alpha_j^*\geq 0$.  We suppose that $D\geq 1$ and that this result is true whenever the number of edges of the $t_i$'s is less than $D-1$.

We can remove one edge in the following way. Let us choose one leaf $v$ of one of the $t_i's$ which has at least one edge. It is an output and for each tree $t_i$ we denote by $v^{(i)}$ the first node (or distinct leaf if there is no node) of the tree of $t_i$ encountered by starting from this output $v$, and by $t^{(i)}$ the branch of $t_i$ between this output $v$ and $v^{(i)}$. Of course, $v^{(i)}$ can be equal to $v$ and $t^{(i)}$ can be trivial, but there is at least one of the $t^{(i)}$'s which is not trivial. Denote by $\breve t_i$ the $n$-graph obtained from $t_i$ after discarding the $t^{(i)}$'s, and whose output $v$ is replaced by $v^{(i)}$. We claim that
\eq
		\tau\big[t|t^*\big] = \tau\big[t|t^*\big] \times \tau\big[ \breve t_i| \breve t_j^* \big].
	\qe
	Firstly, we can identify the pairs $v^{(i)}$ and $v^{(j)}$ in the computation of the left hand-side. Indeed, we write $\tau\big[  t_i|  t_j^* \big] = \sum_\pi\tau^0\big[ ( t_i| t_j^*)^\pi \big]$, and consider a term in the sum for which $\pi$ does not identify $v^{(i)}$ and $v^{(j)}$. Because $ \breve t_i| \breve t_j^*$ is t.e.c., there exists two disjoints paths between $v^{(i)}$ and $v^{(j)}$. But because $t^{(i)}| t^{(j)*}$ contains a third distinct path, by Lemma \ref{Lem:3co} $\pi$ cannot be a cactus if it does not identify $v^{(i)}$ and $v^{(j)}$  and so $ \tau^0\big[ ( t_i|  t_j^*)^\pi\big]$ is zero.
	
Consider a term in the sum $\sum_\pi\tau^0\big[  (t_i| t_j^*)^\pi \big]$ for which $\pi$ identifies the pairs $v^{(i)}, v^{(j)}$. Assume that a vertex $v_1$ of $\breve t_i| \breve t_j^*$ is identified with a vertex $v_2$ which is not in $\breve t_i| \breve t_j^*$. Assume that $\pi$ does not identify $v^{(i)}$ with $v_1$ and $v_2$. Because $\breve t_i| \breve t_j^*$ is t.e.c. there exists two distinct paths between $v_1$ and $v^{(i)}$ out of $\tau\big[ t^{(i)}| t^{(j)*}\big]$. But there exists also a path between $v_2$ and $v^{(i)}$ in $t^{(i)}| t^{(j)*}$. By Lemma \ref{Lem:3co}, we get that $ (t_i|  t_j^*)^\pi$ is not a cactus and so $ \tau^0\big[ ( t_i|  t_j^*)^\pi\big]$ is zero.
\par Hence, to determine which vertices of $\breve t_i| \breve t_j^*$ are identified with some vertices of $ t^{(i)}| t^{(j)*}$, one can first determine which vertices of $\breve t_i| \breve t_j^*$ are identifies with $v^{(i)}= v^{(j)}$ and which vertices of $t^{(i)}| t^{(j)*}$ are identified with this vertex. Hence the sum over $\pi$ partition of the set of vertices of $ t_i|  t_j^*$ can be reduced to a sum over $\pi_1$ partition of the set of vertices of $\breve t_i| \breve t_j^*$ and a sum over $\pi_2$ partition of the set of vertices of the graph $ t^{(i)}| t^{(j)*}$. Moreover, by definition of $\tau$, for two test graphs $T_1$ and $T_2$, if $T$ is obtained by considering the disjoint union of $T_1$ and $T_2$ and merging one of their vertices, one has $\tau^0[T] = \tau^0[T_1] \times \tau^0[T_2]$. Hence, the contribution of $\breve t_i| \breve t_j^*$ factorizes in $ \tau\big[ \breve t_i| \breve t_j^* \big]$ and the contribution of $ t^{(i)}| t^{(j)*}$ factorizes in  $\tau\big[ t^{(i)}| t^{(j)*}\big]$, and we get the expected result.

From Step 1, we know that $A=\left(\tau\big[ t^{(i)}| t^{(j)*}\big]\right)_{i,j}$ is nonnegative. By induction hypothesis, we know that $B=\left(\tau\big[ \breve t_i| \breve t_j^* \big]\right)_{i,j}$ is also nonnegative. We obtain as desired that the Hadamard product of $A$ and $B$ is nonnegative (\cite[Lemma 6.11]{NS}) and in particular, for all $\alpha_i$, we have
$$\sum_{i,j}\alpha_i\bar{\alpha}_j  \tau\big[ t_i|  t_j^*\big]\geq 0.$$
\paragraph{Step 3}Let us prove that, for all $t_i$ such that $t_i|t_i^*$ have no cutting edges, we have $\tau\big[t|t^*\big]\geq 0$.

For a graph $T$, recall that the t.e.c. components are the maximal subgraphs of $T$ with no cutting edges. We call tree of t.e.c. of $T$ the graph whose vertices are the t.e.c. components of $T$ and whose edges are the cutting edges of $T$.
First of all, our condition is equivalent to the condition that, for each $t_i$, any leaf of the tree of the t.e.c. components of $t_i$ is a component containing an output. Here again, we can proceed by induction. Let $D$ be the total number of t.e.c. components of the $t_i$'s which do not consists in a single vertex.

If $D=0$, we are in the case of the previous step. Let us assume that $D>0$ and that the result is true up to the case $D-1$. We can remove one t.e.c. in the following way. Let us choose a t.e.c. component $t^{(k)}$ which is not a single vertex of a certain $n$-graph monomial $t_k$, for some $k$ in $\{1,\ldots,L\}$. We consider $t^{(k)}$ as a multi $^*$-graph monomial, where the outputs are the vertices which are attached to cutting edges. 
Let $\breve t_k$ be the $n$-graph monomial obtained from $t_k$ by replacing the component $t^{(k)}$ by one single vertex. We define also for $i\neq k$ the $^*$-graph monomial $t^{(i)}$ to be the trivial leaf and set $\breve t_i=t_i$. We claim that
	\eq
		\tau\big[T(t_i,t_j^*)\big] = \tau\big[T(\breve t_i, \breve t_j^*)\big] \times  \tau\big[ t^{(i)}\big] \times  \tau\big[ t^{(j)*}\big]
	\qe
	(of course, this equality is nontrivial only if we consider $i=k$ or $j=k$).

Firstly, the outputs of $t^{(i)}$ can be identified. Indeed, consider  $v_1,v_2$ two distinct outputs of $t^{(i)}$. Writing $\tau\big[t_i|t_j^*\big] = \sum_\pi \tau^0\big[ (t_i|t_j^*)^\pi\big]$, consider a term in the sum for which $\pi$ does not identify $v_1$ and $v_2$. Since $t^{(i)}$ is t.e.c. there exist two distinct simple paths $\gamma_1$ and $\gamma_2$  between $v_1$ and $v_2$. Consider a path from $v_2$ to $v_1$ that does not visit $t^{(i)}$ in $t_i|t_j^*$. Such a path exists as $v_1$ and $v_2$ belong to two subtrees of $t_i$ that are attached to outputs of $t_i$, themselves being attached to the connected graph $t_j^*$. The quotient by $\pi$ yields three distinct paths $\gamma$ between $v_1$ and $v_2$ in $(t_i|t_j^*)^\pi$ which implies that $(t_i|t_j^*)^\pi$ is not a cactus by Lemma \ref{Lem:3co}. Hence, by definition of $\tau$, $\tau^0\big[(t_i|t_j^*)^\pi\big]$ is zero. Thus, when we write $\tau\big[t_i|t_j^*\big] = \sum_\pi \tau^0\big[ (t_i|t_j^*)^\pi\big]$ we can restrict the sum over the partition $\pi$ that identify $v_1$ and $v_2$, therefore, we can replace $t_i$ by the graph $\tilde t_i$  where we have identify $v_1$ and $v_2$. Hence we have 
$ \tau\big[ t_i |t_j^*\big] = \tau\big[ \tilde t_i| \tilde t_j^*\big]$.

Let us write $\tau\big[ \tilde t_i| \tilde t_j^*\big]= \sum_\pi \tau^0\big[ ( \tilde t_i |   \tilde t_j^*)^\pi\big]$. Let $\pi$ be as in the sum. Assume that a vertex $v_1$ of $t^{(i)}$ is identified by $\pi$ with a vertex $v_2$ which is not in $t^{(i)}$. Assume that $\pi$ does not identify $w^{(i)}$ with $v_1$ and $v_2$. Since $t^{(i)}$ is t.e.c. there exist two distinct paths between $v_1$ and $w^{(i)}$ in $t^{(i)}$. But $\breve t_i$ is connected and there exists a third path between $v_2$ and $w^{(i)}$. As usual this implies that $(\tilde t_i|\tilde t_j^*)^\pi$ is not a cactus and so $ \tau_\Phi^0\big[ (\tilde t_i| \tilde t_j^*)^\pi\big]$ is zero.
\par Hence, to determine which vertices of $t^{(i)}$ are identified with some vertices out of $t^{(i)}$, one can first determine which vertices of $t^{(i)}$ are identified with $w^{(i)}$ and which vertices out of $t^{(i)}$ are identified with this vertex. Thus the sum over $\pi$ partition of the set of vertices of $\tilde t_i| \tilde t_j^*$ can be reduced to a sum over $\pi_1$ partition of the set of vertices of $t^{(i)}$ and a sum over $\pi_2$ partition of the set of vertices of the graph with $t^{(i)}$ removed. Moreover, by definition of $\tau$, for two $^*$ test graphs $T_1$ and $T_2$, if $T$ is obtained by considering the disjoint union of $T_1$ and $T_2$ and merging one of their vertices, one has $\tau^0[T] = \tau^0[T_1] \times \tau^0[T_2]$. Hence, the contribution of $T( t_i,  t_j^*)$ factorizes in $ \tau\big[ T(\breve t_i,  t_j^*) \big]$ and the contribution of $ t^{(i)}$ factorizes in  $\tau\big[ t^{(i)}\big]$. We can do the same factorization for the $n$-graph monomial $t_j^*$, and we get the expected result.

Now, setting $
		\beta_i = \alpha_i \tau\big[ t^{(i)}\big]$,
	we have 
	\eq
		\tau\big[T(t,t^*)\big] =\sum_{i,j}\beta_i\bar{\beta}_j \tau\big[T(\breve t_i, \breve t_j^*)\big]
		\qe
		which is nonnegative thanks to the induction hypothesis.
\paragraph{Step 4} A direct proof of the positivity in general case requires appropriate tools, and we bypass this difficulty using both the positivity of the free product (Theorem \ref{Th:PosFreeProd}) and the fact that unitary invariant traffics are of cactus type (Proposition \ref{EqCactusUI}).

We define an auxiliary distribution of traffic ${\tau}':\mbb C \mcal T\langle \mcal A \rangle \to \mbb C$ which is defined to be equal to $\tau$ on the test graphs without cutting edges and equal to $0$ on those with cutting edges. This map $\tau'$  induces   a combinatorial distribution on  the $\mcal G$-algebra $\mathbb{C}\mathcal{G}\langle \mathcal{A} \rangle$ of graph polynomials labeled in $\mcal A$.

On the one hand, the map $\tau'$  does satisfy the positivity property since for any $n$-graph polynomial $t=\sum_{i} \alpha_i t_i$, we have
	\eq
		{\tau'}\big[  t|t^*\big]& = & \sum_{i,j}\alpha_i \bar{\alpha}_j {\tau'}\big[ t_i | t_j^* \big]	 =  \hspace{-1cm}\sum_{\substack{ { i,j}\\  {  t_i |t_j^*}\\  {\text{ without cutting edges}}}} \hspace{-1cm}\alpha_i \bar{\alpha}_j{\tau}\big[ t_i|t_j^*\big]=\hspace{-0.5cm}\sum_{\substack{i,j\\t_i|t_i^*, \, t_j|t_j^*\\\text{ without cutting edges}}}\hspace{-1cm}\alpha_i \bar{\alpha}_j\tau_\Phi\big[t_i|t_j^*\big]\geq 0,
\qe
using the result of the previous step. By positivity of the free product, the distribution remains positive if we enlarge $\mathbb{C}\mathcal{G}\langle \mathcal{A} \rangle$ into a traffic space $\mcal B$ with a unitary traffic $u$ such that $(u,u^*)$ is the limit of a Haar unitary matrix, traffic independent from the elements of $\mathbb{C}\mathcal{G}\langle \mathcal{A} \rangle$. 
We consider the function $f: \mathbb{C}\mathcal{T}\langle \mathcal{A} \rangle\to \mathbb{C}\mathcal{T}\langle \mathcal{B} \rangle$ which replaces each edge $e_a:=\overset{a}{\leftarrow}$ of a graph in $\mathbb{C}\mathcal{T}\langle \mathcal{A} \rangle$ by the edges $\overset{u}{\leftarrow} \cdot \overset{e_a}{\leftarrow} \cdot \overset{u^*}{\leftarrow}$, obtaining a graph whose edges are labelled by elements of $\mathcal{G}\langle \mathcal{A} \rangle \cup \{u,u^*\}\subset \mathcal B$. By unitary invariance, the traffic distribution $\tau'\circ f$ is exactly $\tau$ because they coincide on cycles. Hence $\tau$ is the restriction of a positive combinatorial distribution, so it is positive.

\section{Three types of traffics}\label{Sec:three_types_traffics}
From Proposition~\ref{EqCactusUI}, we recall the following for traffics of free type. Let $(\mcal A,\tau)$ be a traffic space. A family $\mbf a=(a_j)_{j\in J}$  of elements of $\mcal B$ is of \textbf{free type} if one of the following equivalent properties holds :
\begin{enumerate}
\item \textbf{Cactus type.} The injective distribution is supported on well oriented cacti that are multiplicative w.r.t. their cycles.
\item \textbf{Unitary invariance.} The family $\mbf a$ has the same traffic distribution as $u\mbf au^*=(ua_ju^*)_{j\in J}$ where $u$ is traffic independent from $\mbf a$ and is a Haar unitary on $\mcal A$ (i.e. $u$ is unitary and $\Phi(u^k {u^*}^\ell) = \delta_{k,\ell}$ for any $k,\ell\geq 0$).
\end{enumerate}
Thus we have two different characterizations of traffic of free type. A distributional symmetry and a property of the injective distribution. In this section, we will state the  corresponding caracterization for the two other types of traffics (see Table~\ref{Independences}).
\begin{table}[h]
\centering
\caption{The three types of traffics}
\label{Independences}
\begin{tabular}{|l|l|l|}
\hline
{\it Type}         & {\it Distributional symmetry}                                   & {\it Injective distribution}                \\
             &                                                 &                              \\ \hline
{\it Tensor}  & {\bf Diagonality:}     $a = \Delta(a)$                                  & Supported on flowers         \\
             & where $\Delta = Z_{\Circlearrowleft}$ is the diagonal projection                             &                              \\ \hline
{\it Boolean} & {\bf $\mbb J$-Invariance:} $a \equiv \mbb J \otimes a$ in distribution, & Supported on trees           \\
             & for $\mbb J$ the limit of the matrix whose entries are $\frac 1 N$ &                              \\ \hline
{\it Free}    & {\bf Unitary Invariance:} $a\equiv u a u^*$        & Supported on cacti \\
             &  in distribution,  for $u$ traffic independent   & and multiplicative \\
          &   and limit of Haar unitary  matrix    & on cycles \\    \hline
\end{tabular}
\end{table}

\label{Sec:freetype}

\subsection{Boolean type}\label{Sec:Booleantype}

 Let $(\mcal A, \tau)$ be an algebraic traffic space and let $\mcal Y$ a family of elements of $\mcal A$. Let us remark that $\mcal Y$ is of Boolean type whenever one of the following equivalent conditions is satisfied:
\begin{enumerate}
	\item For any $T\in \mcal T\langle \mcal Y\rangle$, one has $\tau[T] = 0$ if $T$ is not a tree, or
	\item for any $T\in \mcal T\langle \mcal Y\rangle$, one has $\tau^0[T] = 0$ if $T$ is not a tree.
\end{enumerate}

In that case, the plain and injective combinatorial distributions coincides, namely $\tau[T] = \tau^0[T]$ for any $T\in \mcal T\langle \mcal Y\rangle$. With respect to the trace $\Phi$ associated to $\tau$, $\mcal Y$ has the null distribution since $\Phi(y) = \tau\big[ \Circlearrowleft(y)\big]=0$ for any $y$ in the algebra spanned by $\mcal Y$.

\begin{Lem} If $\mcal Y$ is of Boolean type, then the non unital algebra generated by $\mcal Y$ is of Boolean type.
\end{Lem}

\begin{proof} Let $T$ be a test graph whose edges are labeled by monomials $m_i = y_{i,1} \dots y_{i,n_i}$ with $y_{i,j}$ in $\mcal Y$. Then $\tau[T] = \tau[\tilde T]$ where $\tilde T$ is obtained by replacing each edge of $T$ by the sequence of edges $(\cdot  \overset{ y_1} \leftarrow  \dots   \overset{ y_n} \leftarrow \cdot)$. The graph $T$ is a tree if and only if $\tilde T$ is a tree, hence the result.
\end{proof}

We now associate a \emph{distributional symmetry} for Boolean type variables. The matrix $\mbb J_N$ whose all entries are $\frac 1 N$ converges in traffic distribution to a traffic $\mbb J$ of Boolean type, whose distribution is given by $\tau[T] = \tau^0[T] = \one(T \mrm{ \ is \ a \ tree})$ for any $T \in \mcal T\langle \mbb J\rangle$. 

\begin{Prop} Let $(\mcal A, \tau)$ be an algebraic traffic space and let $\mcal Y$ a family of elements of $\mcal A$. A family of traffics $\mbf A$ is of Boolean type whenever one of the following equivalent conditions is satisfied:
\begin{enumerate}
	\item \textbf{Trees.} for any $T\in \mcal T\langle \mcal Y\rangle$, one has $\tau^0[T] = 0$ if $T$ is not a tree.
	\item \textbf{$\mbb J$-invariance}The family $\mbf A$ as the same distribution as $\mbb J \otimes \mbf A$ in the tensor product of traffic spaces. 
\end{enumerate}
\end{Prop}

\begin{proof} We have for any $T \in \mcal T\langle \mbb J \otimes \mcal A\rangle$, 
	$$\tau[T] = \tau[T_{\mbb J}]  \times \tau[T_{\mcal A}]  =  \one(T \mrm{ \ is \ a \ tree}) \tau[T_{\mcal A}] .$$
	Hence the $\mbb J$-invariance is equivalent to the fact that the traffic distribution of $\mcal A$  is supported on tree, or equivalently the fact that the injective combinatorial distribution of $\mcal A$  is supported on tree.
\end{proof}

\begin{Ex} Let $\mbf A_N$ be a family of random matrices that converges in traffic distribution (such families can be built from Theorem~\ref{Th:Matrices}). Then for any $M=M_N$, sequence of integers that converges to infinity, the family $\mbb J_M \otimes \mbf A_N $ converges to a family of traffics of Boolean type. Moreover the distribution of $\mbb J_M \otimes \mbf A_N$ with respect to $\Psi_N$ is the same as for $\mbf A_N$.

Together with the asymptotic traffic independence theorem, this gives a new procedure to produce asymptotically Boolean independent matrices. More precisely, if  $\mbf A_N$  and $\mbf B_N$ are independent families of random matrices that converge in traffic distribution, and $S$ is a uniform matrix of permutation of size $(M_N\cdot N)\times (M_N\cdot N)$, then   $S(\mbb J_M \otimes \mbf A_N)S^*$ and $\mbb J_M \otimes \mbf B_N $ are independent and asymptotically traffic independent, thanks to \cite[Theorem 1.8]{Male2011}. Because the limiting traffics are of Boolean type, $S(\mbb J_M \otimes \mbf A_N)S^*$ and $\mbb J_M \otimes \mbf B_N $ are asymptotically Boolean independent with respect to the anti-trace $\Psi_N=\frac{1}{N}\sum_{i,j}\left\langle E_{ij},\cdot \right\rangle$

 Note that the size of the matrices is $(M_N\cdot N)\times (M_N\cdot N)$. In contrast, in \cite[Section 3.1]{Fr03}, the author describe a procedure that leads to Boolean independence using tensor product, which produces matrices of size $N^n$, where $n$ is the number of Boolean independent variables.
\end{Ex}

\subsection{Tensor type}\label{Sec:tensortype}

A test-graph is a \emph{flower} if it has only one vertex. Let $(\mcal A, \tau)$ be an algebraic traffic space and let $\mcal Y$ a family of elements of $\mcal A$; \cite[Proposition 5.8]{Male2011} says that if $\mcal Y$ is of tensor type, for any $T\in \mcal T\langle \mcal Y\rangle$, one has $\tau^0[T] = 0$ if $T$ is not a flower.

In fact, the converse is also true and we have the following.
\begin{Prop} Let $(\mcal A, \tau)$ be an algebraic traffic space and let $\mcal Y$ a family of elements of $\mcal A$. $\mcal Y$ is of \emph{tensor type} whenever one of the following equivalent conditions is satisfied:
\begin{enumerate}
	\item \textbf{Diagonality.} For any $a\in \mcal Y$, one has $a=\Delta(a)$.
	\item \textbf{Flowers.} for any $T\in \mcal T\langle \mcal Y\rangle$, one has $\tau^0[T] = 0$ if $T$ is not a flower.
\end{enumerate}
\end{Prop}

\begin{proof}It remains to prove that if the injective distribution of  $ \mcal Y$ is supported on flowers, we have  $a=\Delta(a)$ for all  $a\in \mcal Y$. It suffices to compute
$\Phi((a-\Delta(a))(a-\Delta(a))^*)=0 $ and we deduce that $a=\Delta(a)$.
\end{proof}

\begin{Lem} If $\mcal Y$ is of tensor type, then the traffic space generated by $\mcal Y$ is of tensor type.
\end{Lem}

\begin{proof} For all $K$-graph operation $g$, we have 
\begin{align*}Z_g (a_1\otimes  \cdots \otimes   a_K)&=Z_g (\Delta(a_1)\otimes  \cdots \otimes   \Delta(a_K))\\
&=Z_{g\circ (\Delta,\ldots,\Delta)} (a_1\otimes  \cdots \otimes   a_K)\\
&=Z_{\Delta\circ g\circ (\Delta,\ldots,\Delta)} (a_1\otimes  \cdots \otimes   a_K)\\
&=\Delta\left(Z_{g\circ (\Delta,\ldots,\Delta)} (a_1\otimes  \cdots \otimes   a_K)\right)\\
&=\Delta\left(Z_{g} (\Delta(a_1)\otimes  \cdots \otimes   \Delta(a_K))\right)\\
&=\Delta(Z_g (a_1\otimes  \cdots \otimes   a_K)).
\end{align*}
\end{proof}

\subsection{Canonical traffic spaces}\label{Sec:canonical}
Proposition~\ref{interestingproposition} and Section \ref{Sec:PositivityCactusDistr} allow also to conclude the following.
\begin{Prop} Let $(\mcal A,\Phi)$ be a tracial $^*$-probability space. There exists a traffic space $\mcal {B}$ such that  $\mcal A\subset  \mcal {B}$ as $*$-algebras, the trace induced by $\mcal B$ on $\mcal A$ is $\Phi$, and the family of traffics $\mcal A$ is of \emph{free type}.\end{Prop}

We now deduce from this canonical construction of traffic spaces of free type an analogue construction for traffics of Boolean type.

\begin{Prop} Let $(\mcal A, \Psi)$ be a non-unital $^*$-probability space. Then, there exists a traffic space $(\mcal B, \tau)$, and an injective morphism of non-commutative probability space $\psi: \mcal A \to \mcal B$ such that $\psi(\mcal A)$ is a family of traffics of Boolean type. 
\end{Prop}

\begin{proof}One the one hand, let $\psi_1 : (\mcal A, \Phi) \to (\mcal B_1, \tau_1)$ be the universal construction of Part II, namely whose image consists in unitarily invariant traffics. One the other hand, let $(\mcal B_2, \tau_2)$ be a traffic space generated by the limit $\mbb J$ of the matrix $\mbb J_N$. Then $(\mcal B,\tau):= (\mcal B_1\otimes \mcal B_2, \tau_1\otimes \tau_2)$ and $\psi : a \mapsto \psi_1(a) \otimes \mbb J$ satisfy the expected properties.
\end{proof}

Finally, we have the same result for traffics of tensor types.

\begin{Prop} Let $(\mcal A,\Phi)$ be a commutative $^*$-probability space. There exists a traffic space $\mcal {B}$ such that  $\mcal A\subset  \mcal {B}$ as $*$-algebras, the trace induced by $\mcal B$ on $\mcal A$ is $\Phi$, and the family of traffics $\mcal A$ is of \emph{tensor type}.\end{Prop}
\begin{proof}It is the first example of \cite[Example 4.10.]{Male2011}. One has just to remind that, for a test-graph $T$ whose edges are labelled by  $\gamma:E\to \mcal A$, we have
$$\tau(T)= \Phi(\prod_{e\in E}\gamma(e)),$$
which allows to prove the positivity of the traffic space easily from the positivity of $\Phi$.
\end{proof}

\subsection{Relations between the traffics of different types, conclusion}

We now investigating the independence relations between traffics of tensor, Boolean and free types. 

\begin{Prop} Let $ \mcal Y$ be a family of traffics of Boolean type, traffic independent from a unital subalgebra $\mcal Z$ of traffics of free or tensor type. Then, with respect to the anti-trace, $\mcal Z$ is \emph{monotone independent} from $\mcal Y$.

More generally, the result holds whenever the unital subalgebra $\mcal Z$ is such that $\Psi(z) = \Phi(z)$ for any $z\in \mcal Z$.
\end{Prop}

\begin{proof} For any $n\geq 2$, any $z_i$ in $\mcal Z$, $i=0\etc n$ and any $y_i$ in $\mcal Y$, $i=1\etc n$,
    \eq
    	\Psi[ z_0 y_1 z_1\dots   y_n z_n ] & = & \tau\Big[ \   \cdot  \overset{ z_0} \leftarrow  \cdot  \overset{ y_1} \leftarrow  \dots  \overset{ z_{n-1}} \leftarrow  \cdot   \overset{ y_n} \leftarrow   \cdot  \overset{ z_n} \leftarrow  \cdot \Big].\\
	\qe
Let $\pi$ be a partition of the above test graph $T$ such that the graph of colored components of $T^\pi$ is a tree and the colored components of $T^\pi$ labeled in $\mcal Y$ are tree. Then $\pi$ do not identify vertices that are not extremal vertices of an edge labeled $z_i$, $i=1\etc n$. If $\pi$ does not identify two vertices of an edge labeled $z_i$, then one can factorizes $\tau^0[ \cdot \overset{ z_{i}} \leftarrow  \cdot]$ in the expression of $\tau^0[T^\pi]$. But $\tau^0[ \cdot \overset{ z_{i}} \leftarrow  \cdot] = \Psi(z_i) - \Phi(z_i) = 0$. Hence we have $\Psi[ z_0 y_1 z_1\dots   y_n z_n ] = \tau^0[T^\pi]$ where $\pi$ is the partition identifying the source and target of each edge labeled in $\mcal Z$. We then get
	\eq
		\Psi[ z_0 y_1 z_1\dots   y_n z_n ] & = &  \prod_{i=1}^n \tau^0 \Big[ \overset{z_i}\Circlearrowleft\cdot \Big] \times  \tau^0\Big[  \overset{ y_1} \leftarrow  \cdot  \overset{ y_2} \leftarrow \dots   \overset{ y_n} \leftarrow \cdot \Big] \\
		& = &  \prod_{i=1}^n \Phi(  {z_i}) \times  \Psi[y_1 y_2\dots y_n].
	\qe
We use in the last line the fact that $\tau$ and $\tau^0$ coincide for test graphs labeled by traffics of Boolean types. Since $\Phi=\Psi$ for elements of $\mcal Z$, we get the result.

\end{proof}

\bibliographystyle{plain}
\bibliography{biblio}

\end{document}